\documentclass[letterpaper,11pt,reqno]{amsart}

\usepackage[margin=1in]{geometry}
\usepackage{hyperref}
\hypersetup{
     colorlinks   = true,
     citecolor    = black,
     linkcolor    = blue
}

\usepackage{graphicx,amsmath,amssymb,amsthm,paralist,color,tikz-cd}
\usepackage{mathrsfs}
\usepackage{mathtools}
\usepackage{cite}
\usepackage{setspace}

\usepackage[T1]{fontenc}
\fontencoding{T1}
\usepackage[utf8]{inputenc}

\usepackage[color=blue!30]{todonotes}

\usepackage{amscd}

\usepackage{bbm} 

\newtheorem{theorem}{Theorem}[section]
\newtheorem{lemma}[theorem]{Lemma}

\newtheorem{theoremAlph}{Theorem}

\newtheorem{question}[theorem]{Question}

\newtheorem{proposition}[theorem]{Proposition}

\newtheorem{corollary}[theorem]{Corollary}

\theoremstyle{definition}
\newtheorem{definition}[theorem]{Definition}
\newtheorem*{definition-nono}{Definition}

\newtheorem{remark}[theorem]{Remark}

\newtheorem*{acknowledgement}{Acknowledgements}

\newtheoremstyle{case}{}{}{}{}{}{:}{ }{}
\theoremstyle{case}
\newtheorem{case}{\textbf{Case}}

\newcommand{\N}{\mathbb{N}}
\newcommand{\Z}{\mathbb{Z}}
\newcommand{\Q}{\mathbb{Q}}
\newcommand{\R}{\mathbb{R}}
\newcommand{\C}{\mathbb{C}}

\newcommand{\mc}{\mathcal}

\newcommand{\mbf}{\mathbf}
\newcommand{\mrm}{\mathrm}
\newcommand{\msc}{\mathscr}

\renewcommand{\a}{\alpha}
\renewcommand{\b}{\beta}
\newcommand{\g}{\gamma}
\newcommand{\G}{\Gamma}
\renewcommand{\d}{\delta}
\newcommand{\e}{\varepsilon}

\renewcommand{\L}{\Lambda}
\newcommand{\w}{\omega}
\newcommand{\s}{\sigma}
\newcommand{\vp}{\varphi}
\renewcommand{\t}{\tau}
\renewcommand{\th}{\theta}

\renewcommand{\k}{\kappa}

\newcommand{\set}[1]{\left\{#1\right\}}
\renewcommand{\r}{\rightarrow}

\def\multiset#1#2{\ensuremath{\left(\kern-.3em\left(\genfrac{}{}{0pt}{}{#1}{#2}\right)\kern-.3em\right)}}
\newcommand{\norm}[1]{\left\lVert#1\right\rVert}

\newcommand{\bx}{\mathbf{x}}
\newcommand{\Lcal}{\mc{L}}
\newcommand{\Kcal}{\mc{K}}
\newcommand{\Fcal}{\mc{F}}
\newcommand{\Pcal}{\mc{P}}
\newcommand{\diam}[1]{\mrm{diam}\left(#1\right)}
\newcommand{\Fsc}{\msc{F}}

\numberwithin{equation}{section}

\title[Singular Vectors on Fractals]{Singular Vectors on Fractals and Projections of Self-similar Measures}
\author{Osama Khalil}
\address{Department of Mathematics, Ohio State University, Columbus, OH}
\email{khalil.37@osu.edu}


\date{}

\begin{document}

\begin{abstract}

    Singular vectors are those for which the quality of rational approximations provided by Dirichlet's Theorem can be improved by arbitrarily small multiplicative constants. We provide an upper bound on the Hausdorff dimension of singular vectors lying on self-similar fractals in $\R^d$ satisfying the open set condition. The bound is in terms of quantities which are closely tied to Frostman exponents of projections of the Hausdorff measure supported on the fractal.
    Our bound is optimal in the sense that it agrees with the exact dimension of singular vectors obtained by Cheung and Chevallier when the fractal is trivial (i.e. has non-empty interior).
    As a corollary, we show that if the fractal is the product of $2$ copies of Cantor's middle thirds set or the attractor of a planar homogeneous irrational IFS, then the upper bound is $2/3$ the dimension of the fractal.
    This addresses the upper bound part of a question raised by Bugeaud, Cheung and Chevallier.
    We apply our method in the setting of translation flows on flat surfaces to show that the dimension of non-uniquely ergodic directions belonging to a fractal is at most $1/2$ the dimension of the fractal.
    
\end{abstract}

\maketitle

{
\newcommand{\hK}{\hat{\Kcal}}
\newcommand{\hP}{\hat{\Pcal}}
\newcommand{\Kw}{\Kcal_\w}
\newcommand{\sing}{\mrm{Sing}}
\renewcommand{\l}{\lambda}
\newcommand{\Ccal}{\mc{C}}

    \section{Introduction}

\subsection{Statement of results}
    The goal of this article is to study the dimension of singular vectors on fractals.
    The notion of singular vectors is motivated by Dirichlet's theorem. It states that for every vector $\bx\in\R^d$ and for every $N\in \N$, the inequalities 
    \begin{equation}\label{eq: Dir ineq}
    		\norm{ \mathbf{q}\cdot\bx +p   } \leqslant \e/N, \qquad
            \norm{\mathbf{q}} \leqslant  N^{1/d},
    \end{equation}
    admit a non-zero solution $(p,\mathbf{q}) \in \Z\times \Z^{d}$ when $\e=1$.
    A vector $\bx\in \R^d$ is then said to be \textbf{singular} if for every $0<\e<1$, there exists $N_0 \in \N$ so that the inequalities in~\eqref{eq: Dir ineq} admit a non-zero integral solution for all $N \geqslant N_0$.
    We will denote by $\mrm{Sing}(d)$ the set of singular vectors in $\R^d$.
    It is well-known~\cite{DavenportSchmidt,KleinbockWeiss-Dirichlet} that the Lebesgue measure of $\sing(d)$ is $0$. 
    
    In a remarkable article, the Hausdorff dimension of $\mrm{Sing}(2)$ was determined by Cheung~\cite{Cheung-SingularPairs}. This result was extended by Cheung and Chevallier in~\cite{CheungChevallier} to higher dimensions. They showed that for $d\geq 2$,
    \begin{equation}\label{eq: cheung}
        \dim_H (\mrm{Sing}(d)) = \frac{d^2}{d+1},
    \end{equation}
    where $\dim_H$ denotes Hausdorff dimension.
    More recently, a sharp upper bound on the dimension of singular $(m\times n)$-matrices was found in~\cite{KKLM} and the exact lower bound was later determined in~\cite{DasFishmanSimmonsUrbanski}. 
    On the other hand, the study of Diophantine properties of fractals has attracted a lot of interest in recent years, begining with the work of Kleinbock, Lindenstrauss, and Weiss in~\cite{KleinbockLindenstraussWeiss}.
    This study is motivated by Sprind\v{z}uk's conjecture, resolved in~\cite{KleinbockMargulis}, and concerns finding optimal conditions on measures and subsets of $\R^d$ under which they inherit the Diophantine properties of the ambient space.
    Subsequently, Bugeaud, Cheung, and Chevallier raised the natural question of determining the Hausdorff dimension of singular vectors on fractals in $\R^d$.
    \begin{question}
    [Problem 6,~\cite{BugeaudCheung}] \label{question: BCC}
    What is the dimension of the set of vectors in $\mrm{Sing}(2)$ whose coordinates belong to Cantor's middle thirds set?
    \end{question}

    In this article, we give an upper bound on the dimension of singular vectors belonging to a large class of fractals which arise as the limit sets of iterated function systems (IFS) of contractive similarities on $\R^d$. This class includes such familiar examples as products of the same Cantor set, Koch snowflakes, Sierpi\'nski gaskets, etc.
    Our result is new even in the setting of Question~\ref{question: BCC}.
    We refer the reader to Section~\ref{sec: prelim} for detailed definitions.
    For prior work on Diophantine properties of fractals, see Section~\ref{section: related}.

    To state the result, we need some notation.
    For $1\leq k\leq d$, we denote by $\mc{A}(d,k)$ the collection of all affine subspaces of $\R^d$ of dimension $k$.
    For $\mc{L}\in \mc{A}(d,k)$ and $\e>0$, we use $\mc{L}^{(\e)}$ to denote the open $\e$-neighborhood of $\mc{L}$ in the Euclidean metric.
    Given a compactly supported Borel measure $\mu$ on $\R^d$ and $1\leq \ell\leq d$, we define $\a_\ell(\mu)$ by
    \begin{equation}\label{eqn: proj spec}
        \a_\ell(\mu) =  \liminf_{\e\r 0} \frac{\log \sup_{\mc{L} \in\mc{A}(d, d-\ell)}
        \mu \left(\mc{L}^{(\e)}\right) }{\log \e }.
    \end{equation}
    The quantities $\a_\ell(\mu)$ quantify the concentration of the support of $\mu$ near proper subspaces: the smaller $\a_\ell(\mu)$ is, the more concentrated its support is near affine subspaces of dimension $d-\ell$. On the other hand, proper rational affine subspaces of $\R^d$ are contained in $\sing(d)$.
    Our first main result shows that one can provide an upper bound on the dimension of singular vectors on fractals in terms of these quantities. It is a special case of Theorem~\ref{main dynamics thm} below.
    \begin{theoremAlph}\label{main thm}
    Suppose $\Fcal$ is an IFS consisting of contractive similarities on $\R^d$ and satisfying the open set condition. Let $\mc{K}$ be the limit set of $\Fcal$ and $\mu$ be the restriction of the $s$-dimensional Hausdorff measure to $\Kcal$, where $s = \dim_H(\Kcal)$.
    Then,
    \[ 
        \mrm{dim}_H \left( \mrm{Sing}(d) \cap \Kcal  \right) \leqslant
        s - \min_{1\leq \ell \leq d} \frac{(d-\ell+1)\a_\ell(\mu)}{d+1}.
    \]
    
    \end{theoremAlph}
    
    \begin{remark}
    It is known that in the setting of Theorem~\ref{main thm}, $\mu$ is finite and non-zero, and $\a_d(\mu)=s$ (Proposition~\ref{propn: properties of self similar measure}).
    When $\Fcal$ is irreducible, i.e. no proper affine subspace is invariant by all the maps in $\Fcal$, then $\a_\ell(\mu)>0$ for each $\ell$ (Corollary~\ref{cor: alpha >0}). 
    We also show in Proposition~\ref{prop: existence proj spec} that for self-similar measures, the $\liminf$ in the definition of $\a_\ell(\mu)$ is in fact a limit.
    \end{remark}
    
    \begin{remark}
    If $\mu$ is the Lebesgue measure on $\R^d$, then $\a_\ell(\mu) = \ell$. Moreover, one can realize the unit cube as the attractor of an IFS as in Theorem~\ref{main thm}. Since singular vectors are invariant by integer translations, Theorem~\ref{main thm} shows that $\sing(d)$ has dimension at most 
    $ d^2/(d+1)$, which agrees with the exact dimension obtained by Cheung and Chevallier.
    \end{remark}

    In the setting of Question~\ref{question: BCC}, Theorem~\ref{main thm} yields the following.
    \begin{corollary}\label{cantor prod cor}
    Suppose $\Kcal = \mc{C} \times \mc{C}$, where $\mc{C}\subset [0,1]$ is Cantor's middle thirds set. Then,
    \begin{equation*}
        \dim_H( \sing(2)\cap \Kcal) \leqslant \frac{2\dim_H(\Kcal)}{3} = \frac{4\log 2}{3\log 3}. 
    \end{equation*}
    \end{corollary}

    The quantities $\a_\ell(\mu)$ are closely tied to Frostman exponents of projections of $\mu$, cf.~\cite[Section 1.3]{Shmerkin} and Section~\ref{section: frostman intro} below.
    In a breakthrough article, these exponents were determined by Shmerkin for planar homogeneous irrational fractals in~\cite[Theorem 8.2]{Shmerkin}.
    This result, combined with Theorem~\ref{thm: alpha equal inf} below, allows us to evaluate the formula in Theorem~\ref{main thm} explicitly, yielding Corollary~\ref{homogeneous cor}.
    Recall that a planar IFS $\Fcal = \set{h_i:i\in\L}$ is \textbf{homogeneous} if each $h_i$ is of the form $\rho R_\a + b_i$ where $0<\rho<1$ and $R_\a$ is a rotation by angle $\a$, each of which is independent of $i$. We say $\Fcal$ is \textbf{irrational} if $\a \notin \Q\pi$.

    \begin{corollary} \label{homogeneous cor}
    Suppose $\Fcal$ is a homogeneous irrational IFS on $\R^2$ satisfying the open set condition and let $\Kcal$ be its limit set. Then,
    \begin{equation*}
        \dim_H\left(\sing(2)\cap \Kcal\right) \leqslant \frac{2\dim_H(\Kcal)}{3}.
    \end{equation*}
    \end{corollary}

    For more general self-similar fractals, when the co-dimension of $\Kcal$ is $<1$, it is possible to get an explicit, yet crude, estimate on $\a_\ell(\mu)$, yielding the following corollary.

    \begin{corollary}\label{small codim cor}
    Suppose $\Fcal$, $\Kcal$, and $\mu$ are as in Theorem~\ref{main thm}. Assume further that $s = \dim_H(\Kcal) > d-1$. Then, $\a_\ell(\mu)\geqslant s-d+\ell$ and, hence,
    \[ 
        \mrm{dim}_H \left( \mrm{Sing}(d) \cap \Kcal  \right) \leqslant
        s -  \frac{d(s-d+1)}{d+1}.
    \]
    \end{corollary}

    \begin{remark}
    Corollary~\ref{small codim cor} along with~\eqref{eq: cheung} show that $\dim_H(\sing(d)\cap\Kcal) < \dim_H(\sing(d))$, under the hypotheses of Theorem~\ref{main thm}, whenever $\dim_H(\Kcal)<d$.
    \end{remark}
    We emphasize that we do not expect the estimated upper bound in Corollary~\ref{small codim cor} to agree with the exact lower bound in general.
    However, we conjecture that the upper bound in Theorem~\ref{main thm} (and, in particular, Corollary~\ref{cantor prod cor}) is sharp when the fractal contains a dense set of rational vectors, cf. Question~\ref{question: minimal case}.

    It is worth noting that computing Frostman exponents of projections of self-affine (and in particular self-similar) measures is a rather delicate problem in general. For instance, when $\mu$ is the natural measure supported on the product of $2$ Cantor sets with multiplicatively independent dissection ratios, this problem constitutes the content of Furstenberg's well-known intersection conjecture, recently resolved by Shmerkin in~\cite{Shmerkin}, and independently by Wu in~\cite{MengWu-FurstenbergConj}.


    \subsection{Divergent orbits of the Teichm\"uller flow}
    Under Dani's correspondence, it is known that $\sing(d)$ corresponds to certain divergent orbits on the space of unimodular lattices in $\R^{d+1}$, see Theorem~\ref{main dynamics thm} below for details.
    In Section~\ref{sec: Teich}, we adapt our techniques to the closely related problem of divergent orbits of the Teichm\"uller geodesic flow.
    
    In what follows, we fix a stratum $\mc{H}$ of abelian differentials over a compact oriented surface (see Section~\ref{sec: teich defs} for definitions).
    Then, $\mc{H}$ admits a natural action by $\mrm{SL}(2,\R)$ and, in particular, by the following one parameter subgroups:
    \begin{equation}\label{eq: a_t r_th}
        a_t = \begin{pmatrix} e^t & 0 \\ 0 & e^{-t} \end{pmatrix}, \qquad
        	u(s) = \begin{pmatrix}
    		1 & s \\ 0 & 1
    	\end{pmatrix}, \qquad
        r_\th = \begin{pmatrix} \cos \th & \sin \th \\ -\sin \th & \cos\th \end{pmatrix}.
    \end{equation}
    The $a_t$ action induces the Teichm\"uller geodesic flow on $\mc{H}$.
    For $\w\in \mc{H}$, we say the orbit $(a_t \w)_{t\geqslant 0}$ is \emph{divergent on average}, if for every compact set $Q \subset \mc{H}$, one has
    \begin{equation} \label{defn: divergence on average}
    	\lim_{T\r\infty} \frac{1}{T} \int_0^T \chi_Q(a_t \w)\;dt =0,
    \end{equation}
    where $\chi_Q$ denotes the indicator function of $Q$.

    \begin{theoremAlph} \label{teich thm}
    Suppose $\Fcal$ is an IFS consisting of similarities on $\R$ satisfying the open set condition and let $\mc{K}$ be its limit set. 
    Then, for every $\w \in \mc{H}$, the Hausdorff dimension of the set of $s\in \Kcal$ such that the orbit $(a_t u(s)\w)_{t\geqslant 0}$  diverges on average in $\mc{H}$ is at most $\frac{\dim_H(\Kcal)}{2}$.
    \end{theoremAlph}

    Masur showed in~\cite{Masur_nonergodic} that the set of directions $\th$ around any point $\w\in\mc{H}$ for which the orbit $(a_t r_{\th}\w)_{t\geqslant 0}$ is divergent have dimension at most $1/2$. This was recently extended in~\cite{Khalil-moduli} to show that this upper bound in fact holds for divergent on average directions. Theorem~\ref{teich thm} generalizes both these results.
    More recently, the lower bound of $1/2$ on the dimension of divergent on average directions was established in~\cite{ApisaMasur}.

    The motivation for studying divergent Teichm\"uller geodesics comes from the study of the ergodic properties of billiard flows and interval exchange transformations (IETs). 
    Masur's criterion states that if the vertical straight line flow on $r_\th\w$ is non-uniquely ergodic (NUE), then $(a_t r_{\th}\w)_{t\geqslant 0}$ diverges in $\mc{H}$~\cite{Masur_nonergodic}.
    In this vein, Theorem~\ref{teich thm} has the following corollary.
    \begin{corollary} \label{cor: nue}
    Let $\Kcal$ be as in Theorem~\ref{teich thm}.
    Then, for every $\w \in \mc{H}$, the set of directions $\th\in \arctan(\Kcal)$ such that the vertical flow on $r_\th \w$ is not uniquely ergodic has dimension at most $\frac{\dim_H(\Kcal)}{2}$.
    \end{corollary}
    The size of the set of NUE directions has been extensively studied.
    A celebrated theorem of Kerckhoff, Masur, and Smilie shows that the Lebesgue measure of the set of NUE directions is $0$~\cite{KerchoffMasurSmilie}.
    This result was generalized by Veech in~\cite{Veech-MeasuresUE} to a broader class of measures which includes Lebesgue and the natural measure on a Cantor set.
    Additionally, it is known that first return maps of straight line flows give rise to IETs. This observation can be packaged in the form of a locally defined map from a stratum of abelian differentials to the space of parameters of IETs with a given permutation (cf.~\cite{Masur-IETsMFs} for details).
    Minsky and Weiss characterized the straight lines in the space of IETs that arise as the image of orbits of the horocycle flow $U=\set{u(s):s\in\R}$ on strata~\cite{MinskyWeiss} under such maps and showed that non-uniquely ergodic IETs belonging to these special lines have $0$ mass with respect to a broad class of measures


    \subsection{Relation to Frostman exponents of projections}
    \label{section: frostman intro}
    
    In the Appendix, we study the relationship between the quantities $\a_\ell(\mu)$ and the Frostman exponents of the projections of $\mu$ and prove Theorem~\ref{thm: alpha equal inf} below.
    This result allows us to apply a result of Shmerkin to deduce Corollary~\ref{homogeneous cor} from Theorem~\ref{main thm}.
    To motivate the theorem, we first give an equivalent definition of the quantities $\a_\ell(\mu)$.
    Let $\mrm{Gr}(d,\ell)$ be the Grassmanian of vector subspaces of $\R^d$ of dimension $\ell$.
    We identify subspaces $\pi\in \mrm{Gr}(d,\ell)$ with the associated canonical orthogonal projection from $\R^d$.
    For $\pi\in \mrm{Gr}(d,\ell)$, $x\in \pi$, and $\e>0$, let $B(x,\e)$ be the ball of radius $\e$ around $x$.
    If $X$ is a measure space with measure $\mu$ and $f: X \r Y$ is a measurable map, we denote by $f_\ast \mu$ the push-forward measure.
    Given a compactly supported Borel measure $\mu$ on $\R^d$ and $1\leq \ell\leq d$, $\a_\ell(\mu)$ can be alternatively defined by
    \begin{equation} \label{eq: alpha and projections}
        \a_\ell(\mu) =  \liminf_{\e\r 0} \frac{\log \sup_{\pi\in \mrm{Gr}(d,\ell)}\sup_{x\in \pi} \pi_\ast \mu (B(x,\e)) }{\log \e }.
    \end{equation}
    Given a finite measure $\nu$ on $\R^\ell$, the \textbf{Frostman exponent} of $\nu$, denoted $\dim_\infty(\nu)$, is defined by
    
    \begin{equation}\label{def: frostman}
        \dim_\infty(\nu) = \liminf_{\e \r 0} \frac{\log \sup_{x\in \R^\ell} \nu(B(x,\e)) }{\log \e}.
    \end{equation}
    By definition, we have
    $  \a_\ell(\mu) \leqslant \inf_{\pi_\ast\in \mrm{Gr}(d,\ell)} \dim_\infty(\pi \mu)$.
    For self-similar measures, we show the following.

    \begin{theorem}\label{thm: alpha equal inf}
    Suppose $\Fcal$ is an irrational homogeneous IFS on $\R^2$ satisfying the open set condition with limit set $\Kcal$.
    Let $\mu$ be the restriction of the $s$-dimensional Hausdorff measure to $\Kcal$, where $s=\dim_H(\Kcal)$.
    Then, for Lebesgue almost every $\pi_\th\in \mrm{Gr(2,1)}$,
    \begin{equation*}
        \a_1(\mu) = \inf_{\pi_\b\in \mrm{Gr(2,1)}} \dim_\infty((\pi_\b)_\ast \mu)
        = \dim_\infty ((\pi_\th)_\ast\mu).
    \end{equation*}
    \end{theorem}
    
    It is reasonable to expect the first equality in Theorem~\ref{thm: alpha equal inf} to hold for self-similar measures in greater generality. We hope to address this question in future work.
    
    The proof proceeds by realizing the quantities in question as limits of a certain sub-additive cocycle over an irrational rotation and utilizing an extension of the classical sub-additive ergodic theorem due to Furman~\cite{Furman-MultErgThm}.
    This technique has been used for similar problems in~\cite{PeresShmerkin,NazarovPeresShmerkin,GalicerSagliettieShmerkin}. 
    \begin{remark}
    \begin{enumerate}
        \item Under the hypotheses of Theorem~\ref{thm: alpha equal inf}, it is shown in~\cite[Theorem 8.2]{Shmerkin} that $\dim_\infty((\pi_\th)_\ast\mu) = \min\set{\dim_H(\Kcal), 1}$ for \emph{every} $\th\in [0,2\pi)$.
        The point of Theorem~\ref{thm: alpha equal inf} is establishing equality between $\a_1(\mu)$ and these Frostman exponents. This statement is perhaps not surprising to experts, though we could not locate a reference in the literature.
    
    \item Suppose $\Kcal = \mc{C}\times \mc{C}$ where $\mc{C}$ is Cantor's middle thirds set realized as the attractor of the natural IFS, denoted by $\Fcal$. A remarkable result of Shmerkin in~\cite[Theorem 6.2, Corollary 6.4]{Shmerkin} shows that $\dim_\infty ((\pi_\th)_\ast\mu ) = 1$ whenever $\th\notin \Q\pi$. On the other hand, the projections of $\mu$ on the coordinate axes have Frostman exponent equal to $\dim_H(\mc{C})<1$. This shows that the minimality assumption (irrationality of $R_\a$) cannot be dropped for the second equality in the conclusion of Theorem~\ref{thm: alpha equal inf}.
    \end{enumerate}
    \end{remark}
    
    Theorem~\ref{thm: alpha equal inf} suggests an affirmative answer to the following question.
    \begin{question}
    \label{question: minimal case}
    Suppose $\Fcal, \Kcal, s$, and $\mu$ are as in Theorem~\ref{main thm}.
    Assume that the group generated by the rotation parts of the maps in the IFS is dense in $\mrm{SO}(d,\R)$.
    Is it true that $\a_\ell(\mu) = \dim_\infty(\pi_\ast \mu) = \min(s,\ell)$ for every $\pi\in \mrm{Gr}(d,\ell)$?
    If we further assume that the maps in $\Fcal$ are all defined by rational parameters, is it true that $\dim_H(\sing(d)\cap\Kcal)=sd/(d+1)$?
    \end{question}
    
    
     \subsection{Prior work}\label{section: related}
     Kleinbock and Weiss showed in~\cite{KleinbockWeiss-SingOnFriendly} that irreducible self-similar measures satisfying the open set condition (OSC) give $0$ mass to $\sing(d)$. Indeed, they establish this result for the much wider class of \emph{friendly measures} introduced in~\cite{KleinbockLindenstraussWeiss}.
     When $0<\e<1$ is fixed, the set of vectors $\bx$ which admit non-trivial solutions to the inequalities in~\eqref{eq: Dir ineq} for all large $N$ are referred to as \emph{Dirichlet} $\e$-\emph{improvable} and are denoted $\mrm{DI}_\e$.
     In~\cite{SimmonsWeiss}, generalizing Benoist and Quint's fundamental measure rigidity results for random walks on homogeneous spaces, Simmons and Weiss showed that irreducible self-similar measures with the OSC give $0$ mass to $\mrm{DI}_\e$ for every $0<\e<1$.
     Special cases of this result were obtained previously by Einsiedler, Fishman, and Shapira for measures admitting invariance by expanding maps~\cite{EinsiedlerFishmanShapira}. This latter result relied on entropy methods and measure rigidity results for higher rank diagonlizable actions.
     
     The above results indicate scarcity of singular vectors (and indeed of $\mrm{DI}_\e$ vectors) in the support of self-similar measures.
     On the other hand, Kleinbock and Weiss showed that badly approximable vectors have full dimension in the support of absolutely friendly measures (which include the measures in Theorem~\ref{main thm}) ~\cite{KleinbockWeiss-BadFractals}, cf.~\cite{Fishman-winning}. An observation of Davenport and Schmidt~\cite[Theorem 2]{DavenportSchmidt-DI} shows that badly approximable vectors belong to $\cup_{0<\e<1}\mrm{DI}_\e$. It follows that the set $\cup_{0<\e<1} \mrm{DI}_\e$ has full dimension in the support of these measures.
     Finally, the reader may wish to consult~\cite{DasEtal-QuasiDecaying} for more recent developments in the study of \emph{extremality} of fractal measures.
     
     On the fractal geometric side, dimensions of projections of self-similar sets and measures have been extensively studied, see~\cite{Shmerkin-survey} for a survey.
     It is shown in~\cite[Section 6]{Shmerkin} that the Frostman exponent of the projection in an irrational direction of the restriction of the Hausdorff measure to a Sierpi\'nski carpet of dimension $>1$ or to the $1$-dimensional Sierpi\'nski gasket in the plane is equal to $1$.
     These results build on prior work of Hochman in~\cite{Hochman-InverseEntropy} and the observation that the projections of such sets are themselves self-similar.
     Less is known about projections of these sets in rational directions. However, several regularity results on the dimension of slices of the aforementioned sets with lines of rational slopes have been established in~\cite{BaranyFergusonSimon,BaranyRams} and references therein. 
     When $\Kcal$ is the limit set of an IFS $\Fcal$ satisfying a stronger separation condition than OSC and such that the rotation parts of the maps in $\Fcal$ generate a dense subgroup of $\mrm{SO}(d,\R)$, it is shown in~\cite{HochmanShmerkin} that the image of $\Kcal$ under a projection onto an $\ell$-dimensional subspace has Hausdorff dimension $\min(\ell,\dim_H(\Kcal))$.
     Moreover, under the same hypotheses, it is shown in \textit{loc.cit.} that projections of the associated self-similar measures onto subspaces of dimension $\ell$ are exact-dimensional, with dimension equal to $\min(\ell,\dim_H(\Kcal))$.
     The reader is referred to~\cite{PeresShmerkin,NazarovPeresShmerkin,HochmanShmerkin} and references therein for more results in that direction.

    
    \subsection{Overview of the proof and reduction to dynamics}

    We will deduce Theorem~\ref{main thm} from a stronger dynamical statment, Theorem~\ref{main dynamics thm} below.
    Let $G= \mrm{SL}(d+1,\R)$, $\G= \mrm{SL}(d+1,\Z)$, and $X = G/\G$.
    For $t>0$ and $\bx \in \R^d$, define the following elements of $G$.
    \begin{equation}\label{eq: a_t and u(x)}
    	a_t = \begin{pmatrix}
    	e^{dt}  & \mbf{0}\\
        \mbf{0} & e^{-t}\mrm{I}_d
    	\end{pmatrix}, \qquad
    	u(\bx) = \begin{pmatrix}
    		1 & \bx \\ \mathbf{0} & \mathrm{I}_d
    	\end{pmatrix}        
    \end{equation}
    where $\mathrm{I}_d$ denotes the $d\times d$ identity matrix.
    It was shown by Dani in~\cite{Dani-Divergent} that $\bx$ is singular if and only if the orbit $(a_t u(\bx) \G)_{t\geqslant 0}$ diverges in $X$.
    Theorem~\ref{main thm} follows from the following result.

    \begin{theorem}\label{main dynamics thm}
    Suppose $\Fcal$ is an irreducible IFS on $\R^d$ satisfying the open set condition and let $\mc{K}$ be its limit set. 
    Let $s = \dim_H(\Kcal)$ and $\mu$ be the restriction of the $s$-dimensional Hausdorff measure to $\Kcal$.
    Then, for every $x_0\in G/\G$, the Hausdorff dimension of the set of vectors $\bx\in \Kcal$ such that the forward orbit $(a_t u(\bx)x_0)_{t\geqslant 0}$ diverges on average in $G/\G$ is at most
    \[ 
        s - \min_{1\leq \ell \leq d} \frac{(d-\ell+1)\a_\ell(\mu)}{d+1}.
    \]
    \end{theorem}

    For convenience of the reader, we outline the proof of Theorem~\ref{main dynamics thm}.
    The proof has two main steps: a linear argument and a probablistic scheme.
    The linear argument is concerned with estimating the average rate of expansion of vectors in the exterior powers of the standard representation of $\mrm{SL}(d+1,\R)$ with respect to any measure $\nu$ satisfying $\a_\ell(\nu)>0$ for every $\ell$.
    Roughly, we show that $\norm{a_t u(\bx) v}^{-1}\in L^{\a_\ell(\nu)-\e}(\nu)$, for every $\e>0$, where $v$ is any non-zero vector in $\bigwedge^\ell \R^{d+1}$, equipped with the standard representation of $G$. This is the content of Section~\ref{section: linear expansion}.
    
    The proof of this step is based on a simple but crucial observation regarding \emph{transversality} of the expanding coordinates of $a_t$.
    This is carried out in Propositions~\ref{prop: transversality} and~\ref{prop: transverse implies integrable}.
    Roughly speaking, we show that if the projections of $u(\bx)v$ onto each expanding coordinate of $a_t$ are simultaneously small, this implies that $\bx$ belongs to a neighborhood of an affine subspace of low dimension.
    
    The other ingredient is to use the method of integral inequalities, first introduced in~\cite{EskinMargulisMozes}, to translate the results on linear expansion to recurrence results on the space $G/\G$.
    This is carried out in Section~\ref{section: integ ineq}.
    We construct a Margulis function (see Def.~\ref{defn: height functions}) $f:G/\G\r \R_+$ which measures the depth of orbits into the cusp.
    We show that the average value of $f(a_tu(\bx)x_0)$, for a fixed $t>0$, and with respect to any non-planar measure $d\nu(\bx)$, is a fraction of $f(x_0)$, whenever $x_0\in G/\G$ is sufficiently deep in the cusp.
    This concludes the linear argument.

    The probablistic scheme takes as input the average height contraction established in the previous step and converts it into an upper dimension estimate on singular vectors in the support of the measure, Theorem~\ref{thrm: Hdim and non-divergence}.
    This step is the technical heart of this article and is carried out in Section~\ref{section: abstract setup}.
    Before explaining the strategy, it is worth noting at this stage that we have not yet used the assumption that the measures under consideration are self-similar.
    However, this assumption is indispensible for the probablistic scheme. For example, using Frostman's Lemma, and the fact that $\dim_H(\sing(d)) >d-1$, one can find a measure $\nu$, whose support is contained $\sing(d)$, and satisfying $\a_d(\nu)>d-1$. This in particular implies that $\a_\ell(\nu)>0$ for all $\ell$ (cf. proof of Lemma~\ref{lem: estimate alpha small codim}).
    The results of the previous step (the linear argument) apply to such a measure, however clearly singular vectors have full dimension in the support of $\nu$.
    
    The main step in proving Theorem~\ref{thrm: Hdim and non-divergence} is Proposition~\ref{prop: H-measure bound}.
    The key property of self-similar measures we use is a \emph{renormalization} mechanism to convert global information (average contraction over the entire support) obtained via the linear argument at a small fixed time scale $t$ to local information (average contraction over small pieces of the support) at large time scales.
    We make use of the existence of a faithful representation $\mrm{Sim}(\R^d) \hookrightarrow  N_G(U)$, where $\mrm{Sim}(\R^d)$ is the group of similarities of $\R^d$, $U=\set{u(\bx):\bx\in\R^d}$, and $N_G(U)$ is the normalizer of $U$ in $G$.
    Under this inclusion, the scaling subgroup of $\mrm{Sim}(\R^d)$ corresponds to $\set{a_t:t\in\R}$.

    Analogous steps have been studied before in~\cite{KKLM} in the case of Lebesgue measure on $\R^d$ and in~\cite{Khalil-Curves} for measures supported on curves. The strategy in those cases is to show that the probability that an orbit segment $(a_tu(\bx)x_0)_{0\leq t\leq N}$ spends a large proportion of its time in the cusp decays exponentially with a precise rate.
    One then uses continuity of the flow to show that if $\bx$ is one such point, then a whole neighborhood of $\bx$ of radius $e^{-(d+1)N}$ has roughly the same behavior.
    This converts measure estimates into a count on covers and an estimate on the box dimension.
    
    Unfortunately, this strategy seems to fail when the contraction ratios of the maps in the IFS are not all the same and this introduces considerable difficulties in our case. This is due to the fact that the natural neighborhoods used in constructing the covers are pieces of the fractal whose diameters have distinct exponential decay rates to $0$.
    We introduce a method which does not rely on counting covers and rather works by estimating the Hausdorff dimension directly.
    The key starting inequality of our estimates roughly takes the form:
    \abovedisplayskip=10pt
     \belowdisplayskip=10pt
    \begin{equation}\label{eq: intro ineq}
        \sum_{\Kcal_\w \cap B(N)\neq\emptyset} \diam{\Kcal_\w}^{s-\g} \ll \int_{B(N-1)} \t(\bx)^{-\g} f(a_{\t(\bx)} u(\bx)x_0)\;d\mu(\bx),
    \end{equation}
    where the sum is over self-similar pieces of the limit set $\Kcal$ which meet the set $B(N)$ of vectors $\bx$ whose orbit segment of length $N$ spends a large proportion of its time in the cusp.
    Here, we use the fact that $\mu$ is a Hausdorff measure and, in particular, that $\mu(\Kcal_\w) \asymp \diam{\Kcal_\w}^s$. This allows us to interpret this sum as an integral of a \emph{cocycle} $\t(\bx)^{-\g}$ measuring the diameter of the self-similar pieces.
    Moreover, and crucially to our method, we interpret the factor $\t(\bx)^{\g}$ as being a non-constant contraction ratio of the averaging operator given by integrating against $\mu$ as in the right side of~\eqref{eq: intro ineq}.
    This is to be contrasted with the more standard integral inequalities of Margulis functions of the form $\int f\;d\nu \leq a f+b$, with a uniform contraction factor $0<a<1$. 
    We also estimate the indicator of $B(N)$ by a product of the height function $f$ and the indicator of $B(N-1)$.
    Accordingly, our linear argument is modified to take into account the additional cocycle factor inside the integral.
    An inductive procedure is then carried out to bound the $(s-\g)$-Hausdorff measure of the set $\liminf_N B(N)$.
    
    We conclude this introduction with a natural problem arising from our investigations.
    
    \begin{question}\label{Q: sing kleinian}
    What is the Hausdorff dimension of singular vectors belonging to the limit set of a Zariski-dense, convex cocompact, discrete group of M\"obius transformations of $\R^d$?
    \end{question}
    We refer the reader to~\cite{BeresnevichGhoshEtal-IntrinsicSing} where a related notion of \textit{intrinsic} singular vectors on the limit set of geometrically finite manifolds is introduced and studied. We remark that the Diophantine problems studied in~\cite{BeresnevichGhoshEtal-IntrinsicSing} correspond with the recurrence behavior of the geodesic flow on hyperbolic manifolds, while Question~\ref{Q: sing kleinian} is related to the recurrence of diagonal flows on $\mrm{SL}(d+1,\R)/\mrm{SL}(d+1,\Z)$ via Dani's correspondence in the spirit of Theorem~\ref{main dynamics thm}.
    
    \begin{acknowledgement}
    I would like to thank the anonymous referees for a careful reading and for numerous valuable comments that improved the presentation and corrected several inaccuracies in the initial version of the article.
    I would like to thank Jon Chaika, Yitwah Cheung and Barak Weiss for their interest in the project and for comments on an earlier version of the manuscript. I would also like to thank Pablo Shmerkin for several valuable suggestions and for providing the proof of Lemma~\ref{lem: alpha_1 for Cantor product}.
    \end{acknowledgement}
    

    \section{Preliminaries} \label{sec: prelim}
    We recall some properties of self-similar measures to be used in later sections.
    
\subsection{Hausdorff dimension}\label{sec: hdim}
    We recall the definition of the Hausdorff dimension.    
    Let $A$ be a subset of a metric space $X$. For any  $\kappa,s >0$, we define
        \begin{equation}\label{eq: H-premeasure}
            H_\kappa^s (A) = \inf\set{ \sum_{I\in \mc{U}} \mrm{diam}(I)^s: \mc{U} \text{ is a cover of A by balls of diameter } <\kappa  }.
        \end{equation} 
        The $s$-dimensional Hausdorff measure of $A$ is defined to be
        \[ H^s (A) = \lim_{\kappa \r 0^+} H_\kappa^s(A) = \sup_{\kappa > 0} H_\kappa^s(A). \]
        The Hausdorff dimension of $A$ is defined by
        \[ \dim_{H}(A) = \inf\set{\g\geq 0: H^\g(A) = 0}= \sup \set{\g\geq 0: H^\g(A) = \infty}. \]

\subsection{IFS Notation}
    Fix a finite set $\L$.
    An \emph{iterated function system} (IFS for short) is a finite collection $\mc{F} = \set{h_i:i\in\L}$ of contractive similarities of $\R^d$, i.e., for each $i\in \L$, $h_i$ has the form
    \[ h_i = \rho_i O_i + b_i, \]
    where $0<\rho_i <1$, $O_i \in \mrm{SO}(d,\R)$, and $b_i \in \R^d$.
    The \textbf{similarity dimension} of $\Fcal$ is defined to be the unique solution of the equation $\sum_{i\in\L} \rho_i^s=1$.
    It is shown in~\cite{Hutchinson} that there exists a unique compact set $\mc{K} \subset \R^d$ which is invariant by $\mc{F}$ in the following sense.
    \begin{equation}\label{eq: K self-similar}
    	\mc{K} = \bigcup_{i\in \L} h_i(\mc{K})
    \end{equation}
    We refer to the set $\mc{K}$ as the \textbf{limit set} of $\mc{F}$.
    Hutchinson introduced a notion of separation, the open set condition, that will allow us to treat the union in~\eqref{eq: K self-similar} as if it were disjoint, cf. Proposition~\ref{prop: null overlap} below.
    Following~\cite{Hutchinson}, we say $\Fcal$ satisfies the \textbf{open set condition} (OSC for short) if there exists a non-empty open set $U\subset \R^d$ such that the following holds:
    \begin{equation}\label{eq: OSC}
        \begin{cases}
            h_i(U) \subseteq U, & \text{for every } i\in \L,\\
            h_i(U)\cap h_j(U) =\emptyset, & \text{for every }i\neq j \in \L.
        \end{cases}
    \end{equation}
	Given $\omega=(\omega_i)\in\L^{k}$, we let
    \begin{equation}\label{eq: composition order}
      h_{\omega}=h_{\omega_1}\circ\cdots\circ h_{\omega_{k}}, \qquad
        \Kcal_\w = h_\w(\Kcal),\qquad \Pcal_n = \set{\Kcal_\w: \w\in \L^n}.
    \end{equation}
    The maps $h_\w$ take the form $\rho_\w O_\w + b_\w$, where
    \begin{equation}\label{eq: composition parameters}
     \rho_\w = \prod_{i=0}^{k} \rho_{\w_i},
     \quad O_\w = O_{\w_{1}} \cdots O_{\w_k},
     \quad b_\w = h_\w(0) \in \R^d  .
    \end{equation}
    It will be convenient for us to consider the numbers $\rho_\w$ for finite prefixes of infinite words $\w$. We do this by means of a multiplicative cocycle.
    To this end, let $\s: \L^\N \r \L^\N$ be the shift map, i.e., $\s((\w_k)_{k\in\N}) = (\w_{k+1})_{k\in\N}$.
    Consider the function $\rho: \L^\N \times \N \r \R_+ $ defined as follows for $\w = (\w_k)_{k\in\N} $
    \begin{equation}\label{defn: rho}
        \rho(\w,n) = \prod_{k=1}^n \rho_{\w_k}.
    \end{equation}
    Then, $\rho$ satisfies the cocycle relation: $\rho(\w,m+n) = \rho(\w,m)\rho(\s^m(\w),n)$.

    We wish to regard $\rho(\cdot,n)$ as a function on $\Kcal$.
    Since members of $\Pcal_n$ may not be disjoint, there is ambiguity on the overlaps. For this purpose, we introduce a modified partition $\hat{\Pcal}_n$ consisting of disjoint sets as follows.
    Using any fixed order on the elements of $\L$, we can endow $\L^n$ with a lexicographic order and define for each $\w\in \L^n$:
    \begin{align}\label{eq: disjoint images}
        \hat{\Kcal}_\w = \Kcal_\w \setminus \left( \bigcup_{\a > \w} \Kcal_\w\cap\Kcal_\a \right),
        \qquad  \hat{\Pcal}_n = \set{\hat{\Kcal}_\w:\w\in \L^n}.
    \end{align}
    Let $\bx\in\Kcal$ and let $\w\in\L^n$ be the unique word such that $\bx\in \hat{\Kcal}_\w$. Then, we define
    \begin{equation} \label{eq: cocycle on K}
        \rho(\bx,k) := \rho(\w,k).
    \end{equation}
    The following lemma shows that $\hat{\Pcal}_n$ form a refining sequence of partitions. This fact is used in the inductive procedure of constructing covers for singular vectors by elements of $\hat{\Pcal}_n$.
    \begin{lemma} \label{lem: hK refining}
    For every $m<n\in \N$ and every $\b\in \L^m$, we have $\hK_\b =\bigcup \hK_\a$, where the union is taken over words $\a\in\L^n$ such that $\b \in \L^m$ is the prefix of $\a$ of length $m$.
    \end{lemma}
   \begin{proof}
   Suppose $\a \in \L^n$ is such that $\b$ is a prefix of $\a$.
   Suppose $\w\in \L^m$ is such that $\w>\b$ and suppose $\bx\in \Kcal_\b\cap\Kcal_\w$. By~\eqref{eq: K self-similar}, we have $\Kcal_\w = \cup_{\t\in \L^{n-m}} \Kcal_{\w\t}$, where $\Kcal_{\w\t}$ denotes $h_\w(\Kcal_\t)$. Let $\t\in \L^{n-m}$ be such that $\bx \in \Kcal_{\w\t}$.
   Then, since $\w>\b$ and $\b$ is a prefix of $\a$, it follows that $\w\t > \a$ by definition of the lexicographic order. This implies that $\bx\notin \hK_\a$. This shows that $\hK_\a \subseteq \hK_\b$.
   
   To show the reverse containment, fix $\bx\in\hK_\b$. Let $\a \in \L^n$ be the maximal word in $\L^n$ in the lexicographic order having $\b$ as a prefix and satisfying 1) $\bx\in \Kcal_\a$, and 2) $\a \geqslant \a'$ for every $\a'\in \L^n$ such that $\bx\in \Kcal_{\a'}$.
   Suppose for contradiction that $\bx\notin \hK_\a$. Then, there exists some $\w\in \L^n$ such that $\w>\a$ and $\bx\in \Kcal_\a\cap\Kw$.
   Let $\w_0\in\L^m$ be the prefix of $\w$. 
   By maximality of $\a$, $\w_0\neq \b$. However, by definition of the lexicographic order, since $\w>\a$, it must be that $\w_0 > \b$. In particular, since $\bx\in \Kw\subset \Kcal_{\w_0}$, it follows that $\bx\notin \hK_\b$, contrary to our assumption and concluding the proof.
   \end{proof}

  \subsection{Self-similar measures}  
    Fix a probability vector $\l$ with full support on $\L$. 
    That is $\l = (\l_i)_{i\in\L}$, $\l_i>0$ for every $i$, and $\sum_i \l_i = 1$.
    The \emph{Markov-Feller Operator} $P_\l$ is defined by  
     \begin{equation*}
            P_\l(\nu) =  \sum_{i\in \L}\l_i (h_i)_\ast\nu,
    \end{equation*}
    for all Borel measures $\nu$ on $\R^d$.
    It is shown in~\cite{Hutchinson} that there exists a unique probability measure $\mu$ supported on $\mc{K}$ and satisfying
    \begin{equation} \label{eqn: self similar measure}
    	\mu = P_\l(\mu).
    \end{equation}
    We refer to measures satisfying~\eqref{eqn: self similar measure} as \textbf{self-similar measures} for $(\mc{F},\l)$. Simple induction applied to~\eqref{eqn: self similar measure} shows that
    \begin{equation}\label{eq: iterated P_l}
        \mu = P_\l^k(\mu)= \sum_{\w\in \L^k} \l_\w (h_\w)_\ast \mu,
    \end{equation}
     where $\l_\w = \prod_{i=1}^{k}\l_i$.
     Under the open set condition, Hutchinson showed that the self-similar measure for the natural probability vector $(\rho_i^s)_{i\in\L}$ is in fact the $s$-dimensional Hausdorff measure $H^s$.

    \begin{proposition}[Theorem 5.3(1),~\cite{Hutchinson}]
    \label{propn: properties of self similar measure}
    Suppose $\mc{F}$ is an IFS satisfying the open set condition with similarity dimension $s$ and let $\mc{K}$ denote its limit set.
    Then, $0<H^s(\Kcal)<\infty$. In particular, $dim_H(\Kcal) = s$.
    Let $\mu$ denote the normalized restriction of $H^s$ to $\Kcal$. Then, $\mu$ is the self-similar measure for the probability vector $(\rho_i^s)_{i\in\L}$.
    Moreover, there exist constants $a,b>0$, such that for every $x\in \mc{K}$ and every $r\in (0,1)$,
    \begin{equation*}
    	a r^s \leqslant \mu(B(x,r)) \leqslant b r^s.
    \end{equation*}
    \end{proposition}
    
    
    \subsection{Consequences of null overlaps}
    
    In general, the overlap between members of $\Pcal_n$ causes serious problems in the analysis.
    However, the following result, obtained in~\cite{Hutchinson}, shows that the OSC insures that these overlaps are negligible from the point of view of self-similar measures.
    \begin{proposition}[Proposition 5.1(4), Theorem 5.3(1)\cite{Hutchinson}]
    \label{prop: null overlap}
        Suppose $\mc{F}$ satisfies the open set condition and let $s$ be its similarity dimension.
        Let $\l = (\rho_i^s)_{i\in\L}$ and $\mu$ be the self-similar probability measure for $(\mc{F},\l)$. Then, for every $k\in \N$ and all $\a\neq \w\in \L^k$, $\mu(\Kcal_\w \cap \Kcal_\a) =0$.
    \end{proposition}
    
    We now state two consequences of Proposition~\ref{prop: null overlap} which we use in our proof.
    For a Borel set $A$ and a Borel measure $\mu$, we denote by $\mu\vert_A$ the restriction of $\mu$ to $A$. That is for every Borel set $B$, $\mu\vert_A(B) = \mu(B\cap A)$.
    The following lemma will be useful in estimating Hausdorff dimension.
    \begin{lemma}\label{lem: diam K = K hat}
    Suppose $\Fcal$ satisfies the open set condition.
    Then, for every $n\in \N$ and $\w\in \L^n$, $\Kcal_\w$ and $\hat{\Kcal}_\w$ have the same diameter.
    \end{lemma}
    
    \begin{proof}
    It suffices to show that $\hat{\Kcal}_\w$ is dense in $\Kcal_\w$.
    For $\a\in\L^\N$ and $k\in\N$, denote by $\a|_k$ its length $k$ prefix.
    Let $\pi: \L^\N \r \Kcal$ denote the coding map defined by
    \begin{equation*}
        \pi(\a) = \lim_{k\r\infty} h_{\a|_k}(0).
    \end{equation*}
    The space $\L^\N$ is endowed with the product topology induced from the discrete topology on $\L$.
    In this topology, the map $\pi$ is onto and continuous~\cite[Theorem 3.1.3(vii)]{Hutchinson}. Moreover, if $\l=(\rho_i^s)$ and $\mu$ is the unique self-similar measure satisfying~\eqref{eqn: self similar measure}, then
    \[ \pi_\ast(\l^\N) = \mu. \]
    Let $A = \pi^{-1}(\Kcal_\w)$ and $\hat{A} = \pi^{-1}(\hK_\w)$. It suffices to show that $\hat{A}$ is dense in $A$ by continuity of $\pi$. 
    Note that $A$ is the cylinder set consisting of all sequences whose prefix of length $n$ is $\w$. Moreover, by Proposition~\ref{prop: null overlap}, $\l^\N(\hat{A}) = \l^\N(A)>0$. Since $\l^\N\vert_A$ has full support in $A$, $\hat{A}$ is dense in $A$.
    \end{proof}

    We record another useful consequence of the null overlaps.
    \begin{lemma}\label{lem: transformation of self-similar measures}
    Suppose $\mc{F}$ satisfies the open set condition, $\l=(\rho_i^s)$ and $\mu = P_\l(\mu)$. Then, for every $k\in \N$ and all $\w\in \L^k$,
    \begin{equation*}
         \mu \vert_{\Kcal_\w}= \l_\w (h_\w)_\ast\mu.
    \end{equation*}
    \end{lemma}
    \begin{proof}
    This follows from Proposition~\ref{prop: null overlap} and equation~\eqref{eq: iterated P_l}.
    \end{proof}
    
 
 \subsection{Elementary facts on projections}
    We say $\mc{F}$ is \textbf{irreducible} if no finite collection of proper affine subspaces of $\R^d$ is invariant by each $h_i\in \mc{F}$.
    When $s$ is the similarity dimension of $\mc{F}$, Proposition~\ref{propn: properties of self similar measure} implies that
    \begin{equation} \label{eq: alpha_d is the dimension}
        \a_d(\mu) = s,
    \end{equation}
    where $\mu$ is the restriction of the $s$-Hausdorff measure to $\Kcal$. Moreover, it is easy to see that
    \begin{equation}\label{eq: proj spec increasing}
        i\leq j \Longrightarrow \a_i(\mu) \leq \a_j(\mu).
    \end{equation}
    
    The following lemma provides a simple lower esimtate for $\a_\ell(\mu)$.

    \begin{lemma}
    \label{lem: estimate alpha small codim}
    Suppose $\Fcal$ is an irreducible IFS on $\R^d$ satisfying the open set condition and let $\Kcal$ be its limit set. Suppose $s=\dim_H(\Kcal) >d-1$. Then, for each $1\leq \ell \leq d$, $\a_\ell(\mu)\geqslant s-d+\ell$, where $\mu$ is the restriction of $H^s$ to $\Kcal$.
    \end{lemma}
    \begin{proof}
    This result is well-known, we provide a proof for completeness. Given an affine subspace $\Lcal$ of dimension $d-\ell$ and $\e>0$, the boundedness of $\Kcal$ implies that we can cover the set $\Kcal \cap \Lcal^{(\e)}$ with $O(\e^{-(d-\ell)})$ balls of radius $\e$. By Proposition~\ref{propn: properties of self similar measure}, each such ball has measure at most $O(\e^s)$. In particular, $\mu(\Lcal^{(\e)}) \ll \e^{s-d+\ell}$. As $\Lcal$ and $\e$ were arbitrary, this completes the proof.
    \end{proof}

    
    \section{The Contraction Hypothesis and Divergent Trajectories}
\label{section: abstract setup}
	In this section, we prove an abstract recurrence result for orbits of the diagonal flow in~\eqref{eq: a_t and u(x)} starting from fractals in actions of $\mrm{SL}(n,\R)$ on metric spaces.
    Theorem~\ref{thrm: Hdim and non-divergence} is the main result of this section establishing a bound on the dimension of divergent orbits.
    In later sections, we verify the hypotheses of this theorem in the settings of the results stated in the introduction.

	\subsection{The Contraction Hypothesis for Actions of SL(n,R)}
    We fix the following more convenient parametrization of the diagonal subgroup $a_t$ of $\mrm{SL}(d+1,\R)$, defined in~\eqref{eq: a_t and u(x)}, which we denote by $g_t$ for $t>0$,
    \begin{equation} \label{linear forms g_t}
    	g_t = \begin{pmatrix}
    	t^{-d/(d+1)}  & \mbf{0}\\
        \mbf{0} & t^{1/(d+1)}\mrm{I}_d
    	\end{pmatrix},  
    \end{equation}
    where $\mathrm{I}_d$ denotes the $d\times d$ identity matrix.
    Note that in this parametrisation, one has for every $t,s >0$,
    \begin{equation*}
        g_t \circ g_s = g_{ts}.
    \end{equation*}

    Recall the definition of self-similar measures in~\eqref{eqn: self similar measure}.
    The following is the key recurrence property for the action and the measure which underlies the results stated in the introduction.
	\begin{definition} [The Contraction Hypothesis]
      \label{defn: height functions}
      Suppose $X$ is a metric space equipped with an action of $G=\mrm{SL}(d+1,\R)$ and let $\mu$ be a self-similar probabilty measure for an IFS $\Fcal = \set{h_i=\rho_i O_i +b_i :i\in \L}$ on $\R^d$.
      Given a collection of functions $\mathscr{F}=\set{ f_k:X\r (0,\infty]: k\in \N}$ and real numbers $0\leq \g_0 <\b$, we say that $\mu$
      satisfies the $(\Fcal,\Fsc,\b,\g_0)$-\textbf{contraction hypothesis} on $X$ if the following properties hold:
      \begin{enumerate}
      \item The set $Z =  \set{f_k=\infty}$ is independent of $k$ and is $G$-invariant.
      \item For every $k\in\N$, $f_k $ is $\mrm{SO}(d+1,\R)$-invariant and uniformly log Lipschitz with respect to the $G$ action. That is for every bounded neighborhood $\mc{O}$ of identity in $G$, there exists a constant $C_\mc{O}\geq 1$ such that for every $g\in \mc{O}$, $x\in X$ and $k\in \N$,
      \begin{equation}\label{defn: log lipschitz}
          C_\mc{O}^{-1} f_k(x) \leqslant f_k(gx) \leqslant C_\mc{O} f_k(x).
      \end{equation}
        
		\item There exists a constant $c \geq 1$ such that the following holds:
        for every $k\in \N$ and $\g_0 \leq \g \leq \b$, there exists $T >0$ such that for all $y \in X$ with $f_k(y)>T$,
        \begin{equation} \label{eqn: CH}
        	\int_\Kcal \rho(\bx,k)^{-\g} f_k(g_{\rho(\bx,k)} u(\bx)y) \;d\mu(\bx) 
            		\leqslant c  f_k(y) \left(\int_\Kcal \rho(\bx,k)\;d\mu(\bx)\right)^{\b-\g},
        \end{equation}
        where $\rho(\bx,k)$ is the cocycle defined in~\eqref{eq: cocycle on K} and $\Kcal$ is the limit set of the IFS $\Fcal$.
      \end{enumerate}
   The functions $f_k$ will be referred to as \textbf{height functions}.
    \end{definition}
    
    The notion of height functions was introduced in homogeneous dynamics in~\cite{EskinMargulisMozes} and was used in~\cite{KKLM} to find a sharp upper bound on the dimension of singular systems of linear forms. The ``Contraction Hypothesis" terminology is due to~\cite{BQ-RandomWalkRecurrence}.
    
    \begin{remark}
    \begin{enumerate}

        \item 
    Inequality~\eqref{eqn: CH} should be thought of as a Margulis inequality with a \emph{non-uniform} contraction ratio. To best illustrate this analogy, consider the case where $\rho(\cdot,1) \equiv \rho$ for some fixed constant $\rho\in (0,1)$. Setting $\g=\b$ in~\eqref{eqn: CH} yields
    \begin{equation}\label{eq: CH simplified}
        \int_\Kcal f_k(g_{\rho^k} u(\bx)y) \;d\mu(\bx) \leqslant c\rho^{k\b} f_k(y) + \rho^{-k\b} T,
    \end{equation}
    for all $k\in\N$ and and all $y\in X$, where $T$ is the constant in Def.~\ref{defn: height functions}(3). In particular, if $k$ is large enough so that $\rho^{k\b}c < 1$, the inequality in~\eqref{eq: CH simplified} recovers the classical form of Margulis inequalities of the form (cf.~\cite{EskinMargulisMozes,EskinMargulis-RandomWalks,BQ-RandomWalkRecurrence,EskinMasur,EMM,KKLM,Khalil-Curves})
    \begin{equation*}
        \int_\Kcal f(g_t u(\bx)y) \;d\mu \leqslant a f(y) + b,
    \end{equation*}
    where $f=f_k$, $t= \rho^k$, $a= c\rho^{k\b} <1$, and $b=T\rho^{-k\b}$.
    
    \item The collection of height functions we use in our applications  will be equivalent in the following sense: for all $k,\ell\in \N$, there exists a constant $C=C(k,\ell)\geq 1$, such that 
        \begin{equation*}
            C^{-1}f_k \leqslant f_\ell \leqslant C f_k.
        \end{equation*}
        Allowing this flexibility in Def.~\ref{defn: height functions} is necessary however for verifying the contraction inequality~\eqref{eqn: CH} in the settings of flows on homogeneous spaces and strata of abelian differentials. We refer the reader to the proofs of Theorem~\ref{thm: constraction in X} and Corollary~\ref{cor: teich f} where this dependence of the height functions on $k$ is exploited. For the purposes of the discussion in this section, this flexibility does not play a role in the proofs and the reader may wish to regard the collection $\Fsc$ as consisting of a single height function.
    
    \item In our applications, the constant $T$, in Def.~\ref{defn: height functions}(3), can be chosen to be independent of $\g$. 
    \end{enumerate}
    \end{remark}
    
    We note that allowing height functions to assume the value $\infty$ has proven useful in several important applications~\cite{BQ-RandomWalkRecurrence,EMM}.

    \begin{definition}\label{def: doa collection}
    In the presence of a collection $\Fsc=\set{f_k}$ of height functions on a metric space $X$ with a $G$-action, we say an orbit $(a_t x)_{t \geqslant 0}$ is $\Fsc$-\textbf{divergent on average}, if for every $k\in\N$ and every $M>0$,
   \[ \frac{1}{T} \int_0^T \chi_{M,k}(a_t x)\;dt \r 0, \]
    where $\chi_{M,k}$ is the indicator function of $\set{y\in X: f_k(y) \leq M}$. We often drop $\Fsc$ from the notation: $\Fsc$-divergent on average when $\Fsc$ is understood from context.  
    \end{definition}

    The following is the main result of this section.
    \begin{theorem} \label{thrm: Hdim and non-divergence}
    	Let $X$ be a metric space equipped with an action by $G=\mrm{SL}(d+1,\R)$. Suppose $\mc{F}$ is an IFS on $\R^d$ satisfying the open set condition and denote by $\mc{K}$ its limit set. Let $s=\dim_H(\Kcal)$ and let $\mu$ be the restriction of the $s$-dimensional Hausdorff measure to $\Kcal$.
    	Assume that $\mu$ satisfies the $(\Fcal,\Fsc,\b,\g_0)$-contraction hypothesis on $X$ for a collection $\Fsc=\set{f_k:k\in\N}$ of height functions and real numbers $0\leq \g_0 <\b \leq s$.
        Then, for all $x_0\in X\backslash \set{f_1=\infty}$, 
        \begin{equation*}
           \dim_H \left(\bx\in \mc{K}: (a_tu(\bx)x_0)_{t\geqslant 0} \text{ is } \Fsc\text{-divergent on average}
           \right) \leqslant s-\b.
        \end{equation*}
    \end{theorem}

    The main applications of our results are to the $G$ action on the space of unimodular lattices in $\R^d$  and to $\mrm{SL}(2,\R)$ actions on moduli spaces of abelian differentials.
    The remainder of this section is dedicated to the proof of Theorem~\ref{thrm: Hdim and non-divergence}.
    
    \subsection{Notation}
   Throughout the remainder of this section, we let $X,G,\Fcal$ and $\mu$ be as in Theorem~\ref{thrm: Hdim and non-divergence}. We fix a finite set $\L$ so that $\Fcal = \set{h_i:i\in \L}$ and denote by $\Kcal$ the limit set of $\Fcal$.
   We use the notation of Definition~\ref{defn: height functions} pertaining to the height functions $f_k$. In particular, for $s=\dim \Kcal$, we have
   \begin{equation*}
       \g_0 < \b \leqslant s.
   \end{equation*}
   
   Moreover, we retain the iterated function systems notation of Section~\ref{sec: prelim}.
   In particular, given a word $\w\in \L^n$, we use the notation:
   \begin{equation*}
       \Kcal_\w := h_\w(\Kcal),
   \end{equation*}
   where $h_\w$ is given by~\eqref{eq: composition order}.
   For an integrable function $\vp$ on $\R^d$, we use $\int \vp \;d\mu$ and $\int_\Kcal \vp\;d\mu$ interchangeably.

\subsection{Integral inequalities and covering estimates}
The goal of this section is to use the contraction hypothesis to control the Hausdorff dimension of divergent orbits. 

    Using the log Lipschitz property $(2)$ of Definition~\ref{defn: height functions}, one can easily verify that the orbit $(a_tu(\bx)x_0)_{t\geqslant 0}$ diverges on average if and only if
    \begin{equation}\label{eq: doa characterization}
        \frac{1}{N} \sum_{l=1}^N \chi_{M,k}(g_{\rho(\bx,lk)}u(\bx)x_0) \r 0, \qquad \text{as } N \r\infty,
    \end{equation}
    for all $k\in\N$ and $M>0$, where $\chi_{M,k}$ is the indicator of $\set{f_k \leq M}$.
    This observation is very useful in handling the case where the contraction ratios of maps in $\Fcal$ are not all the same. This case poses significant difficulties in the proof. We are thus naturally led to studying the following sets:
    for $0<\d < 1$, $M>0$, $x\in X$, and $N,k \in \N$, define $Z_x(M,N,k,\d)$ by
	\begin{equation} \label{defn: Z_x(M, N, delta)}
	Z_x(M, N,k, \d) := \set{ \bx\in \mc{K}: \#\set{ 1\leq l \leq N: f_k(g_{\rho(\bx,lk)}u(\bx)x)>M} > \d N}.
    \end{equation}

    The following proposition is one of the main technical results of this article.
    \begin{proposition} \label{prop: H-measure bound}
    There exists a constant $c_0 \geq 1$, depending only on the support of $\mu$ and on the constant $c$ in Def.~\ref{defn: height functions}(3), such that the following holds.
    For every $k\in \N$, $ \g \in (\g_0,\b)$ and $x\in X\backslash\set{f_1=\infty}$, there exists $M_0 = M_0(k,x,\g) >0$, so that for all $M>M_0$ the following holds.
     For all $0<\d<1$, $N\geq 1$, one has that
     \begin{equation*}
       \sum \mrm{diam}(\hK_\w)^{s-\g} \leqslant c_0^N 
       \bigg(\int_\Kcal \rho(\bx,k)^{-\g} \;d\mu\bigg)^{(1-\d)N}
\bigg(\int_\Kcal \rho(\bx,k) \;d\mu\bigg)^{\d(\b-\g) N},
     \end{equation*}
     where the sum is taken over words $\w \in \L^N$ satisfying $\hK_\w \cap Z_x(M,N,k,\d) \neq \emptyset$.
     Moreover, $M_0$ can be chosen uniformly as $x$ varies in a fixed sub-level set $\set{f_k\leq L}$ for any $L>0$.
    \end{proposition}
    
    \subsection*{Notational Convention} For the remainder of this section, we use $f$ to denote $f_1$ to simplify notation.

    We need technical preparation before the proof which occupies the next $3$ subsections. 
    Define the following constants:
    \begin{equation}\label{eq: R}
        K = \diam{\Kcal},\qquad  R = \sup_{\bx\in \Kcal} \norm{\bx}.
    \end{equation}
    Using $(2)$ of Definition~\ref{defn: height functions}, we can find $A \geq 1$ such that 
        \begin{equation} \label{eq: A}
        	A^{-1} f_k(y) \leqslant  f_k(u(\bx) y) \leqslant A f_k(y),
        \end{equation}
    for all $k\in\N$, $y\in X$ and all $\bx$ in a ball around $\mathbf{0}\in \R^d$ of radius $2R$.
    We also fix a constant $B\geq 1$ so that
    \begin{equation} \label{eq: B}
        	B^{-1} f_k(y) \leqslant  f_k(g_{\rho(\bx,1)} y) \leqslant B f_k(y),
    \end{equation}
    for all $k\in\N$, $\bx\in \Kcal$ and $y\in X$.
    For $x\in X$, $M>0$ and natural numbers $m,n \in \N$, we define the following sets:
	\begin{equation*} \label{defn: B_x(M,t; m+n)}
		B_x(M, m;n) = \set{ \bx\in \mc{K}:
        	f_1(g_{\rho(\bx,m+l)}u(\bx)x) \geqslant M, \text{ for } 1 \leq l \leq n}.
	\end{equation*}
	Recall the definition of the sets $\hat{\Kcal}_\w$ and the partitions $\hat{\Pcal}_n$ in~\eqref{eq: disjoint images}.
	We frequently use the fact that $\diam{\Kcal_\w} = \diam{\hK_\w} = K\rho_\w$ for every $\w \in \cup_k \L^k$ and that
	\begin{equation}
	    \mu(\Kcal_\w) = \mu(\hK_\w) = \rho_\w^s.
	\end{equation}
	These facts follow from Lemmas~\ref{lem: diam K = K hat} and~\ref{lem: transformation of self-similar measures}.

\subsection{Averages of multiplicative cocycles}

    We record the following cocycle relation.
    \begin{lemma}\label{lem: time cocycle on K}
    Let $\w \in \L^n$. Then, for every $m\in \N$ and $\bx\in\hK_\w$,
    \begin{equation*}
        \rho(\bx,n+m) = \rho(h_\w^{-1}(\bx),m)\rho(\bx,n).
    \end{equation*}
    \end{lemma}
    
    \begin{proof}
    Let $u\in \L^m$ be such that $\bx\in \hK_{\w u}$, where $\w u$ is the concatenated word. In particular, by definition
    \[ \Kcal_{\w u} = h_\w (h_u(\Kcal)). \]
    Then, we have $\rho(\bx,m+n) = \rho_\w\rho_u$, where $\rho_\w,\rho_u$ are the contraction ratios of $h_\w$ and $h_u$ respectively.
    Moreover, Lemma~\ref{lem: hK refining} implies that $h_\w^{-1}(\bx) \in \hK_u$.
    Hence, $\rho(h_\w^{-1}(\bx),m) = \rho_u$.
    Finally, since $\hK_u \subset \Kcal$, we see that $\bx\in \hK_\w$ and, in particular, $\rho(\bx,n) = \rho_\w$.
    \end{proof}
    
    The next lemma is a special case of a general principle: averages of ``locally constant" submultiplicative cocycles form a submultiplicative sequence. 
    \begin{lemma}\label{lem: average of submul coc is subadditive}
    For all $ \g \in \R$ and all $n\in \N$,
    \begin{equation*}
        \int \rho(\bx,n)^{\g} \;d\mu(\bx) = \left[\int \rho(\bx,1)^\g \;d\mu(\bx)\right]^{ n}
    \end{equation*}
    \end{lemma}
    
    \begin{proof}
    Let $a_n = \int \rho(\bx,n)^{\g} \;d\mu$.
    Then, for all $m,n \in \N$, by Lemma~\ref{lem: time cocycle on K}, we obtain
    \begin{align*}
        a_{m+n} &= \int \rho(\bx,m+n)^\g \;\mrm{d}\mu 
        = \sum_{\w\in \L^{m}} \int_{\hK_\w} \rho(\bx,m)^\g \rho(h_{\w}^{-1}(\bx),n)^\g \;\mrm{d}\mu\\
        &=  \sum_{\w\in \L^{m}} \rho_\w^\g \int_{\hK_\w} \rho(h_{\w}^{-1}(\bx),n)^\g \;\mrm{d}\mu.
    \end{align*}
   By Lemma~\ref{lem: transformation of self-similar measures}, applied with $\l_\w = \rho(\w,m)^s$, it follows that
    \begin{equation*}
       (h_\w^{-1})_\ast \left(\mu\vert_{\hK_\w}\right) = \rho_\w^s  \mu = \mu(\hK_\w) \mu.
    \end{equation*}
    Hence, we get that
    \begin{equation*}
        a_{m+n} = \int \rho(\bx,n)^\g\;\mrm{d}\mu \sum_{\w\in \L^{m}} \rho_\w^\g \mu(\hK_\w)
        = a_n a_m.
    \end{equation*}
    The lemma follows by induction.
    \end{proof}
    
     The following lemma allows us to control complete sums over covers. 
    \begin{lemma} \label{lem: bound over good intervals}
    Let $k,m \in \N$.
    For every $\a\in \L^k$ and $\g\in \R$,
    \begin{equation*}
        \sum_{\w \in \L^{m}} \diam{\hK_{\a\w}}^{s-\g}
        \leqslant  \diam{\hK_\a}^{s-\g} 
        \int \rho(\bx,m)^{-\g} \;d\mu(\bx) ,
    \end{equation*}
    where, for $\w\in \L^m$, $\a\w$ denotes the concatenation of $\a$ and $\w$. In particular, $\Kcal_{\a\w} = h_\a(\Kcal_\w)$.
    \end{lemma}
    
    \begin{proof}
    By Lemma~\ref{lem: diam K = K hat} and the cocycle property of $\rho$, we have
    \[   \diam{\hK_{\a\w}}=\diam{\Kcal_{\a\w}}=K \rho(\a\w,k+ m) =
            K \rho(\a,k) \rho(\w,m).
    \]
    Moreover, we have that $\rho(\bx,m+k)$ is constant almost everywhere on $\hK_{\a\w}$ and equal to $\rho(\a\w,m+k)$ for every $\w\in\L^m$.
    Thus, using the fact that $\mu(\Kcal_{\a\w})=\rho(\a\w,m+k)^s$, we obtain
    \begin{align*}
        \sum_{\w \in \L^{m}} \diam{\hK_{\a\w}}^{s-\g}
        = K^{s-\g} \sum_{\w\in\L^m} \rho_{\a\w}^{s-\g}
        = K^{s-\g} \sum_{\w\in\L^m} \rho_{\a\w}^{-\g} \mu(\Kcal_{\a\w})
        =  K^{s-\g} \int_{\hK_\a} \rho(\bx,m+k)^{-\g}\;d\mu. 
    \end{align*}
    By Lemma~\ref{eq: cocycle on K}, for almost every $\bx\in \Kcal_\a$
    \begin{align*}
        \rho(\bx,m+k) = \rho(\bx,k)\rho(h_{\a}^{-1}(\bx),m) = \rho(\a,k) \rho(h_{\a}^{-1}(\bx),m).
    \end{align*}
    It follows that
    \begin{align*}
    \sum_{\w \in \L^{m}} \diam{\hK_{\a\w}}^{s-\g}
    =K^{s-\g} \rho(\a,k)^{-\g} \int_{\hK_\a} \rho(h_\a^{-1}(\bx),m)^{-\g}\;d\mu
    \end{align*}
    By Lemma~\ref{lem: transformation of self-similar measures}, for every integrable function $\vp$, we have $\int_{\Kcal_\a}\vp(h_\a^{-1}(\bx))\;d\mu = \mu(\Kcal_\a) \int \vp(\bx)\;d\mu $.
    This implies that
    \begin{align*}
    \sum_{\w \in \L^{m}} \diam{\hK_{\a\w}}^{s-\g}
        &=  K^{s-\g} \rho(\a,k)^{-\g} \mu(\hK_\a) 
        \int \rho(\bx,m)^{-\g}\;d\mu\\
        &= K^{s-\g} \rho(\a,k)^{s-\g} 
        \int \rho(\bx,m)^{-\g}\;d\mu
        = \diam{\hK_\a}^{s-\g} \int \rho(\bx,m)^{-\g}\;d\mu.
    \end{align*}
    \end{proof}
    

\subsection{Consequences of the log-Lipschitz property}
    
    The next $3$ lemmas provide us with simple consequences of the log-Lipschitz property of the function $f$ in Definition~\ref{defn: height functions}.

    \begin{lemma} \label{lem: 0-1 law}
    Suppose $\hat{\Kcal}_\w \in \hat{\Pcal}_{m+n}$ is such that $ \hat{\Kcal}_\w \cap B_x(M,m;n) \neq \emptyset$. Then, 
    \[ \hat{\Kcal}_\w \subseteq B_x(M/A,m;n).\]
    \end{lemma}
    \begin{proof}
    Suppose $\bx_0 \in  \hat{\Kcal}_\w \cap  B_x(M,m;n)$.
    Let $N = m+n$ and let $m< l \leq N$ be such that $f(g_{\rho(\bx_0 , l)}u(\bx_0)x)$ is greater than $M$.
    Let $\bx\in \hat{\Kcal}_\w$ be any other vector.
    Note that $\rho(\cdot,l)$ is constant on elements of $\hat{\Pcal}_{N}$.
    This implies
    \begin{equation*}
        g_{\rho(\bx,l)}u(\bx) = u\left(\rho(\bx_0,l)^{-1}(\bx-\bx_0)\right)g_{\rho(\bx_0,l)}u(\bx_0).
    \end{equation*}
    Let $\mbf{y},\mbf{y}_0 \in \hat{\Kcal}$ be such that $\bx = h_\w(\mbf{y})$ and $\bx_0 = h_\w(\mbf{y}_0)$. The invariance of the Euclidean norm by $\mrm{SO}(d,\R)$ implies 
    \begin{equation*}
         \norm{\bx - \bx_0 } = \norm{h_\w(\mbf{y}) -h_\w(\mbf{y}_0)  } = \rho(\w,N) \norm{\mbf{y}-\mbf{y}_0}\leqslant 2\rho(\w,N)R.
    \end{equation*}
    Thus, since $\rho(\bx_0,l) \geq \rho(\bx_0,N) = \rho(\w,N)$, the choice of the constant $A$ in~\eqref{eq: A} implies
    \begin{equation*}
        f(g_{\rho(\bx,l)}u(\bx)x) \geqslant M/A.
    \end{equation*}
    This being true for all $\bx\in \hat{\Kcal}_\w$ concludes the proof.
    \end{proof}
    
    \begin{lemma}\label{lem: b_omega is also bad}
    Suppose $f( g_{\rho(\bx_0,\ell)} u(\bx_0)y) > M $ for some $y\in X$, $\bx_0\in\hK_\w$ and some $\w \in \L^\ell$. Then,
    \[f(g_{\rho(\w,\ell)} u(b_\w) y) > M/A,\] where $b_\w = h_\w(0)$.
    \end{lemma}
    \begin{proof}
    The proof is completely anaolgous to that of Lemma~\ref{lem: 0-1 law}.
    \end{proof}
    
    \begin{lemma}\label{lem: f is harmonic}
    Let $\w \in \L^\ell$. Then, for all $y\in X$,
    \begin{equation*}
        f(g_{\rho(\w,\ell)} u(b_\w)y) \leqslant \frac{A}{\mu(\hK_\w)} \int_{\hK_\w} f(g_{\rho(\bx,\ell)} u(\bx)y) \;d\mu,
    \end{equation*}
    where $A$ is given by~\eqref{eq: A} and $b_\w = h_\w(0)$.
    \end{lemma}
    \begin{proof}
    The proof follows from the fact that $\rho(\bx,\ell) =\rho(\w,\ell)$ everywhere on $\hK_\w$ and the following estimate:
    \[ \rho(\w,\ell)^{-1} \norm{b_\w - \bx} \leqslant R, \]
    for all $\bx\in \hK_\w$, where $R$ is given by~\eqref{eq: R}.
    \end{proof}


\subsection{Consequences of the contraction property}

    For every $\ell\in\N$ and $\w\in \L^\ell$, we define the following elemenets of $\mrm{SO}(d+1,\R)$:
   \begin{equation}
       k_\w = \begin{pmatrix}
    	1& \mbf{0} \\
    	\mbf{0} & O_\w
    	\end{pmatrix},
    \end{equation}
    where $O_\w $ is the rotation part of the similarity $h_\w$ and is given by~\eqref{eq: composition parameters}.
    Note that $O_\w \in \mrm{SO(d,\R)}$ and that each $k_\w$ commutes with $g_t$.
    The following lemma is the first main step in the proof of Proposition~\ref{prop: H-measure bound}.
    \begin{lemma} \label{lem: contract once}
    Let $A>0$ be the constant in~\eqref{eq: A}.
    Let $\g\in [\g_0,\b]$ and let $T >0$ be the constant provided by $(3)$ of Definition~\ref{defn: height functions} with $k=1$.
    Suppose that $f( g_{\rho(\bx_0,\ell)} u(\bx_0)y) > AT $ for some $y\in X$, $\ell\in\N$, $\w\in \L^\ell$, and $\bx_0\in\hK_\w$. Then,
    
    \begin{align*}
        \int_{\hK_\w} \rho(\bx,\ell+1)^{-\g} f(g_{\rho(\bx,\ell+1)} u(\bx)y) &\;d\mu(\bx) \\
        &\leqslant cA \int_{\hK_\w} \rho(\bx,\ell)^{-\g} f(g_{\rho(\bx,\ell)} u(\bx)y) \;d\mu(\bx) 
        \left(\int \rho(\bx,1) \;d\mu(\bx)\right)^{\b-\g},
    \end{align*}
    where $c$ is as in $(3)$ of Definition~\ref{defn: height functions}.
    \end{lemma}
    
    \begin{proof}
    By Lemma~\ref{lem: time cocycle on K}, it follows that
    \begin{align}\label{eq: apply cocycle}
        \int_{\hK_\w} \rho(\bx,\ell+1)^{-\g} f(g_{\rho(\bx,\ell+1)} u(\bx)y) \;d\mu
        &= \int_{\hK_\w} \rho(\bx,\ell)^{-\g} \rho(h_\w^{-1}(\bx),1)^{-\g} f(g_{\rho(\bx,\ell+1)} u(\bx)y) \;d\mu \nonumber\\
        &= \rho(\w,\ell)^{-\g} \int_{\hK_\w}  \rho(h_\w^{-1}(\bx),1)^{-\g} f(g_{\rho(\bx,\ell+1)} u(\bx)y) \;d\mu,
    \end{align}
    where on the second line we used the fact that $\rho(\bx,\ell)= \rho(\w,\ell)$ everywhere on $\hK_\w$.
    
    Note that since $O_\w \in \mrm{SO}(d,\R)$, we have $O_\w^{-1} = O_\w^t$, where $O_w^t$ denotes the transpose of $O_\w$. Moreover, $h_\w^{-1}(\bx) = \rho(\w,\ell)^{-1}O_\w^{-1}(\bx-b_\w)$, where $b_\w=h_\w(0)$.
    Thus, the following identity holds.
    \begin{equation*}
        g_{\rho(\bx,\ell+1)}k_\w u(\bx) =     g_{\rho(h_\w^{-1}(\bx),1)}u(h_\w^{-1}(\bx))g_{\rho(\w,\ell)}k_\w u(b_\w).
    \end{equation*}
    Observe that $k_\w$ commutes with $g_t$ and recall that the function $f$ is $\mrm{SO}(d+1,\R)$-invariant.
    This implies
    \begin{equation}\label{eq: commutation relation}
       f(g_{\rho(\bx,\ell+1)} u(\bx)y) =
       f( g_{\rho(h_\w^{-1}(\bx),1)}u(h_\w^{-1}(\bx))g_{\rho(\w,\ell)}k_\w u(b_\w)y).
    \end{equation}
    By Lemma~\ref{lem: transformation of self-similar measures}, for every $\vp\in\mrm{L}^1(\mu)$,
    \begin{equation*}
        \frac{1}{\mu(\hK_\w)} \int_{\hK_\w} \vp(h_\w^{-1}(\bx)) \;d\mu(\bx) = \int \vp(\bx)\;d\mu(\bx).
    \end{equation*}
    Combining this fact with~\eqref{eq: commutation relation}, we obtain, for $z=g_{\rho(\w,\ell)}k_\w u(b_\w)y$,
    \begin{equation}\label{eq: apply comm rel, invariance}
        \int_{\hK_\w}  \rho(h_\w^{-1}(\bx),1)^{-\g} f(g_{\rho(\bx,\ell+1)} u(\bx)y) \;d\mu
        = \mu(\hK_\w) \int \rho(\bx,1)^{-\g} f(g_{\rho(\bx,1)}u(\bx)z) \;d\mu.
    \end{equation}
    Since $k_\w$ commutes with $g_t$ and $f$ is invariant by $k_\w$, it follows that $f(z)=f(g_{\rho(\w,\ell)}u(b_\w)y)$.
    In particular, by Lemma~\ref{lem: b_omega is also bad}, we have $f(z) > T$.
    Thus, the contraction property of $f$ in $(3)$ of Definition~\ref{defn: height functions} implies
    \begin{equation*}
        \int \rho(\bx,1)^{-\g} f(g_{\rho(\bx,1)}u(\bx)z) \;d\mu
        \leqslant c f(z)  \left(\int \rho(\bx,1) \;d\mu\right)^{\b-\g}.
    \end{equation*}
    Finally, we apply Lemma~\ref{lem: f is harmonic} to get
    \begin{equation}\label{eq: apply harmonic}
        f(z) \leqslant \frac{A}{\mu(\hK_\w)} \int_{\hK_\w} f(g_{\rho(\bx,\ell)} u(\bx)y) \;d\mu.
    \end{equation}
    Combining~\eqref{eq: apply cocycle},~\eqref{eq: apply comm rel, invariance}, and~\eqref{eq: apply harmonic}, along with the fact that $\rho(\bx,\ell)=\rho(\w,\ell)$ for all $\bx\in\hK_\w$  yields the desired estimate and concludes the proof.
    \end{proof}
    
    The following lemma uses Lemma~\ref{lem: contract once} as a base step in an inductive procedure to establish an exponentially decaying estimate for similar averages over points with long cusp excursions.
    \begin{lemma}\label{lem: induction on contraction}
    Let $A>0$ and $B>0$ be the constants in~\eqref{eq: A} and~\eqref{eq: B} respectively.
    Let $\g\in [\g_0,\b]$ and let $T >0$ be the constant provided by $(3)$ of Definition~\ref{defn: height functions} with $k=1$.
    For all $M>ABT$, $m,n\in \N$, $\a\in \L^m$ and $x\in X$,
    
    \begin{align*}
        \int_{B_x(M,m;n-1)\cap\hK_\a} \rho(\bx,m+n)^{-\g} f(g_{\rho(\bx,m+n)}& u(\bx)x) \;d\mu \leqslant \\
        &\leqslant \th^n  \int_{B_x(M,m;n-1)\cap\hK_\a} \rho(\bx,m)^{-\g} f(g_{\rho(\bx,m)} u(\bx)x) \;d\mu 
        ,
    \end{align*}
    where $\th$ is given by:
    \begin{equation}\label{eq: theta induction lem}
        \th = cA \left(\int \rho(\bx,1) \;d\mu\right)^{(\b-\g)},
    \end{equation}
    and $c$ is as in $(3)$ of Definition~\ref{defn: height functions}.
    \end{lemma}
    
    \begin{proof}
    If $n=1$, then $B_x(M,m;n-1)=\emptyset$ and the statement follows trivially. Thus, we may assume that $n>1$ and that $B_x(M,m,n-1)\neq \emptyset$.
    Let $H =B_x(M,m;n-1)\cap\hK_\a $.
    Lemma~\ref{lem: hK refining} implies that $H \subseteq \bigcup \hK_\w$, where the union is taken over words $\w\in \L^{m+n-1}$ so that $\hK_\w\cap H\neq \emptyset$. In particular, we get
    \begin{align}\label{eq: apply contraction once}
        \int_{H} &\rho(\bx,m+n)^{-\g} f(g_{\rho(\bx,m+n)} u(\bx)x) \;d\mu \leqslant \nonumber \\
        &\leqslant
        \sum_{\substack{\w\in \L^{m+n-1}\\ \hK_\w\cap H\neq\emptyset }} \int_{\hK_\w} \rho(\bx,m+n)^{-\g} f(g_{\rho(\bx,m+n)} u(\bx)x) \;d\mu\leqslant \nonumber\\
        &\leqslant 
        \theta
        \sum_{\substack{\w\in \L^{m+n-1}\\ \hK_\w\cap H\neq\emptyset }}
        \int_{\hK_\w} \rho(\bx,m+n-1)^{-\g} f(g_{\rho(\bx,m+n-1)} u(\bx)x) \;d\mu
         \quad \text{by Lemma~\ref{lem: contract once}},
    \end{align}
    where $\theta$ is given by~\eqref{eq: theta induction lem}. 
    Next, we note that the following inclusion holds by Lemma~\ref{lem: hK refining}.
    \begin{equation} \label{eqn: refining property}
        \bigcup_{\substack{\w\in \L^{m+n-1}\\ \hK_\w\cap H\neq\emptyset }} \hK_\w
         \subseteq
         \bigcup_{\substack{\w\in \L^{m+n-2}\\ \hK_\w\cap H\neq\emptyset } } \hK_\w.     
        \end{equation}
    Moreover, Proposition~\ref{prop: null overlap} shows that the indicator functions of the above unions are equal almost everywhere to sums of the indicator functions of the members of the union.
    Hence, combining~\eqref{eq: apply contraction once}, and~\eqref{eqn: refining property} yields the following estimate
    \begin{align*}
        \int_{H} \rho(\bx,m+n)^{-\g} f(g_{\rho(\bx,m+n)} &u(\bx)x) \;d\mu  \nonumber \\
        &\leqslant \theta   \sum_{\substack{\w\in \L^{m+n-2}\\ \hK_\w\cap H\neq\emptyset } }
        \int_{\hK_\w} \rho(\bx,m+n-1)^{-\g} f(g_{\rho(\bx,m+n-1)} u(\bx)x) \;d\mu.
    \end{align*}
    By an iterated application of~\eqref{eq: apply contraction once} and~\eqref{eqn: refining property}, we obtain
    \begin{align*}
        \int_{H} \rho(\bx,m+n)^{-\g} f(g_{\rho(\bx,m+n)} &u(\bx)x) \;d\mu  \nonumber \\
        &\leqslant \theta^{n-1}   \sum_{\substack{\w\in \L^{m}\\ \hK_\w\cap H\neq\emptyset } }
        \int_{\hK_\w} \rho(\bx,m+1)^{-\g} f(g_{\rho(\bx,m+1)} u(\bx)x) \;d\mu.
    \end{align*}
    To apply Lemma~\ref{lem: contract once} at this stage, we need to ensure that for each $\w\in \L^m$ such that $\hK_\w\cap H\neq\emptyset$, we have that $f(g_{\rho(\bx_0,m)} u(\bx_0)x)>AT$ for some $\bx_0 \in \hK_\w$. Recall that we are assuming that $n>1$ and $H\neq \emptyset$. 
    Let $\bx_0\in \hK_\w \cap H$ for some $\w\in \L^m$.
    Then, we have that 
    \[f(g_{\rho(\bx_0,m+1)} u(\bx_0)x) \geqslant M > ABT.\]
    Moreover, by Lemma~\ref{lem: time cocycle on K}, we can write $\rho(\bx_0,m+1) = \rho(\bx_0,m)\rho(\mbf{y},1)$, for some $\mbf{y}\in \Kcal$.
    Thus, by the choice of $B$ in~\eqref{eq: B}, this implies that $f(g_{\rho(\bx_0,m)} u(\bx_0)x) > AT$. Hence, we obtain
     \begin{align*}
        \int_{H} \rho(\bx,m+n)^{-\g} f(g_{\rho(\bx,m+n)} u(\bx)x) \;d\mu 
        \leqslant \theta^{n}   \sum_{\substack{\w\in \L^{m}\\ \hK_\w\cap H\neq\emptyset } }
        \int_{\hK_\w} \rho(\bx,m)^{-\g} f(g_{\rho(\bx,m)} u(\bx)x) \;d\mu.
    \end{align*}
    Finally, by Lemma~\ref{lem: 0-1 law}, it follows that 
    \begin{equation*}
        \bigcup_{\substack{\w\in \L^{m}\\ \hK_\w\cap H\neq\emptyset } } \hK_\w \subseteq B_x(M/A,m;n-1) \cap \hK_\a.
    \end{equation*}
    \end{proof}

    The next ingredient is to provide an upper estimate of the sum appearing in Proposition~\ref{prop: H-measure bound} using integral estimates of the height function $f$.
    The next lemma is a first step in that direction.
    \begin{lemma}\label{lem: H-measure to mu}
    For all $x\in X$, $\g\in \R$, $m,n\in \N$, and any $\a\in \L^m$,
    \begin{align*}
    \sum_{\substack{\w\in \L^{m+n} \\\hK_\w \cap B_x(M,m;n) \cap \hK_\a\neq \emptyset}}
    \mrm{diam}(\hK_\w &)^{s-\g} \leqslant\\
    &\frac{K^{s-\g}}{M/A} 
    \int_{B_x(M/A,m;n-1)\cap\hK_\a} \rho(\bx,m+n)^{-\g} f(g_{\rho(\bx,m+n)}u(\bx)x) \;d\mu.
     \end{align*}
     where $K$ is given by~\eqref{eq: R} and $A$ is as in~\eqref{eq: A}.
    
    \end{lemma}
    
    \begin{proof}
    
    For any finite word $\w\in \L^N$, we have
    \begin{equation*}
        \diam{\hK_\w} =K \rho(\w,N).
    \end{equation*}
    Moreover, we have that $\mu(\hK_\w)= \mu(\Kw) = \rho(\w,N)^s$.
    Indeed, this follows from the self-similarity of $\mu$ in~\eqref{eqn: self similar measure} and Proposition~\ref{prop: null overlap} showing that the distinct sets $\Kw$ have null overlap.
    To simplify notation, let $H =  B_x(M,m;n) \cap \hK_\a$, $N=m+n$ and define
     \[ B = \bigcup_{\w\in \L^N:\hat{\Kcal}_\w \cap H \neq \emptyset} \hK_\w.  \]
    
    The other ingredient is to note that $\rho(\bx,N) = \rho(\w,N)$ for every $\bx\in \hK_\w$. Hence, we get
    \begin{equation*}
        \sum_{\w\in \L^N:\hK_\w \cap H\neq \emptyset} \mrm{diam}(\hK_\w)^{s-\g} \leqslant
        K^{s-\g}
        \int \rho(\bx,m+n)^{-\g} \chi_B(\bx) \;d\mu(\bx).
    \end{equation*}
    In view of Lemma~\ref{lem: 0-1 law}, we have that
    \begin{equation*}
        B \subseteq B_x(M/A,m;n).
    \end{equation*}
    Finally, we observe that the following inequality
    \begin{equation*}
        \chi_{B_x(M/A,m;n)}(\bx) \leqslant \frac{f(g_{\rho(\bx,m+n)}u(\bx)x)}{M/A}
    \end{equation*}
    holds for all $\bx \in \R^d$ by definition, where for a set $G$, $\chi_G$ denotes its indicator function.
    \end{proof}

    \subsection{Proof of Proposition~\ref{prop: H-measure bound}}
    Fix some $k\in \N$ and define
    \begin{equation*}
        \Fcal^k = \set{h_\a:\a\in\L^k},\qquad \Fsc^k := \set{f_{\ell k}: \ell\in\N}\subseteq \Fsc.
    \end{equation*}

    Our hypothesis implies that $\mu$ satisfies the $(\Fcal^k,\Fsc^k,\b,\g_0)$-contraction hypothesis (note that $\mu$ is self-similar with respect to $\Fcal^k$).
    In particular, Lemmas~\ref{lem: time cocycle on K} -~\ref{lem: H-measure to mu} hold with $\Fcal^k$ and $\Fsc^k$ in place of $\Fcal$ and $\Fsc$ respectively. 
    In our proof below, the only dependence of the constants on $k$ is in the constant $B$ defined in~\eqref{eq: B} (since the cocycle $\rho(\cdot,1)$ for $\Fcal^k$ is given by $\rho(\cdot,k)$ for $\Fcal$), and $T$ given by $(3)$ of Definition~\ref{defn: height functions}. This dependence will appear only in the choice of the constant $M_0$ in~\eqref{eq: M_0} below.
    Since our conclusion states that $M_0$ depends on $k$ and by replacing $\Fcal$ and $\Fsc$ by $\Fcal^k$ and $\Fsc^k$ respectively, we may hence assume that $k=1$. 
    For simplicity, we use the following notation:
    \begin{equation*}
        Z_x(M,N,\d) := Z_x(M,N,1,\d), \qquad f:= f_1.
    \end{equation*}

    Let $x\in X\setminus \set{f=\infty}$ and $\g\in (\g_0,\b)$ be given.
    Let $T >0$ be the constant provided by $(3)$ of Definition~\ref{defn: height functions} with $k=1$.
    We define $M_0$ as follows
        \begin{equation} \label{eq: M_0}
        	M_0 = \max\set{AB T, f(x)},
        \end{equation}
    where $A$ and $B$ are the constants in~\eqref{eq: A} and~\eqref{eq: B} respectively.
    Suppose $M>M_0$, $N\in \N$ and $0<\d<1$ are given.
    To simplify notation, let
    \begin{equation}\label{eq: xi, zeta, theta}
        \xi = \int \rho(\bx,1)^{-\g} \;d\mu, \qquad
        \zeta= \left(\int \rho(\bx,1) \;d\mu\right)^{\b-\g}, \qquad
        \theta = cA \zeta,
    \end{equation}
    where $c\geq 1$ is the constant in $(3)$ of Definition~\ref{defn: height functions}.

    Consider a subset $Q \subseteq \set{1,\dots,N}$ containing at least $\d N$ elements.
    Define the following set of trajectories whose behavior is determined by $Q$:
    \begin{equation*}
    	Z(N,Q) = \set{\bx \in Z_x(M, N, \d): f(g_{\rho(\bx,l)}u(\bx)x) \geqslant M \textrm{ iff } l \in Q}.
    \end{equation*}
    We decompose the set $Q$ and its complement into maximal ``connected" intervals
    as follows
    \begin{equation*}
    	Q = \bigsqcup_{i=1}^q B_i, \qquad \set{1,\dots,N}\setminus Q = \bigsqcup_{j=1}^p G_j,
    \end{equation*}
    for some integers $p,q \geq 0$. Note that $|p-q| \leq 1$.
    We claim that 
    \begin{equation}\label{eq: main claim}
       \sum_{\substack{\w\in\L^N\\ \hK_\w\cap Z(N,Q)\neq\emptyset}} \mrm{diam}(\hat{\Kcal}_\w)^{s-\g} \leqslant (cA)^{3N} 
       \big(\zeta^\d\xi^{1-\d} \big)^{N}.
     \end{equation}
    Since the set $Z_x(M, N, \d)$ is a union of at most $2^N$ subsets of the form $Z(N,Q)$, the claim of the proposition follows by taking $c_0 = 2(cA)^3$.

    Order the intervals $B_i$ and $G_j$ in the way they appear in the sequence $1 \leq\cdots\leq N$.
    Write $I_r$ for the $r^{\text{th}}$ interval in this sequence for $1\leq r\leq p+q$.
    For a subset $J\subseteq \set{1,\dots,N}$, we use $|J|$ to denote its cardinality.
    For purposes of induction, we write $\L^0$ for a set with one element and $\hK_\a := \Kcal$ whenever $\a\in \L^0$.
    
    \begin{case} $I_{p+q} \subseteq Q$ so that $I_{p+q} = B_q$.
    Let $\a \in \L^{N-|B_q|}$ be such that $\hK_\a \cap Z(N,Q) \neq \emptyset$.
    Then, we note that
    \begin{equation*}
        Z(N,Q)\cap \hK_\a \subseteq B_x(M,N-|B_q|;|B_q|) \cap \hK_\a,
    \end{equation*}
    where the sets $B_x(\cdot,\cdot,\cdot)$ were defined in~\eqref{defn: B_x(M,t; m+n)}.
    Hence, we may apply Lemma~\ref{lem: H-measure to mu}
    with $m=N-|B_q|$ and $n=|B_q|$ to get
    \begin{align*}
       \sum_{\substack{\w\in\L^N\\ \hK_\w\cap Z(N,Q)\cap \hK_\a\neq\emptyset}} \mrm{diam}(\hat{\Kcal}_\w)^{s-\g} &\leqslant \sum_{\substack{\w\in\L^N\\ \hK_\w\cap  B_x(M,m;n)\cap \hK_\a\neq\emptyset}} \mrm{diam}(\hat{\Kcal}_\w)^{s-\g} \\
       &\leqslant \frac{K^{s-\g}}{M/A} 
    \int_{B_x(M/A,m;n-1)\cap\hK_\a} \rho(\bx,m+n)^{-\g} f(g_{\rho(\bx,m+n)}u(\bx)x) \;d\mu,
    \end{align*}
    where $K = \diam{\Kcal}$.
    We can then apply Lemma~\ref{lem: induction on contraction} to get
    \begin{align*}
       \sum_{\substack{\w\in\L^N\\ \hK_\w\cap Z(N,Q)\cap \hK_\a\neq\emptyset}} \mrm{diam}(\hat{\Kcal}_\w)^{s-\g}
       \leqslant  \frac{\th^n K^{s-\g}}{M/A}  \int_{B_x(M,m;n-1)\cap\hK_\a} \rho(\bx,m)^{-\g} f(g_{\rho(\bx,m)} u(\bx)x) \;d\mu,
    \end{align*}
    where $\th$ is defined in~\eqref{eq: xi, zeta, theta}.
    
    Recall that $\a\in \L^{m}$ was chosen so that $\hK_\a \cap Z(N,Q)\neq \emptyset$.
    Moreover, the choice of $M_0$ implies that $1\in G_1$. In particular, since $I_{p+q}=B_q$ is a maximal sub-interval of $Q$, we see that $f(g_{\rho(\bx_0,m)} u(\bx_0)x) <M$ for some $\bx_0 \in \hK_\a$.
    The choice of the constant $A$ in~\eqref{eq: A} then implies that $f(g_{\rho(\bx,m)} u(\bx)x) \leqslant MA$ for all $\bx\in \hK_\a$. 
    
    Moreover, $\rho(\bx,m)$ is constant everywhere on $\hK_\a$. It follows that
    \begin{align*}
        \sum_{\substack{\w\in\L^N\\ \hK_\w\cap Z(N,Q)\cap \hK_\a\neq\emptyset}} \mrm{diam}(\hat{\Kcal}_\w)^{s-\g}
       &\leqslant  \frac{\th^n K^{s-\g}}{M/A} MA \rho(\a,m)^{-\g} \mu(\hK_\a) 
       = \theta^n A^2 \diam{\hK_\a}^{s-\g}.
    \end{align*}
    In the last equality, we used the fact that $\diam{\hK_\a} =K\rho(\a,m)$ and $\mu(\hK_\a)=\mu(\Kcal_\a)= \rho(\a,m)^s$.
    This follows from Lemmas~\ref{lem: diam K = K hat} and~\ref{lem: transformation of self-similar measures} respectively.
    We thus arrive at the following estimate
    \begin{align*}
        \sum_{\substack{\w\in\L^N\\ \hK_\w\cap Z(N,Q)\neq\emptyset}} \mrm{diam}(\hat{\Kcal}_\w)^{s-\g} &\leqslant
        \sum_{\substack{\a \in \L^{N-|B_q|}\\\hK_\a\cap Z(N,Q) \neq \emptyset }} \text{ }
        \sum_{\substack{\w\in\L^N\\ \hK_\w\cap Z(N,Q)\cap \hK_\a\neq\emptyset}} \mrm{diam}(\hat{\Kcal}_\w)^{s-\g}\\
        &\leqslant \theta^{|B_q|} A^2 \sum_{\substack{\a \in \L^{N-|B_q|}\\\hK_\a\cap Z(N,Q) \neq \emptyset }} \diam{\hK_\a}^{s-\g}.
    \end{align*}
    In view of the inclusion $Z(N,Q) \subseteq Z(N-|B_q|,Q\setminus B_q)$,
    it follows that
    \begin{align}\label{eq: case 1}
        \sum_{\substack{\w\in\L^N\\ \hK_\w\cap Z(N,Q)\neq\emptyset}} \mrm{diam}(\hat{\Kcal}_\w)^{s-\g} \leqslant \theta^{|B_q|} A^2 \sum_{\substack{\a \in \L^{N-|B_q|}\\\hK_\a\cap Z(N-|B_q|,Q\setminus B_q) \neq \emptyset }} \diam{\hK_\a}^{s-\g}.
    \end{align}
    \end{case}
    
    \begin{case} $I_{p+q} \subseteq \set{1,\dots,N}\setminus Q $ so that $I_{p+q} = G_p$.
    In this case, we apply Lemma~\ref{lem: bound over good intervals} to obtain
    \begin{align*}
    \sum_{\substack{\w\in\L^N\\ \hK_\w\cap Z(N,Q)\neq\emptyset}} \mrm{diam}(\hat{\Kcal}_\w)^{s-\g} &\leqslant
        \sum_{\substack{\a \in \L^{N-|G_p|}\\\hK_\a\cap Z(N,Q) \neq \emptyset }} \sum_{\substack{\w\in\L^N\\ \hK_\w\cap Z(N,Q)\cap \hK_\a\neq\emptyset}} \mrm{diam}(\hat{\Kcal}_\w)^{s-\g}\\
        &\leqslant \sum_{\substack{\a \in \L^{N-|G_p|}\\\hK_\a\cap Z(N,Q) \neq \emptyset }} \diam{\hK_\a}^{s-\g} \int \rho(\bx,|G_p|)^{-\g} \;d\mu(\bx).
    \end{align*}
    Then, using Lemma~\ref{lem: average of submul coc is subadditive}, it follows that
    \begin{align*}
    \sum_{\substack{\w\in\L^N\\ \hK_\w\cap Z(N,Q)\neq\emptyset}} \mrm{diam}(\hat{\Kcal}_\w)^{s-\g} &\leqslant
    \xi^{|G_p|}
    \sum_{\substack{\a \in \L^{N-|G_p|}\\\hK_\a\cap Z(N,Q) \neq \emptyset }} \diam{\hK_\a}^{s-\g}.
    \end{align*}
    Finally, using that $Z(N,Q) \subseteq Z(N-|G_p|,Q)$, we obtain
    \begin{equation}\label{eq: case 2}
        \sum_{\substack{\w\in\L^N\\ \hK_\w\cap Z(N,Q)\neq\emptyset}} \mrm{diam}(\hat{\Kcal}_\w)^{s-\g} \leqslant
    \xi^{|G_p|}
    \sum_{\substack{\a \in \L^{N-|G_p|}\\\hK_\a\cap Z(N-|G_p|,Q) \neq \emptyset }} \diam{\hK_\a}^{s-\g}.
    \end{equation}
    \end{case}
    
    Equipped with the estimates in~\eqref{eq: case 1} and~\eqref{eq: case 2}, we iteratively bound the sum in~\eqref{eq: main claim} by similar sums over covers of sets of the form $Z(L, V)$ with $L<N$ and $V\subseteq Q$ yielding the following upper bound by downward induction on $N$:
    \begin{align*}
        \sum_{\substack{\w\in\L^N\\ \hK_\w\cap Z(N,Q)\neq\emptyset}} \mrm{diam}(\hat{\Kcal}_\w)^{s-\g} &\leqslant
        \theta^{|Q|} A^{2q} \xi^{(N-|Q|)}.
    \end{align*}
    Recall that $ \th=cA\zeta $ and $q\leq N$.
    Moreover, since $\g$ is positive and strictly less than $\b$, we have that $\zeta <1$, $\xi >1$.
    Hence, since $|Q|\geq \d N$, it follows that
    \begin{align}
        \sum_{\substack{\w\in\L^N\\ \hK_\w\cap Z(N,Q)\neq\emptyset}} \mrm{diam}(\hat{\Kcal}_\w)^{s-\g} 
         \leqslant (cA)^{3N} \big(\zeta^{\d } \xi^{(1-\d)}\big)^N.
    \end{align}
    This implies the main claim~\eqref{eq: main claim}.

\subsection{Proof of Theorem~\ref{thrm: Hdim and non-divergence}}
	
    Having established Proposition~\ref{prop: H-measure bound}, the proof of Theorem~\ref{thrm: Hdim and non-divergence} follows from the definition of Hausdorff dimension.
    Recall the definition of the Hausdorff (outer) measures given in Section~\ref{sec: hdim}.
    
    Let $x\in X$ and let $Z_x \subseteq \Kcal$ denote the set of vectors $\bx$ for which the trajectory $a_t u(\bx)x$ diverges on average.
    As we noted in~\eqref{eq: doa characterization}, in view of the log Lipschitz property $(2)$ of Definition~\ref{defn: height functions}, the set $Z_x$ is related to the sets $Z_x(M,N,k,\d)$, defined in~\eqref{defn: Z_x(M, N, delta)}, via the following inclusion:
        \begin{align} \label{Z contained in liminf set}
        	Z_x \subseteq \liminf_{N\r\infty} Z_x(M,N,k,\d)
        		= \bigcup_{N_0\geq 1} \bigcap_{N\geq N_0} Z_x(M,N,k,\d),
        \end{align}
        for all $M,\d>0$ and $k\in\N$.
        We wish to apply Proposition~\ref{prop: H-measure bound}.
        Fix some $\g \in (\g_0,\b)$ and let $c_0 \geq 1$ be as in the conclusion of the proposition.
        Recall the definition of $\xi>1$ and $\zeta<1$ in~\eqref{eq: xi, zeta, theta}.
        Let $\d \in (0,1)$ be sufficiently close to $1$ such that 
        \begin{equation*}
            \zeta^\d \xi^{1-\d} <1.
        \end{equation*}
        Note that $\d\r 1$ as $\g\r\b$.
        Choose $k\in \N$ large enough so that
        \begin{equation}\label{eq: meaning of contraction}
            \big(\zeta^\d \xi^{1-\d}\big)^k= 
            \bigg(\int \rho(\bx,k) \;d\mu\bigg)^{(\b-\g)\d}
            \bigg(\int \rho(\bx,k)^{-\g} \;d\mu\bigg)^{1-\d}
            < 1/c_0,
        \end{equation}
        where we used Lemma~\ref{lem: average of submul coc is subadditive} in the first equality.
        Let $M_0=M_0(k,x,\g) >0$ be as in Proposition~\ref{prop: H-measure bound} and suppose $M>M_0$.
        For each $N\in \N$, define
        \begin{equation*}
            \kappa(N) = \max\set{\diam{\Kcal_\w}: \w\in \L^{kN}}.
        \end{equation*}
        
        Suppose $N_0 \in \N$ is given.
        Then, for every $J\geq N_0$, Proposition~\ref{prop: H-measure bound} and the first equality in~\eqref{eq: meaning of contraction} show that
        \begin{equation*}
            H_{\kappa(N_0)}^{s-\g} \left( \bigcap_{N\geq N_0} Z_x(M,N,k,\d) \right) \leqslant
            H_{\kappa(N_0)}^{s-\g} \left(  Z_x(M,J,k,\d) \right)\leqslant
            \left(c_0 \big(\zeta^\d \xi^{1-\d}\big)^k\right)^{J}.
        \end{equation*}
        By taking $J$ to infinity, it follows that
        \begin{equation*}
            H^{s-\g}\left(\bigcap_{N\geq N_0} Z_x(M,N,k,\d) \right) =0.
        \end{equation*}
        Recall that $\dim_H(\cup_n A_n) = \sup_n \dim_H(A_n)$ for any countable collection of Borel sets $A_n$.
        Since $Z_x$ is contained in the union of countably many sets of the form $\bigcap_{N\geq N_0} Z_x(M,N,k,\d)$ by~\eqref{Z contained in liminf set}, it follows that $\dim_H(Z_x) \leqslant s-\g$.
        This being true for $\g_0\leq \g<\b$, this implies that $\dim_H(Z_x)\leqslant s-\b$ as desired.

    
    \section{Transversality of Expanding Coordinates}
\label{section: transverse}

  In this section, we establish the first step towards verifying the contraction hypothesis on the space of lattices $\mrm{SL}(d+1,\R)/\mrm{SL}(d+1,\Z)$.
  The goal is to prove a key observation which allows us to obtain optimal average contraction rates with respect to any measure $\mu$ for which $\a_\ell(\mu)>0$ for every $\ell$, where $\a_\ell(\mu)$ is defined in~\eqref{eqn: proj spec}.
  The main results are Proposition~\ref{prop: transversality} and~\ref{prop: transverse implies integrable}.
  
  \subsection{The exterior power representation}
  We begin by giving a description of the coordinates of the fundamental representation of $G=\mrm{SL}(d+1,\R)$ on the following vector space
  \[ V= \bigoplus_{\ell=1}^{d} \bigwedge^\ell \R^{d+1}. \]
  An element $g \in G $ acts on $V$ via the linear map $\bigoplus_{k=1}^{d} \bigwedge^\ell g$.
  Consider the basis $\set{\mbf{e}_0,\dots, \mbf{e}_d}$ of $\R^{d+1}$, where $\mbf{e}_i$ denotes the $i^{th}$ standard basis element.
  For each index set $I = \set{ i_1 < \cdots < i_\ell } \subset\set{0,\dots,d}$, we let
  \begin{equation} \label{eqn: basis elements}
  	\mbf{e}_I := \mbf{e}_{i_1} \wedge \cdots \wedge \mbf{e}_{i_\ell}.
  \end{equation}
  The collection of monomials $\mbf{e}_I$ gives a basis of $V_\ell = \bigwedge^\ell \R^{d+1}$ for each $ 1\leq \ell\leq d+1$.
  We denote by $\langle \cdot,\cdot \rangle$ the standard Euclidean inner product on $V$, making the monomials $\mbf{e}_I$ an orthonormal basis of $V$.
  Note that this basis consists of joint eigenvectors of the linear maps $\bigoplus_{\ell=1}^{d} \bigwedge^\ell g$, where $g\in G$ is any diagonal matrix.
  
  Let $v\in V_l \backslash\set{0}$ and write
    \begin{equation}\label{eq: v}
        v = \sum_{I \subset \set{1,\dots,d+1}} v_I \mbf{e}_I.
    \end{equation} 
  where the sum is over index sets of cardinality $l$.  
  Let $\mbf{x}= (\mbf{x}_1,\dots,\mbf{x}_d) \in  \R^{d}$.
  First, we note that $u(\bx)$ fixes $\mbf{e}_0$ and maps $\mbf{e}_i$ to $\mbf{e}_i+\bx_i \mbf{e}_0$ for $i=1,\dots,d$.
  This implies the following.
  \begin{equation*}
  	u(\bx)  \mbf{e}_I = 
      \begin{cases}
		\mbf{e}_I \qquad \qquad \qquad \qquad \qquad \qquad  0\in I,\\
        \mbf{e}_I + \sum_{i\in I} \pm \bx_i \mbf{e}_{(I\cup \set{0})\backslash\set{i}}, \qquad \textrm{otherwise},
      \end{cases}
  \end{equation*}
  where the sign depends on $I$.
  In particular, we get that
  \begin{equation}\label{eq: unipotent coordinates}
    	u(\bx) v = \sum_{\substack{I\subset \set{0,\dots,d} \\ 0\notin I}} v_I \mbf{e}_I
        + \sum_{\substack{I\subset \set{0,\dots,d} \\ 0\in I}} 
        \left( v_I + \sum_{i\notin I} \pm v_{(I\cup \set{i})\backslash\set{0}} 
        \bx_i \right) \mbf{e}_I.
    \end{equation}

  \subsection{Transversality of the expanding coordinates}

    Let $v\in V_l$ and write $v_I = \langle v,e_I\rangle$ for each index set $I$.
    For each index set $I$ containing $0$, let $\mc{L}_I(v)$ be the affine subspace defined by
    \begin{equation*}
    \mc{L}_I(v) = \set{\bx\in\R^d: \langle u(\bx)v,\mbf{e}_I\rangle = 0 }
    = \set{ \bx\in \R^d:   v_I + \sum_{i\notin I} \pm v_{(I\cup \set{i})\backslash\set{0}}  \bx_i =0}.
    \end{equation*}
    Note  that it is possible that $\mc{L}_I(v) = \emptyset$ or $\mc{L}_I(v) = \R^d$.
    Denote by $n_I$ the normal vector of $\mc{L}_I(v)$ given by
    \begin{equation}\label{eqn: normal vecs}
        n_I =  \sum_{i\notin I} \left(\langle u(e_i)v,\mbf{e}_I\rangle - v_I\right) e_i
        = \sum_{i\notin I} \pm v_{(I\cup \set{i})\backslash\set{0}} e_i,
    \end{equation}
    where $e_i$ denotes the standard basis of $\R^d$ and the choice of the signs is the same as in~\eqref{eq: unipotent coordinates}.
    Given an index set $J\subset \set{1,\dots,d}$ ($0\notin J$) of size $\ell = |J|$, define $\mc{J}(J) $ by
    \begin{equation}\label{eq: fancy J}
        \mc{J}(J) = \set{I\subset \set{0,\dots,d} : 
        0\in I,
        J = (I\cup \set{i})\backslash\set{0}, \textrm{ for some } i\notin I  }.
    \end{equation}
    We note if $v_J \neq 0$, then $I\in \mc{J}(J)$ if and only if $\pm v_J$ appears as a coordinate of $n_I$.
    
    The following elementary proposition is a key observation for our proof.
    \begin{proposition} [Transversality]
    \label{prop: transversality}
    Suppose $v\in V_\ell$ and let $I\subset \set{1,\dots,d}$ be an index set. Let $\ell = \# \mc{J}(I)$. Then,
    \[ \norm{ \bigwedge_{J \in \mc{J}(I) } n_J } \geqslant \left| v_I\right|^\ell.  \]
    \end{proposition}
    
    \begin{proof}
    Note that $0\notin I$ by definition and in particular $\mc{J}(I)\neq \emptyset$. 
    Consider the map $\vp: I\r \mc{J}(I) $ defined by
    \begin{equation*}
        \vp(j) = (I\cup\set{0})\setminus\set{j},
    \end{equation*}
    for every $j\in I$.
    One easily verifies that $\vp$ is a bijection.
    In particular, $\ell = |I|$.
    We claim that for each $j\in I$, the following holds:
    \begin{equation}\label{eq: transversality as a dichotomy}
        \left|  \langle n_{J}, e_j\rangle \right| =
        \begin{cases}
        | v_I| & J = \vp(j),\\
        0 & J \in \mc{J}(I), J\neq \vp(j).
        \end{cases}
    \end{equation}
    Indeed, when $J = \vp(j)$, then $ \left|  \langle n_{J}, e_j\rangle \right| = |v_I|$ by definition of $n_J$ in~\eqref{eqn: normal vecs}.
    Otherwise, if $J'\in \mc{J}(I)$ satisfies $J'\neq \vp(j)$, then we observe that $j\in J'$. Indeed, $J'=(I\cup\set{0})\setminus\set{j'}$ for some $j'\in I$, not equal to $j$.
    In this case, the definition of $n_{J'}$ in~\eqref{eqn: normal vecs} shows that the coefficient of $e_j$ is $0$.
    This completes the proof of~\eqref{eq: transversality as a dichotomy}.
    
    Equation~\eqref{eq: transversality as a dichotomy} implies that the $(\ell\times\ell)$-matrix $( \langle n_{J}, e_j\rangle)_{J\in \mc{J}(I), j\in I}$ is diagonal, up to a permutation of the rows, with $\pm v_I$ on the diagonal.
    It then follows from the definition of the inner product on $\bigwedge^\ell \R^{d+1}$ that 
    \[ \left|\left\langle \bigwedge_{J \in \mc{J}(I) } n_J, e_I \right\rangle\right|
        = \left|\mrm{det}(  ( \langle n_{J}, e_j\rangle)_{J\in \mc{J}(I), j\in I} ) \right|
    = |v_I|^\ell. \]
    This completes the proof.
    \end{proof}
    
    For an affine subspace $\mc{L} \subset \R^d$ and $\d >0$, recall that we denote by $\mc{L}^{(\d)}$ the open $\d-$neighborhood of $\mc{L}$. More precisely,
    \begin{equation*}
        \mc{L}^{(\d)} = \set{\mbf{x}\in \R^d: d(\mbf{x},\mc{L}) <\d  },
    \end{equation*}
    where $d(\cdot,\cdot)$ is the Euclidean distance.
    
    \begin{proposition}[From Transversality to Integrability]
    \label{prop: transverse implies integrable}
    Suppose $\mc{L}_1,\dots,\mc{L}_\ell$ are affine hyperplanes in $\R^d$ with $1\leq \ell \leq d-1$. Suppose that $n_k$ is a unit normal vector of $\mc{L}_k$ for each $1\leq k\leq \ell$. Assume further that 
    \[\norm{n_1 \wedge \cdots \wedge n_\ell } \geqslant \kappa, \]
    for some $\kappa>0$. Then, there exists $C \geqslant 1$, depending only on $\kappa$ and d, so that for all $\e >0$,
    \begin{equation} \label{eqn: intersection of nbhd in nbhd of intersection}
        \bigcap_{k=1}^{\ell} \mc{L}_k^{(\e)} \subseteq \left(\bigcap_{k=1}^{\ell} \mc{L}_k \right)^{(C \e)}.
    \end{equation}
    \end{proposition}
    
    \begin{proof}
    Let $P_\ell = \bigcap_{k=1}^{\ell}  \mc{L}_k$. First, we claim that $P_\ell \neq \emptyset$.
    For each $k$, let $v_k\in \R$ be such that $\Lcal_k = \set{\bx: \langle n_k,\bx\rangle = v_k}$.
    Let $A'$ denote the $\ell\times d$ matrix whose $k$\textsuperscript{th} row is $n_k$ and write $\mbf{v}\in \R^\ell$ for the column vector whose $k$\textsuperscript{th} entry is $v_k$.
    Then, $P_\ell$ is the set of solutions of the system $A' \bx = \mbf{v}$.
    The assumption that $n_1\wedge\cdots \wedge n_\ell\neq 0$ implies that $A'$ has full rank. In particular, this implies that $P_\ell$ is non-empty.
    
    By applying a translation, we may assume that $0\in P_\ell$. Let $\pi:\R^d \r \R^d/P_\ell \cong \R^\ell$ denote the canonical projection parallel to $P_\ell$. Let $L_k = \pi(\Lcal_k)$ for each $k$. We note that it suffices to show
    \begin{equation*}
        \bigcap_{k=1}^{\ell} \left( L_k\right)^{(\e)} \subseteq \left(\bigcap_{k=1}^{\ell} L_k \right)^{(C \e)},
    \end{equation*}
    for every $\e >0$. Note that $\pi(n_k)$ is a unit vector orthogonal to $L_k$ for each $k$. We continue to denote by $n_k$ the image of $n_k$ under $\pi$.
    
    Suppose $v \in  \bigcap_{k=1}^{\ell} \left( L_k\right)^{(\e)}$.
    Since each $L_k$ passes through the origin and each $n_k$ is a unit normal to $L_k$, the Euclidean distance of $v$ to $L_k$ is given by $d(v,L_k) = \left| \langle n_k,v\rangle \right|$.
    In particular, $\left|\langle n_k,v\rangle\right| <\e$.
    Note that $\bigcap_{k=1}^{\ell} L_k = \set{0}$. Thus, our task is to show that $\norm{v} \leqslant C\e$ for an appropriate uniform constant $C>0$.
    
    Denote by $A$ the square matrix whose rows are $n_k$.
    As $n_k$ are unit vectors, there is a constant $K >0$, depending only on $d$, such that $\norm{A} \leqslant K$.
    Moreover, we have that $\left|  \det A \right| = \norm{n_1 \wedge \cdots \wedge n_\ell} \geqslant \kappa$ and in particular that $A$ is invertible.
    Recall that if $g\in \mrm{SL}(\ell,\R)$, then $\norm{g^{-1}}\leqslant \norm{g}^{\ell}$.
    Thus, the following norm estimates follow:
    \begin{equation*}
        \norm{A^{-1}} \leqslant \left| \det A \right|^{-1-1/\ell} \norm{A}^\ell \leqslant \kappa^{-1-1/\ell} K^\ell.
    \end{equation*}
    Moreover, we have that 
    \[ \norm{A^{-1}}^{-1} \norm{v} \leqslant \norm{Av} \leqslant \max_k \left|\langle n_k,v\rangle\right| <\e. \]
    Together those two inequalities imply that $\norm{v} \leqslant C\e$, where $C = \kappa^{-1-1/\ell} K^\ell$.

    \end{proof}


\section{Decay Exponents, Transversality, and Expansion}
\label{section: linear expansion}

    This section is dedicated to proving estimates on the average rate of expansion of vectors in linear representations with respect to general measures, Proposition~\ref{propn: expansion in linear representations}.
    The main point of the result is the precise integrability exponent for the functions $\norm{g_{\t(\bx)} u(\mbf{x})v}^{-1}$.
    A key ingredient in the proof is the transversality result obtained in the previous section.
  Recall the definition of $\a_\ell(\mu)$ in~\ref{eqn: proj spec} and the parametrization of $g_t$ in~\eqref{linear forms g_t}.
  \begin{proposition} \label{propn: expansion in linear representations}
    Let $V_\ell =  \bigwedge^\ell \R^{d+1}$ for some $1\leqslant \ell < d+1$.
    Suppose $\mu$ is a compactly supported Borel probability measure on $\R^d$ and suppose $\a_\ell = \a_\ell (\mu)>0$.
    Let $0<\l \leqslant \a_\ell$ be given.
    Then, for all $0<\d <1$, there exists a constant $C = C(\d,\mu)\geq 1 $ so that the following holds: for every $v \in V_\ell\backslash\set{0}$, $\g\in \R$, and every measurable function $\t:\R^d\r (0,1)$ :
    \begin{equation*}
    	 \int \t(\bx)^\g \norm{g_{\t(\bx)} u(\mbf{x})v}^{-\d\l} \;d\mu(\mbf{x})
        \leqslant C \norm{v}^{-\d\l}
        \left(\int \t(\bx)^{p(\g+\k)}\;d\mu\right)^{1/p},
    \end{equation*}
    where
    \begin{equation*}
        \k=\frac{\d\l (d-\ell+1)}{d+1}, \qquad p = \frac{1+\d}{1-\d}.
    \end{equation*}
  \end{proposition}
  
  \begin{remark}
  It is worth noting that the constant $C$ in Proposition~\ref{propn: expansion in linear representations} has a delicate dependence on the measure $\mu$. In particular, it depends on the rate of convergence of the $\liminf$ in the definition of $\a_\ell(\mu)$.
  Moreover, in our proof, $C\r\infty$ as $\d\l\r \a_\ell$. In particular, it is not clear whether Proposition~\ref{propn: expansion in linear representations} holds with $\d \l=\a_\ell$ and with a finite constant $C$.
  \end{remark}
  
  Before the proof, we state two elementary lemmas which will be useful for us.
   The first lemma is immediate from the definition.
    \begin{lemma} \label{lem: meaning of alpha_ell}
  Suppose $\mu$ is a Borel measure on $\R^d$ such that $\a_\ell(\mu)$ exists for some $1\leq \ell \leq d$.
  Then, for every $\eta>0$, there exists $0<\e_0<1$ so that for all $0<\e\leq \e_0$
  \[ \sup\set{ \mu\left( \Lcal^{(\e)}   \right): \Lcal \textrm{ is a proper affine subspace of dimension } d-\ell } \leqslant \e^{\a_\ell -\eta}.  \]
  \end{lemma}
  
  The next lemma is a simple application of Fubini's Theorem.
  \begin{lemma} \label{lemma: cavalieri}
  Let $\mu$ be a Borel measure and $f$ a non-negative Borel function on a separable metric space $X$. Then,
  \[ \int_X f\;d\mu = \int_0^\infty \mu\left(\set{x\in X: f(x)\geqslant t}\right) \;dt \]
  \end{lemma}
  
  \begin{proof}[Proof of Proposition~\ref{propn: expansion in linear representations}]

    Let $0\neq v\in V_\ell $ and without loss of generality assume $\norm{v} = 1$.
    Denote by $\mc{B}$ the following collection of vectors
    \begin{equation} \label{defn: basis of V+}
    	\mc{B} = \set{ \mbf{e}_{I}: 0\in I\subset \set{0,\dots,d}   }.
    \end{equation}
    We use $V^+$ to denote the linear span of $\mc{B}$ and we let $V_\ell^+ = V^+ \cap V_\ell$.
    Then, $V^+ $ is the expanding subspace corresponding to $g_t$.
    Denote by $\pi_+: V_\ell \r V_\ell^+ $ be the canonical projection.
    Denote by $v^+$ the image of $v$ under $\pi_+$ and let $v^- = v- v^+$. Define $K$ as follows:
    \[ K = \max\set{1, \sup_{\bx \in \mrm{supp}(\mu)} \norm{\bx}}. \]
    Then, $K < \infty$ since $\mu$ is compactly supported.
    Observe that for $t>0$, we have that
    \begin{equation} \label{eqn: expansion by g_t}
    	\norm{g_{t} u(\bx)v} \geqslant \norm{g_{t} \pi_+(u(\bx)v) }
        = t^{-\frac{d-\ell+1}{d+1}} \norm{\pi_+(u(\bx))v)}
    \end{equation}
    Fix some $\d \in (0,1)$ and let $0<\l\leqslant \a_\ell$ be given.
    We wish to apply H\"older's inequality.
    To this end, let
    \begin{equation*}
        \b=\d\l,\qquad q = \frac{1+\d}{2\d}, \qquad p= \frac{1+\d}{1-\d}.
    \end{equation*}
    Note that $p$ and $q$ are H\"older conjugates.
    Then, H\"older's inequality and~\eqref{eqn: expansion by g_t} imply
    \begin{align}\label{eq: apply Holder}
        \int \t(\bx)^\g \norm{g_{\t(\bx)} u(\bx)v}^{-\b}  \;d\mu &\leqslant
        \int \t(\bx)^{\g+\kappa} \norm{\pi_+(u(\bx))v)}^{-\b}  \;d\mu \nonumber \\
        &\leqslant \left( \int \t(\bx)^{p(\g+\k)} \;d\mu \right)^{1/p}
        \left( \int \norm{\pi_+(u(\bx))v)}^{-q\b}  \;d\mu \right)^{1/q}
    \end{align}
     Hence, it remains to show that the integral of $\norm{\pi_+(u(\bx)v)}^{-q\b}$ is uniformly bounded.
     We split the analysis into two cases based on the size of $v^-$. Recall the expression of $u(\bx)v$ in standard coordinates given in~\eqref{eq: unipotent coordinates}.
    \setcounter{case}{0}
    \begin{case}
    $\norm{v^-} \leqslant 1/3K$. Then, since $\norm{v} = 1$, there is an index set $I$ containing $0$ so that $|v_I| \geqslant 2/3 $. It follows that for each $\bx\in \mrm{supp} (\mu)$
     \begin{align*} 
     \norm{\pi_+(u(\bx))v)} 
        \geqslant 
         \left| v_I + \sum_{i\notin I} \pm v_{(I\cup \set{i})\backslash\set{0}} 
        \bx_i \right| \geqslant
        |v_I| - \norm{v^-} \norm{x}
        \geqslant \frac{2}{3} - \frac{1}{3} = \frac{1}{3}. 
    \end{align*}
    This implies that 
    $\left( \int \norm{\pi_+(u(\bx))v)}^{-q\b}  \;d\mu \right)^{1/q} \leqslant 3^{-\b}$ and  concludes the proof in this case.
    \end{case}

    \begin{case} $\norm{v^-}> 1/3K$. Then, there exists some index set $J$, not containing $0$ so that $|v_J| > 1/3K$. Define $\mc{J}(J)$ by
    \begin{equation*}
         \mc{J}(J) = \set{I\subset \set{0,\dots,d} : 
        0\in I,
        J = (I\cup \set{i})\backslash\set{0}, \textrm{ for some } i\notin I  }.
    \end{equation*}
    Let $I\in \mc{J}(J)$ and define $n_I$ by
    \[n_I =  \sum_{i\notin I} \pm v_{(I\cup \set{i})\backslash\set{0}} e_i,\]
    where the choice of signs is as in~\eqref{eq: unipotent coordinates}.
    Note that $\pm v_J$ appears as a coordinate of $n_I$.
    In particular, $\norm{n_I} \geq |v_J|>1/3K \neq 0$.
    Moreover, Proposition~\ref{prop: transversality} then shows that $
         \norm{\bigwedge_{I\in \mc{J}(J)} n_I } \geqslant (3K)^{-\ell}$.
    
    Consider the hyperplane
    $\Lcal_I = \set{\bx: \langle n_I,\bx\rangle = -v_I}$. Then, a simple calculation shows
    \[ d(\bx,\Lcal_I) = \frac{\left|\langle n_I,\bx\rangle + v_I \right|}{\norm{n_I}}, \]
    where $d(\cdot,\cdot)$ denotes Euclidean distance.
    For each $\e>0$, we define the set $E(v,\e)$ as follows:
    \begin{equation*}
     E(v,\e) = \set{\bx \in \mrm{supp}(\mu) : \norm{\pi_+(u(\bx) v)} \leqslant \e}.
    \end{equation*}
    Suppose that $\bx\in E(v,\e)$. Then, for each index set $I$ containing $0$, the following holds.
    \begin{equation*}
    	\left|\langle n_I,\bx\rangle + v_I \right| =
    	 \left| v_I + \sum_{i\notin I} \pm v_{(I\cup \set{i})\backslash\set{0}} 
        \bx_i \right| \leqslant \e.
    \end{equation*} 
     It follows that $\bx \in \bigcap_{I\in \mc{J}(J)} \Lcal_I^{(3K\e)} $.
     For simplicity, denote by $\Lcal(J):= \bigcap_{I\in \mc{J}(J)} \Lcal_I$.
     Hence, Proposition~\ref{prop: transverse implies integrable}, applied with $\k = (3K)^{-\ell}$, implies
     \begin{equation} \label{eqn: bad set is intersection of planes}
         E(v,\e) \subseteq \left( \Lcal(J) \right)^{(C_1\e)},
     \end{equation}
     for some constant $C_1\geq 1$ depending only on $d$ and $K$.
    Applying Lemma~\ref{lemma: cavalieri}, we obtain
    \begin{align}\label{eq: apply cavalieri}
    \int \norm{\pi_+(u(\bx))v)}^{-q\b}  \;d\mu
        =   \int_0^\infty
        \mu\left(E\left(v,r^{-1/q\b}\right)\right) \;dr.
    \end{align}
    
    The next ingredient is to apply Lemma~\ref{lem: meaning of alpha_ell}. We observe that $q\b=(1+\d)\l/2 $ is strictly smaller than $\a_\ell$, since $0<\d<1$ and $\b=\d\l$.
    Let $\eta = (\a_\ell-q\b)/2$. Then, $0<\eta < \a_\ell$.
    Applying Lemma~\ref{lem: meaning of alpha_ell} with this $\eta$, we get that there exists $\e_0>0$ so that for all $0<\e< \e_0$, 
    \begin{equation*}
        \mu \left( \left( \Lcal(J) \right)^{(\e)} \right) \leqslant \e^{\a_\ell -\eta}.
    \end{equation*}
    Here, we use the fact that $\Lcal(J)$ is an affine subspace of dimension $d-\ell$ since it is the intersection of $\#(\mc{J}(I))=\ell$ transverse affine hyperplanes by Proposition~\ref{prop: transversality}.
    Let $R>1$ be sufficiently large such that $C_1 R^{-1/q\b} < \e_0$.
    Note that the choice of $R$ here depends only on $\d$ and $\mu$.
    Since $\mu$ is a probability measure, it follows that
   \begin{equation}\label{eq: before decay}
       \int_0^R \mu\left(E\left(v,r^{-1/q\b}\right)\right) \;dr \leqslant R.
   \end{equation}
    Moreover, by~\eqref{eqn: bad set is intersection of planes}, we obtain
    \begin{align}\label{eq: after decay}
       \int_R^\infty \mu\left(E\left(v,r^{\frac{-1}{ q\b}}\right)\right)\;dr
       &\leqslant \int_R^\infty \mu \left( \left( \Lcal(J) \right)^{\left(C_1 r^{\frac{-1}{ q\b}}\right)} \right) \;dr 
       \leqslant (C_1)^{\a_\ell-\eta} \int_R^\infty r^{-\frac{\a_\ell -\eta}{q\b}} \;dr =: C_2.
    \end{align}
    Finally, note that our choice of $\eta$ implies that
    $ \frac{\a_\ell -\eta}{q\b} = 1+ \frac{\eta}{q\b} >1$.
    In particular, we have that $C_2<\infty$.
    Combining~\eqref{eq: apply Holder},~\eqref{eq: apply cavalieri},~\eqref{eq: before decay}, and~\eqref{eq: after decay} concludes the proof.
    \end{case}
  \end{proof}

    
    \section{Height Functions and Integral Inequalities}
\label{section: integ ineq}

	In this section, we construct a proper function on the space of unimodular lattices.
	Using the results in the previous section, we verify that this function satisfies the properties listed in Definition~\ref{defn: height functions}.
	The key idea that allows converting integral estimates in linear representations into integral estimates over the space of lattices is the use of the so-called systems of integral inequalities which first appeared in~\cite{EskinMargulisMozes}.
	The main result of this section is Theorem~\ref{thm: constraction in X}.

    \subsection{Preliminary notation}
Throughout this section, we set
	\begin{equation*}
	    G = \mrm{SL}(d+1,\R),\qquad \G= \mrm{SL}(d+1,\Z),\qquad X = G/\G.
	\end{equation*}
	
	 In view of Proposition~\ref{propn: expansion in linear representations}, the results of this section apply to general Borel measures $\mu$ and time parametrizations $\t$ on $\R^d$. 
	 However, for our application, we restrict ourselves to the setting of Theorem~\ref{main thm}.
	 We fix a finite set $\L$ and an irreducible IFS $\Fcal=\set{\rho_iO_i+b_i: i\in \L}$ on $\R^d$ satisfying the open set condition with limit set $\Kcal$. We denote by $s$ the Hausdorff dimension of $\Kcal$ and $\mu$ the unique self-similar probability measure supported on $\Kcal$ for the canonical probability vector $(\rho_i^s)_i$. Recall that $\mu$ in this case coincides with the normalized restriction of the $s$-dimensional Hausdorff measure to $\Kcal$.

    We denote by $V$ and $V_\ell$ the following vector spaces, endowed with the standard representations of $G$,
    \begin{equation*}
        V_\ell = \bigwedge\nolimits^\ell \R^{d+1},\qquad   V =  \bigoplus_{\ell =1}^{d} V_\ell.
    \end{equation*}
    Motivated by Proposition~\ref{propn: expansion in linear representations}, we define exponents of the form $\b_\ell$, for $1\leq \ell \leq d$ as follows:
    \begin{equation}\label{delta_lamda}
        \varpi = \min_{1\leq \ell \leq d} \a_\ell(\mu)(d-\ell+1), \qquad 
        \b_\ell:= \frac{d-\ell+1}{\varpi},
    \end{equation}
    where $\a_\ell(\mu)$ was defined in~\eqref{eqn: proj spec}.
    Corollary~\ref{cor: alpha >0} shows that $\a_\ell(\mu) >0 $ since the IFS $\Fcal$ is irreducible.
    
    
    The space $X = G/\G$ is identified with the space of unimodular lattices in $\R^{d+1}$ via the map
    $g \mrm{SL}(d+1,\Z) \mapsto g\Z^{d+1}$.
    For $x \in X$, let $P(x)$ denote the set of all \textbf{primitive} subgroups of the lattice $x$.
    Recall that a subgroup $L$ of a lattice $x$ in $\R^{d+1}$ is primitive if $L = \Z^{d+1} \cap \mrm{span}_\R (L) $, where $\mrm{span}_\R (L)$ is the $\R$-span of any $\Z$-basis of $L$.
    We say a monomial $v_1 \wedge \cdots \wedge v_\ell \in V_\ell $ is $x$-\textbf{integral} if the abelian subgroup of $\R^{d+1}$ generated by
    $\set{v_1,\dots,v_\ell}$ is primitive, i.e. belongs to $P(x)$.

    For every $0<\ell <d+1$, we define a function $\vp_{\ell}:X\r[1,\infty)$ as follows: ,
    \begin{align} \label{defn: phi}
    \vp_{\ell}(x) =  \max \set{\norm{v}^{-1}: v\in V_\ell \text{ is an } x\text{-integral monomial}  }.
    \end{align}
    For $\ell=0,d+1$, set $\vp_\ell \equiv 1$.
    For a compact set $Q\subset G$, define
    \begin{equation} \label{eqn: omega Q}
          \norm{Q} = \sup_{g\in Q} \max\left( \norm{g}, \norm{ g^{-1}}  \right)^{d+1}.
          \end{equation}
        where $\norm{\cdot}$ is the operator norm induced by the Euclidean norm on $V $.
        It follows from the definitions that for every $h\in Q$, $\e>0$ and every $v\in \bigwedge^\ast \R^{d+1}$,
    \begin{equation}\label{eq: log smoothness}
        \norm{Q}^{-1} \vp_{\ell}(x) \leqslant \vp_{\ell}(hx) \leqslant \norm{Q} \vp_\ell(x).
    \end{equation}

    \subsection{The Contraction Hypothesis on X}

    We recall the following Lemma from~\cite{EskinMargulisMozes} which underlies the main property of Margulis functions we prove later in the section.
    \begin{lemma}[Lemma 5.6 in~\cite{EskinMargulisMozes}] \label{lemma: mini mother inequality}
    Let $x\in X$ and let $\L_1, \L_2 \in P(x)$.
    Then,
      \[ \norm{\L_1} \norm{\L_2} \geqslant \norm{\L_1 \cap \L_2} \norm{\L_1 + \L_2}.  \]
    \end{lemma}

    The following proposition establishes the fundamental property of Margulis functions obtained in~\cite{EskinMargulisMozes}. It is obtained via the method of integral inequalities first introduced in~\cite{EskinMargulisMozes}
    \begin{proposition}\label{prop: combine complexities}
    For every $0< \varrho  <1$, there exists a constant $C$, depending only on $\varrho$ and $\mu$, such that 
    for every $k\in \N$, there exists $\w_0=\w_0(k,\mu)\geqslant 1$, so that for all $x_0\in X$ and and all $\g\in \R$
        \begin{align*}
        	 \int \rho(\bx,k)^\g \vp_{\ell}^{\varrho/\b_\ell}(g_{\rho(\bx,k)} u(\mbf{x})x_0) \;d\mu(\mbf{x})
            \leqslant C  &\vp_\ell^{\varrho/\b_\ell}(x_0)\left(\int  \rho(\bx,k)^{p(\g+\k)}\;d\mu(\mbf{x})\right)^{1/p}\\
            &+ \omega_0^{2\varrho/\b_\ell} \max_{1 \leqslant j \leqslant \min\set{\ell,d+1-\ell}} \left(\sqrt{\vp_{\ell+j}(x)\vp_{\ell-j}(x) }\right)^{\varrho/\b_\ell} 
            ,
        \end{align*}
        where
    \begin{equation}\label{eq: conclusion of combine complexities prop}
        \k=\frac{\varrho\varpi}{d+1}, \qquad p = \frac{1+\varrho}{1-\varrho},
    \end{equation}
    and $\varpi$ was defined in~\eqref{delta_lamda}.
    \end{proposition}
    
    \begin{proof}
     Let $0<\varrho <1$ be given.
    Fix $k\in \N$ and let
    \begin{equation*}
        Q = \set{g_{\rho(\bx,k)}u(\bx): \bx\in \mrm{supp}(\mu) }.
    \end{equation*}
    Let $\w=\norm{Q}$ as defined in~\eqref{eqn: omega Q} for $Q$ as above.
    Following~\cite{EskinMargulisMozes}, let $\Psi_\ell$ denote the finite subset of $P(x_0)$ of rank $\ell$ subgroups $L$ 
    of $x_0$ satisfying 
    \[ \norm{L}^{-1} \geq \omega^{-1} \vp_\ell(x_0).  \]
    The finiteness of $\Psi$ follows from the discreteness of the lattice $x_0$.
    Suppose that $\Psi_\ell$ consists of a single element and denote it by $\L_\ell$.
    In this case, by~\eqref{eq: log smoothness}, we see that for all $\bx\in \mrm{supp}(\mu)$, 
    \[ \vp_\ell( g_{\rho(\bx,k)} u(\mbf{x})x_0) = \norm{g_{\rho(\bx,k)} u(\mbf{x}) \L_\ell}^{-1}. \]
    Observe further that the definition of the exponents $\b_\ell$ implies
    $ \l = \frac{1}{\b_\ell}\leqslant \a_\ell(\mu)$,
    for each $\ell$.
    In particular, upon applying Proposition~\ref{propn: expansion in linear representations} with $v=\L_\ell$, $\l=1/\b_\ell $, $\t(\cdot)=\rho(\cdot,k)$, and $\d=\varrho$, we obtain
    \begin{align} \label{eqn: integral estimate for high points}
    	\int \rho(\bx,k)^\g \vp_{\ell}^{\varrho/\b_\ell}(g_{\rho(\bx,k)} u(\mbf{x})x_0) \;d\mu
    	&=  \int \rho(\bx,k)^\g \norm{g_{\rho(\bx,k)} u(\mbf{x}) \L_\ell}^{-\varrho/\b_\ell} \;d\mu \nonumber \\
        &\leqslant C \vp_\ell^{\varrho/\b_\ell}(x_0) \left(\int  \rho(\bx,k)^{p(\g+\k)}\;d\mu(\mbf{x})\right)^{1/p},
    \end{align}
    where $C \geq 1$ is the constant in Proposition~\ref{propn: expansion in linear representations} and $p$ and $\k$ are as in~\eqref{eq: conclusion of combine complexities prop}.
    
    Alternatively, suppose the cardinality of $\Psi_\ell$ is at least $2$.
    Let $\L_\ell\in \Psi_\ell$ be such that $\vp_{\ell}(x_0) = \norm{\L_\ell}^{-1}$,
    and let $L \neq \L_\ell$ be another element of $\Psi_\ell$.
    Then, the group $\L_\ell + L$ has rank $\ell+j$ for some $j >0$.
    Moreover, by definition of $\Psi_\ell$, 
    $\norm{\L_\ell}^{-1} \leqslant  \w\norm{L}^{-1}$.
    Hence, in view of Lemma~\ref{lemma: mini mother inequality},
    for all $\bx\in\mrm{supp}(\mu)$,
    \begin{align*}
    \vp_{\ell}(g_{\rho(\bx,k)} u(\mbf{x})x_0) \leqslant \omega \vp_{\ell}(x_0) = \frac{\omega}{\norm{\L_\ell}} 
    	&\leqslant \frac{\omega^{2}}{\sqrt{\norm{\L_\ell} \norm{L}} } 
        \leqslant \frac{\omega^{2}}{ \sqrt{ \norm{\L_\ell+L}\norm{\L_\ell\cap L}  } }\\
        &\leqslant \omega^2 \max_{1\leqslant j \leqslant \min\set{\ell,d+1-\ell}} \sqrt{\vp_{\ell+j}(x_0)\vp_{\ell-j}(x_0) }.        
    \end{align*}
    Combining this estimate with~\eqref{eqn: integral estimate for high points}, we get the desired conclusion with 
    \begin{equation*}
        \w_0 = \w \max\set{\rho(\bx,k)^\g: \bx\in\mrm{supp}(\mu)  }.
    \end{equation*}
    \end{proof}

    Given $\e >0$ and $0<\varrho <1$, we define the function $f_{\e,\varrho}: X \r \R_+$ by 
    \begin{align} \label{defn: Margulis fn}
    	f_{\e,\varrho}(x) = 2 + \sum_{\ell=1}^{d} \e^{\ell}  \vp_{\ell}^{\varrho/\b_\ell }(x).
    \end{align}
    Define $\b$ as follows
    \begin{equation}\label{eq: beta}
        \b = \min_{1\leq \ell \leq d} \frac{\a_\ell(\mu)(d-\ell+1) }{d+1} = \frac{\varpi}{d+1}.
    \end{equation}
    
    We are now ready to verify the contraction hypothesis in our context.
    The idea of the deduction of the following result from Proposition~\ref{prop: combine complexities} is due to Margulis and first appeared in~\cite{EskinMargulisMozes}.
    In that article, all the exponents of $\vp_\ell$ were the same.
    Since our exponents $1/\b_\ell$ are distinct, this introduces a complication which we address in the proof.
    The reader may wish to consult variants of this idea in~\cite[Proposition 4.1]{KKLM} and~\cite[Claim 5.9]{BQ-RandomWalkRecurrence}.

    \begin{theorem}\label{thm: constraction in X}
    For every $0< \varrho  <1$, there exists a constant $C_1$, depending only on $\varrho$ and $\mu$, such that 
    for every $k\in \N$, there exists $b=b(k,\mu)\geqslant 1$ and $0<\e_0<1$, depending only on $k$ and $\mu$, so that for all $x_0\in X$ and and all $\g\in \R$ satisfying $\varrho\b - \frac{1-\varrho}{1+\varrho} \leq \g \leq \varrho\b$,
    
    \begin{equation*}
        \int \rho(\bx,k)^{-\g} f_{\e_0,\varrho}(g_{\rho(\bx,k)} u(\mbf{x})x_0)\;d\mu(\bx) \leqslant C_1 \left(\int  \rho(\bx,k)\;d\mu(\mbf{x})\right)^{\varrho\b-\g}
        f_{\e_0,\varrho}(x_0) + b.
    \end{equation*}
    
    \end{theorem}
    
    \begin{remark}
    The proof of Theorem~\ref{thm: constraction in X} will show that the lower bound restriction on $\g$ is only for aesthetic reasons.
    \end{remark}
    
    We first prove the following elementary but crucial lemma.
    \begin{lemma} \label{lem: concavity magic}
    For all natural numbers $0<j,\ell <d+1$, we have
    \begin{equation}
        \ell - \frac{\b_{\ell-j}(\ell-j) + \b_{\ell+j}(\ell+j)}{2\b_\ell} \geqslant \frac{1}{d}.
    \end{equation}
    \end{lemma}
    \begin{proof}
    We note that the function $q(x) = x(d+1-x)$ satisfies the following concavity property:
    \begin{equation*}
        \frac{q(\ell-j) + q(\ell+j)}{2} = q(\ell) -j^2.
    \end{equation*}
    It follows that
    \begin{equation*}
       \ell- \frac{\b_{\ell-j}(\ell-j) + \b_{\ell+j}(\ell+j)}{2\b_\ell} =\ell- \frac{q(\ell)-j^2}{d-\ell+1} = \frac{j^2}{d-\ell+1}.
    \end{equation*}
    Since $j^2 \geqslant 1$ and $d+1-\ell \leqslant d$, the lemma follows.
    \end{proof}
    
    \begin{proof}[Proof of Theorem~\ref{thm: constraction in X}]
    Fix $k\in \N$ and $0<\varrho <1$. Let $C,\w_0 \geq 1$ be the constants provided by Proposition~\ref{prop: combine complexities}.
    To simplify notation, let
    \begin{equation*}
        a = C \left(\int  \rho(\bx,k)^{p(\k-\g)}\;d\mu(\mbf{x})\right)^{1/p},
    \end{equation*}
    where $p$ and $\k$ are as in~\eqref{eq: conclusion of combine complexities prop}.
    Note that $a$ depends on $k$.
    Let $0<\e_0<1$ be a constant to be determined.
    Suppose $x_0 \in X$.
    It follows from Proposition~\ref{prop: combine complexities} that
    \begin{align} \label{eq: apply comb comp prop}
        \int \rho(\bx,k)^{-\g} &f_{\e_0,\varrho}(g_{\rho(\bx,k)} u(\mbf{x})x_0)\;d\mu =
        2 + \sum_{\ell=1}^d \e_0^{\ell} \int \rho(\bx,k)^{-\g} \vp_{\ell}^{\varrho/\b_\ell}(g_{\rho(\bx,k)} u(\mbf{x})x_0) \;d\mu  \leqslant \nonumber\\
        &\leqslant 2 + a \sum_{\ell=1}^d \e_0^{\ell} \vp_{\ell}^{\varrho/\b_\ell}(x_0)
        + \w_0^{2\varpi}   \sum_{\ell=1}^d \e_0^{\ell} 
        \max_{1 \leqslant j \leqslant \min\set{\ell,d+1-\ell}} \left(\sqrt{\vp_{\ell+j}(x_0)\vp_{\ell-j}(x_0) }\right)^{\varrho/\b_\ell}\nonumber\\
        &=  a f_{\e_0,\varrho}(x_0) + 2(1-a) +  
        \w_0^{2\varpi}   \sum_{\ell=1}^d \e_0^{\ell} 
        \max_{1 \leqslant j \leqslant \min\set{\ell,d+1-\ell}} \left(\sqrt{\vp_{\ell+j}(x_0)\vp_{\ell-j}(x_0) }\right)^{\varrho/\b_\ell},
    \end{align}
    where we used the fact that $\w_0^{2\varrho/\b_\ell} \leqslant \w_0^{2\varpi}$.
    We observe that the exponents $\b_\ell$ satisfy the following relation:
    \begin{equation*}
        \b_{\ell-j} + \b_{\ell+j} = 2\b_\ell,
    \end{equation*}
    for all $0<j,\ell <d+1$.
    In particular, this implies
    \begin{equation}\label{eq: Margulis magic}
        \left(\sqrt{\vp_{\ell+j}(x_0)\vp_{\ell-j}(x_0) }\right)^{\varrho/\b_\ell}
        \leqslant \e_0^{\frac{-\left(\b_{\ell-j}(\ell-j) + \b_{\ell+j}(\ell+j)\right)}{2\b_\ell}} f_{\e_0,\varrho}(x_0).
    \end{equation}
    Moreover, Lemma~\ref{lem: concavity magic} shows that for all $0<j,\ell <d+1$:
    \begin{equation}\label{eq: concavity of delta_ell}
        \ell - \frac{\b_{\ell-j}(\ell-j) + \b_{\ell+j}(\ell+j)}{2\b_\ell} \geqslant \frac{1}{d}.
    \end{equation}
    Applying the estimates~\eqref{eq: Margulis magic} and~\eqref{eq: concavity of delta_ell} to the last sum in~\eqref{eq: apply comb comp prop}, using the fact that $\e_0<1$, yields
    \begin{align*}
     \int \rho(\bx,k)^{-\g} &f_{\e_0,\varrho}(g_{\rho(\bx,k)} u(\mbf{x})x_0)\;d\mu
     \leqslant a f_{\e_0,\varrho}(x_0) + 2(1-a) +  d \e_0^{1/d} \w_0^{2\varpi}
        f_{\e_0,\varrho}(x_0).
    \end{align*}
    Choosing $\e_0^{1/d} = \frac{a}{\w_0^{2\varpi}d}$, $b=2(1-a)$, and $C_1 = 2C$, we obtain
    \begin{equation*}
        \int_\Kcal \rho(\bx,k)^{-\g} f_{\e_0,\varrho}(g_{\rho(\bx,k)} u(\mbf{x})x_0)\;d\mu(\bx) \leqslant C_1 \left(\int_\Kcal  \rho(\bx,k)^{p(\k-\g)}\;d\mu(\mbf{x})\right)^{1/p}
        f_{\e_0,\varrho}(x_0) + b.
    \end{equation*}
    Note that $b$ also depends on $k$.
    Since $\varrho\b - \frac{1}{p} \leq \g \leq \varrho\b=\k$, Jensen's inequality implies that
    \begin{equation*}
        \left(\int_\Kcal  \rho(\bx,k)^{p(\k-\g)}\;d\mu(\mbf{x})\right)^{1/p} \leqslant
         \left(\int_\Kcal  \rho(\bx,k)\;d\mu(\bx)\right)^{\varrho\b-\g},
    \end{equation*}
    thus completing the proof.
    \end{proof}
    
    
    To demonstrate the power of the method of integral inequalities, we state a consequence of the above analysis, which can be obtained with a little more work.
    We do not need this statement for our purposes, so we omit the proof and refer the interested reader to~\cite[Claim 5.9]{BQ-RandomWalkRecurrence} for the proof of a similar statement.
    Given $\e >0$, consider the function $F_\e: X \r \R_+$ defined by 
    \begin{align} 
    	F_\e(x) = \max \e^{\varpi\ell }  \norm{v}^{\frac{-1}{\b_{\ell}}},
    \end{align}
    where the maximum is taken over all $x$-integral monomials $v\in V_\ell$ and all $0<\ell<d+1$.
    
    \begin{proposition}
\label{prop: isolation ala BQ}
 For every compact set $Q\subset G$, there exist constants $C_1 \geqslant 1$ and $\e_0 >0$, depending only on $Q$, such that for every $x\in X$, whenever $F_{\e_0}(x) > C_1$, the set $\Psi(x)$ of $x$-integral monomials $v$ satisfying
        \begin{equation} \label{eqn: defn of psi}
        	 \e_0^{\varpi\ell }  \norm{v}^{\frac{-1}{\b_{\ell}}} \geqslant F_{\e_0}(x)/\norm{Q}^{2\varpi},
        \end{equation}
        contains at most one primitive vector up to a sign in each $V_\ell$ with $0<\ell<d+1$.
        
\end{proposition}
    

    \subsection{Proof of Theorem~\ref{main dynamics thm} and its corollaries}
    Let $0<\varrho < 1$ be given.
    For every $k$, let $\e_0(k)$ be the constant provided by Theorem~\ref{thm: constraction in X}. Consider the collection of height functions:
    \begin{equation*}
        \Fsc=\set{f_k=f_{\e_0(k),\varrho}:k\in\N}. 
    \end{equation*}
    Theorem~\ref{thm: constraction in X} shows that the action of $G$ on $X=G/\G$ satisfies the $(\Fcal,\Fsc,\b,\g_0)$-contraction hypothesis with $\b$ and $\g_0$ given by:
    \begin{equation*}
        \b=\frac{\varrho \varpi}{d+1}, \qquad \g_0 = \frac{\varrho \varpi}{d+1} - \frac{1-\varrho}{1+\varrho}, 
    \end{equation*}
    where $\varpi$ was defined in~\eqref{delta_lamda}.
    Indeed, the contraction property $(3)$ of Definition~\ref{defn: height functions} follows from Theorem~\ref{thm: constraction in X}, where $T$ can be chosen as follows:
    \begin{equation*}
        T = \frac{b}{C_1 \left(\int  \rho(\bx,k)\;d\mu(\mbf{x})\right)^{\varrho\b-\g}}.
    \end{equation*}
    The remaining properties follow directly from the definition of $f_{\e_0,\varrho}$.
    
    Hence, Theorem~\ref{thrm: Hdim and non-divergence} applies and shows that the dimension of the set in the conclusion of Theorem~\ref{main dynamics thm} is at most $s-\varrho \varpi/(d+1)$.
    Since $0<\varrho < 1$ was arbitrary, this completes the proof.
    
    Theorem~\ref{main thm} follows from Theorem~\ref{main dynamics thm} by taking $x_0$ to be the identity coset and using Dani's correspondence.
    Corollary~\ref{cantor prod cor} follows from Theorem~\ref{main thm} along with Lemma~\ref{lem: alpha_1 for Cantor product} showing that $\a_1(\mu)=\log 2/\log 3$ in this case.
    Finally, in the setting of Corollary~\ref{homogeneous cor}, it is shown in~\cite[Theorem 8.2]{Shmerkin} that $\dim_\infty(\pi_\th\mu) = \min\set{\dim_H(\Kcal), 1}$ for every $\th\in [0,2\pi)$, where $\pi_\th \mu$ is the projection of $\mu$ in direction $\th$ and $\dim_\infty(\pi_\th\mu)$ is the Frostman exponent of $\pi_\th \mu$ defined in~\eqref{def: frostman}.
    Hence, applying Theorem~\ref{thm: alpha equal inf}, we get that $\a_1(\mu)=\min\set{\dim_H(\Kcal), 1}$.
    Corollary~\ref{homogeneous cor} then follows from Theorem~\ref{main thm}.
    
    
    \section{Fractals and the Teichm\"uller Flow} \label{sec: Teich}
    The goal of this section is to prove Theorem~\ref{teich thm}.
    The height functions needed to apply Theorem~\ref{thrm: Hdim and non-divergence} in this context were constructed by Eskin and Masur in~\cite{EskinMasur}.
    In using this construction, we apply Proposition~\ref{propn: expansion in linear representations} with $G=\mrm{SL}(2,\R)$. We remark that the proof of Proposition~\ref{propn: expansion in linear representations} simplifies in this case and does not require the results of Section~\ref{section: transverse}.

    \subsection{Background and Definitions}\label{sec: teich defs}
    For background on Teichm\"uller dynamics and translation flows, the reader is referred to~\cite{ForniMatheus} for an excellent survey.
    Suppose $S$ is a compact oriented surface of genus $g\geqslant 1$.
    An abelian differential on $S$ is an isotopy class of pairs 
$(M,\omega)$, where $M$ is a Riemann surface structure on $S$ and $\omega$ is a holomorphic $1$-form.
    Then, $\w$ induces a (possibly singular) flat metric on $S$.
    A unit area abelian differential is one in which $M$ has area $1$ in the induced flat metric.
    A saddle connection of $(M,\w)$ is a flat geodesic segment joining zeros of $\w$. 
    
    If $\Sigma\subset S$ denotes the set of zeros of $\w$, then $S\setminus \Sigma$ admits an atlas of charts to the complex plane so that all transition maps are given by translations. In these coordinates, $\w$ is given by the pull-back of the canonical holomorphic $1$-form $dz$ on $\C$. Moreover, the ``vertical" vector field (parallel to the imaginary axis) on $\C$ induces a well-defined vector field on $S\setminus \Sigma$. The flow defined by this vector field is referred to as the vertical flow on $(M,\w)$. The induced area measure by $\w$ is invariant for this flow. The flow is said to be uniquely ergodic if this is the only invariant measure.
    
    Let $\alpha=(\alpha_1, \dots, \alpha_n)$ be an integral partition of $2g-2$, i.e. $\a_i\in \N$ and $\sum \a_i = 2g-2$.
    By a stratum of abelian differentials of order $\a$, we mean the space of unit area abelian differentials on $S$ whose zeroes have multiplicities $\alpha_1, \dots, \alpha_n$.
    Strata of abelian differentials are non-compact. This can be seen by taking a sequence of abelian differentials in which the systole on the associated Riemann surface tends to $0$. If the integral partition definining the stratum contains two distinct elements, then a sequence of abelian differentials may ``diverge" if the distance between two distinct zeros tends to $0$.
    
    There are local coordinates on a stratum into $\C^N$ for appropriate $N$, called period coordinates (e.g., see \cite[Section 2.3]{ForniMatheus} for details), such that all changes of coordinates are given by affine maps. In period coordinates, $\mathrm{SL}(2,\mathbb R)$ acts naturally on each copy of $\mathbb C$.
    The Teichm\"uller geodesic flow is the flow induced by the action of the diagonal group in $\mrm{SL}(2,\R)$. The behavior of the orbit of $(M,\w)$ under $\mrm{SL}(2,\R)$ and its various subgroups determines many of the ergodic properties of the vertical flow.
    Most relevant to our application is Masur's criterion asserting that the $a_t$-orbit of $(M,\w)$ diverges if the vertical flow defined by $\w$ is not uniquely ergodic.

     In the sequel, we fix one such stratum and denote it by $\mc{H}$. For simplicity, we use $\w$ to denote elements of $\mc{H}$.

    \subsection{The Contraction Hypothesis on strata of abelian differentials}
    
    Suppose $\nu$ is a Borel probability measure on $\mrm{SL}(2,\R)$. We say $\nu$ is $(a,\a)$-\textbf{linearly expanding} on $\R^2$ for some constants $a,\a>0$if for all $v\in \R^2\backslash\set{0}$ the following holds
        \begin{equation*}
            \int \norm{gv}^{-\a} \;d\nu(g) \leqslant a \norm{v}^{-\a}. 
        \end{equation*}
    
    For any $g\in \mrm{SL}(2,\R)$, denote by $\norm{g}$ the operator norm of $g$ in its action on $\R^2$.
    Given a compact set $Q\subset \mrm{SL}(2,\R)$, define $\norm{Q}$ as follows:
    \begin{equation}\label{eq: set norm}
        \norm{Q} = \sup_{g\in Q} \max \left(\norm{g}, \norm{g^{-1}}\right).
    \end{equation}

    \begin{theorem}[Lemma 7.5,~\cite{EskinMasur}]
    \label{thm: eskin masur}
    There exist $p_0 = p_0(\mc{H}) \in \N$ such that the following holds.
    Suppose $\a>0$ is given. Then, there exist functions $f_{i,\a}:\mc{H}\r [1,\infty) $ for each $1\leq i\leq p_0$ such that $f_{1,\a}$ is a proper function and each $f_{i,\a}$ satisfies the log-Lipschitz property in Def.~\ref{defn: height functions}(2).
    Moreover, suppose $\nu$ is a compactly supported Borel probability measure on $\mrm{SL}(2,\R)$ which is $(a,\a)$-linearly expanding on $\R^2$.
    Then, there exist constants $\w $ and $b$, depending only on $a$, $\a$ and $\norm{\mrm{supp}(\nu)}$ (cf.~\eqref{eq: set norm}), such that for every $1\leq i \leq p_0$ and $x_0\in \mc{H}$:
    \begin{equation*}
        \int f_{i,\a}(g x_0) \;d\nu(g) \leqslant a f_{i,\a}(x_0) + b + w \sum_{j>i} f_{j,\a}(x_0).
    \end{equation*}
    \end{theorem}
    
    \begin{proof}
        
        Theorem~\ref{thm: eskin masur} was obtained in~\cite[Lemma 7.5]{EskinMasur} for the special measures $d\nu = \d_{a_t r_\th} \;d\th$ for any $t>0$, where $d\th$ is the normalized Lebesgue measure on $[0,2\pi)$. 
        The main part of the proof is the construction of the functions $f_{i,\a}$ in~\cite[p.464]{EskinMasur} (denoted $\a_i$ in \textit{loc. cit.}) using the notion of admissible complexes. The definition of these functions depends on a parameter $\d$ which we take to be $\a-1$ in our notation.
        Inspection of the (short) proof of~\cite[Lemma 7.5]{EskinMasur} shows that the only input used to establish the desired contraction property is the $(a,\a)$-linear expansion and the compactness of the support of the measure. The other parts of the argument are independent of the shape of the measure.
    \end{proof}
    
    \begin{remark}
    The function $f_{1,\a}$ in Theorem~\ref{thm: eskin masur} is given by a power of the reciprocal of the length of the shortest saddle connection.
    \end{remark}
    
    \begin{corollary}[Lemma 2.10,~\cite{Athreya2006}]
    \label{cor: teich f}
        Suppose $p_0$, $\nu$, $b$ and $\w$ are as in Theorem~\ref{thm: eskin masur}. Then, there exist constants $\e_i>0$, depending only on $a$ and $\w$, such that the following holds. Let $\e=(\e_i)\in \R^{p_0}$ and
        let $f_\e = \sum_i \e_i f_{i,\a}$. Then, for all $x_0\in\mc{H}$:
        \begin{equation*}
            \int f_\e(gx_0) \;d\nu(g) \leqslant 2a f_\e(x_0) + bp_0.
        \end{equation*}
    \end{corollary}
    
    \begin{proof}
        
        Let $\e_i = (1+ a/\w)^{i-1}$. Then, one verifies that $\sum_{j< i}\e_j \leqslant \e_i a/\w$ for each $i$. It follows that
        \begin{align*}
         \int f_\e(gx_0) \;d\nu(g) &\leqslant af_\e(x_0) + p_0 b + \w \sum_{i=1}^{p_0} \e_i 
         \sum_{j>i} f_{j,\a}(x_0) \\
         &= af_\e(x_0) + p_0 b + \w \sum_{i=1}^{p_0} f_{i,\a}(x_0) \sum_{j< i}\e_j
         \leqslant 2 af_\e(x_0) + p_0b.
        \end{align*}
        
    \end{proof}

    \subsection{Proof of Theorem~\ref{teich thm} and Corollary~\ref{cor: nue}}
    We wish to apply Theorem~\ref{thrm: Hdim and non-divergence}.
    Let $\Fcal=\set{h_i:i\in \L}$ be an IFS as in the statement of the theorem and let $\Kcal$ be its limit set.
    Denote by $s =\dim_H(\Kcal)$ and $\mu$ the normalized restriction of $H^s$ to $\Kcal$.
    Then, $\a_1(\mu) = s$ by Proposition~\ref{propn: properties of self similar measure}.
    
    Fix $k\in \N$ and $0<\d <1$.
    Let $\k = \d s/2$ and $p = (1+\d)/(1-\d)$.
    Suppose $\g$ is any number satisfying $\k - (1/p) \leqslant \g \leqslant \k$ and let $\t(\bx) = \rho(\bx,k)$.
    Define a Borel measure $\nu$ on $\mrm{SL}(2,\R)$ by
    \[ \int \vp \;d\nu = \int \vp(g_{\t(\bx)} u(\bx)) \t(\bx)^{-\g}\;d\mu(\bx),\]
    for every compactly supported continuous function $\vp$ on $\mrm{SL}(2,\R)$.

    Applying Proposition~\ref{propn: expansion in linear representations} with $G=\mrm{SL}(2,\R)$, $V_1=\R^2$, and $\l=\a_1(\mu)=s$ then shows that $\nu$ is $(a,\d s)$-linearly expanding on $\R^2$, with
    \[ a = C      \left(\int \t(\bx)^{p(\k-\g)}\;d\mu\right)^{1/p},\]
    for a constant $C\geqslant 1$ depending only on $\d$ and $\mu$.
    Hence, Theorem~\ref{thm: eskin masur} and Corollary~\ref{cor: teich f} apply and provide, for every $k\in \N$, constants $\e(k)=(\e_i) \in \R^{p_0}_{>0}$ and $\bar{b}$ (depending on $k$) and a function $f_{\e(k)}$ such that for every $\w\in \mc{H}$, 
    \begin{equation*}
            \int \t(\bx)^{-\g} f_{\e(k)}(g_{\t(\bx)} u(\bx)\w) \;d\mu(\bx) \leqslant 2a f_{\e(k)}(\w) + \bar{b} \leqslant 2C f_{\e(k)}(\w) \left(\int \t(\bx)\;d\mu\right)^{\k-\g} + \bar{b},
    \end{equation*}
   by Jensen's inequality. For $k\in\N$, define
   \begin{equation*}
       f_k := f_{\e(k)}, \qquad T = \frac{\bar{b}}{4C  \left(\int \t(\bx)\;d\mu\right)^{\k-\g}}.
   \end{equation*}
   
   Fix $k\in\N$ and suppose that $f_k(\w)>T$. Then, the above estimate becomes
    \begin{equation*}
            \int \t(\bx)^{-\g} f_k(g_{\t(\bx)} u(\bx)\w) \;d\mu(\bx)  \leqslant 2C f_k(\w) \left(\int \t(\bx)\;d\mu\right)^{\k-\g} + \bar{b}
            \leqslant 4C f_k(\w) \left(\int \t(\bx)\;d\mu\right)^{\k-\g}.
    \end{equation*}
    
    Let $c=4C$, $ \b = \k, \g_0 = \k-(1/p)$, and $\Fsc=\set{f_k:k\in\N}$.
    The above argument shows that $\mu$ satisfies the $(\Fcal,\Fsc,\b,\g_0)$-contraction hypothesis.
   In particular, this completes the verification of the hypotheses of Theorem~\ref{thrm: Hdim and non-divergence} in this setting and shows that the dimension of divergent on average directions belonging to $\Kcal$ is at most $\d s/2$. Since $\d\in (0,1)$ was arbitrary, Theorem~\ref{teich thm} follows.
   
   To prove Corollary~\ref{cor: nue}, we first observe that the restriction of the map $\arctan$ to (the compact set) $\Kcal$ is bi-Lipschitz onto its image.
   Moreover, we note that for $\th\in(-\pi/2,\pi/2)$, 
   \[r_\th = \prescript{t}{}{u}(-\tan\th) a_{\log \cos\th}u(\tan\th) ,\]
   where for $g\in\mrm{SL}(2,\R)$, $\prescript{t}{}{g}$ denotes its transpose.
   Since $a_t$ contracts $\prescript{t}{}{u}(-\tan\th)$ and commutes with $a_{\log \cos\th}$, it follows that the orbit $(a_t r_\th\w)_{t\geqslant0}$ diverges on average in $\mc{H}$ if and only if the orbit $(a_t u(\tan\th)\w)_{t\geqslant0}$ does. 
   Since bi-Lipschitz maps preserve Hausdorff dimension, the corollary follows from Theorem~\ref{teich thm}.

    
    \appendix
    \section{Frostman Exponents of Projections of Self-similar Measures}
\label{sec: dim theory}

    The goal of this section is to complete the proofs of Corollary~\ref{cantor prod cor} by computing $\a_1(\mu)$ in this case in Lemma~\ref{lem: alpha_1 for Cantor product}. We also provide a proof of Theorem~\ref{thm: alpha equal inf} relating $\a_1(\mu)$ to the Frostman exponents of projections of $\mu$ for planar homogeneous fractal measures, thus completing the proof of Corollary~\ref{homogeneous cor}.
    Finally, we show that the limit in the definition of the exponents $\a_\ell(\mu)$ exists for every $\ell$ (Proposition~\ref{prop: existence proj spec}) in full generality.

 \subsection*{Notation} If $(X,\mu)$ is a measure space and $f: X \r Y$ is a measurable map, we denote by $f \mu$ the push-forward measure.

 
 \subsection{Projections of products of Cantor sets}
 
 Consider the IFS $\Fcal$ on $\R^2$ given by maps of the form
 \begin{equation*}
     h_v(\bx) = \frac{\bx+v}{3}, \qquad v\in \L:=\set{0,2}^2.
 \end{equation*}
 The limit set $\Kcal$ of $\Fcal$ coincides with a product of $2$ copies of Cantor's middle thirds set.
  For convenience, let $\L^0$ denote the set consisting of a single point and, for $\w\in \L^0$, denote by $h_\w$ the identity mapping.
  For $\w\in \L^n$, $h_\w$ takes the form $\bx\mapsto 3^{-n}\cdot \bx + b_\w $, where $b_\w \in \Ccal\times \Ccal$ is a rational vector satisfying $3^n b_\w \in \Z^2$.
 
 The following lemma computes the value of $\a_1(\mu)$ in this case and completes the proof of Corollary~\ref{cantor prod cor}.
 We thank Pablo Shmerkin for providing its proof.
 \begin{lemma}\label{lem: alpha_1 for Cantor product}
 Suppose $\Kcal =\Ccal \times \Ccal$, where $\Ccal$ is Cantor's middle-thirds set and let $s=\dim_H(\Kcal)$. Let $\mu$ denote $H^s\vert_\Kcal$. Then, $\a_1(\mu) = \log 2/\log 3$.
 \end{lemma}
 
 \begin{proof}
 That $\a_1(\mu) \leqslant  \log 2/\log 3$ follows from projecting $\mu$ onto the coordinate axes.
 For the reverse inequality, it suffices to show that for each $n$ and each affine line $\Lcal$
 \begin{equation}\label{eq: lines and cantor sets}
     \mu\left(\Lcal^{(3^{-n})}\right) \ll 2^{-n}, 
 \end{equation}
 where the implied constant is absolute.
 Let $I=[0,1]$.
 We claim that for every line $\Lcal$, the number of squares of the form $h_\w(I^2)$ with $\w\in \L^n$ which meet $\Lcal$ is $\ll 2^n$, with an absolute implied constant.
 Assuming the claim, fix a line $\Lcal$ and $n\in\N$, and let $\Lcal_1,\Lcal_2$ denote the $2$ lines parallel to $\Lcal$ and bounding the neighborhood $\Lcal^{(3^{-n})}$.
 Then, we note that $\Lcal^{(3^{-n})}\cap \Kcal$ is contained in the union of the squares $h_\w(I^2)$ which meet either of $\Lcal,\Lcal_1$, or $\Lcal_2$, for some $\w\in \L^n$.
 Moreover, $\mu$ assigns mass $4^{-n}$ to each such square. It follows that $\mu\left(\Lcal^{(3^{-n})}\right)\ll 3 \times 2^{-n}$, and hence~\eqref{eq: lines and cantor sets} follows.
 
 For a line $\Lcal$, denote by $b_\Lcal$ the $x$-coordinate of the intersection of $\Lcal$ with the $x$-axis.
 For $i=0,1$, denote by $\Lcal_{i}$ the line of slope $1$ such that $b_{\Lcal_i} = (-1)^i/3 $.
 Denote by $\Lcal'_i$ the line of slope $-1$ such that $b_{\Lcal'_i} = (-1)^i/3 +1$.
    We say a line $\Lcal$ is exceptional if 
    \begin{equation*}
        \Lcal \in \set{h_\w(\Lcal_i): \w\in \bigcup_{n\geqslant 0} \L^n, i=1,2  } \bigcup 
        \set{h_\w(\Lcal'_i): \w\in \bigcup_{n\geqslant 0} \L^n, i=1,2  }.
    \end{equation*}
    
    We note that each of the exceptional lines $\Lcal_i$ and $\Lcal'_i$ meet $3$ of the squares $h_v(I^2)$ at their corners. 
    Moreover, a line $\Lcal'$ meets three squares of the form $h_\w (h_v(I^2)) \subset h_\w(I^2)$, $\w\in \cup_n \L^n$ if and only if $\Lcal'$ is exceptional.

 Suppose $\Lcal$ is an affine line which is not expceptional.
 Then, we observe that $\Lcal$ meets at most $2$ squares of the form $h_v(I^2)$, where $v\in \L$.
 It follows by induction that $\Lcal$ meets at most $2^n$ squares of the form $h_\w(I^2)$, for $n\in\N$ and $\w\in\L^n$.
 Indeed, if $\w'\in \L^{n-1}$ is the prefix of $\w$, then $h_{\w'}^{-1}(\Lcal)$ is a non-exceptional affine lineand hence meets at most $2$ squares of the form $h_v(I^2)$, $v\in \L$.

    Finally, assume $\Lcal$ is an exceptional line of slope $=1$, the case of $-1$ slope being identical.
    Let $\w\in \L^k$ be such that $\Lcal=h_\w(\Lcal_i)$ for $i\in\set{0,1}$.
    Writing $b_\w = (b_{\w_1},b_{\w_2})$, the line $h_\w(\Lcal_i)$ has the form:
    \begin{equation}\label{eq: exceptional lines}
        h_\w(\Lcal_i) = \set{\bx=(x_1,x_2) \in\R^2: x_2 = x_1 -b_{\w_1}+b_{\w_2} + (-1)^i\cdot 3^{-(k+1)} }.
    \end{equation}
    Since $b_{\w_1},b_{\w_2}\in 3^{-k}\Z$, it follows that $b_{\Lcal} = a/3^{k+1}$, for $a\in\Z$ coprime to $3$.
    Moreover, suppose $\a\in \L^n$, for $n\neq k$. Since $h_\a(\Lcal_i)$ is parallel to $\Lcal$, $h_\a(\Lcal_i) = \Lcal$ if and only if $b_\Lcal = b_{h_\a(\Lcal_i)}$.
    A calculation similar to that yielding~\eqref{eq: exceptional lines} shows that $b_{h_\a(\Lcal_i)}= m\cdot 3^{-n-1}$, where $m\in \Z$ is coprime to $3$. 
    Thus, it follows that
    \begin{equation}\label{eq: L rigid}
        \Lcal \notin \set{h_\a(\Lcal_i): \a\in \bigcup_{n\geqslant 0, n\neq k} \L^n, i=1,2  }.
    \end{equation}
    
    Let $B_n$ denote the set of squares of the form $h_\a(I^2)$ which meet $\Lcal$ with $\a\in \L^n$ and denote by $q_n$ the cardinality of $B_n$.
    We show that~\eqref{eq: L rigid} implies that $q_n \leqslant (3/2)\cdot 2^n$, which will conclude the proof.
    Indeed, suppose $n\neq k$ and let $h_\a(I^2)\in B_n$. Then, since $h_\a^{-1}(\Lcal)\neq \Lcal_0,\Lcal_1$, it follows that $h_\a^{-1}(\Lcal)$ meets at most $2$ squares of the form $h_v(I^2)$.
    Thus, $q_{n+1} \leqslant 2 q_n$. 
    Alternatively, if $n=k$, then $h_\a^{-1}(\Lcal)$ meets at most $3$ such squares and, hence, $q_{k+1} \leqslant 3 q_k$.
    By induction, we conclude that $q_n \leqslant (3/2)\cdot 2^n$ as desired.

 \end{proof}


\subsection{Frostman exponents of projections and Theorem~\ref{thm: alpha equal inf}}

    Let the notation be as in the statement of the theorem.
    A key ingredient in the argument is a a generalization of the sub-additive ergodic theorem for uniquely ergodic systems due to Furman in~\cite{Furman-MultErgThm}. It gives information about the behavior of every orbit, as opposed to the almost everywhere statement of the subadditve ergodic theorem.
    Recall that if $(X,\mu,T)$ is a measure preserving system, a sequence of functions $(\phi_n)$ is a \textbf{sub-additive cocycle} over $T$ if for all $m,n$ and almost every $x\in X$:
    \begin{align*}
        \phi_{m+n}(x) \leqslant \phi_m(x) + \phi_n(T^mx).
    \end{align*}

    \begin{theorem}[Theorem 1,~\cite{Furman-MultErgThm}]
    \label{thm subadd UE}
    Suppose $(X,T,\nu)$ is a uniquely ergodic probability measure preserving system, where $X$ is a compact metric space and $T$ is continuous. 
    Let $\phi_n:X\r\R$ be a continuous sub-additive cocycle over $T$.
    Then, the following holds
    \begin{equation}
        \limsup_{n\r\infty} \frac{\phi_n(x)}{n} \leqslant \inf_{n\in\N} \left\{ \int\frac{\phi_n}{n} \;d\nu  \right\},
    \end{equation}
    uniformly over all $x\in X$.
    \end{theorem}

   Given an affine line $\mc{L}$ and $\th\in [0,2\pi)$, we write $\mc{L}\perp \th$ whenever $\mc{L}$ is orthogonal to any line of slope $\tan \th$.
   Recall that $\mc{A}(d,\ell)$ denotes the collection of all affine subspaces of $\R^d$ of dimension $\ell$.
    For every $\th\in [0,2\pi)$ and $n\in \N$, $\e>0$, we define
    \begin{equation*}
        t(\th,\e) := \sup_{\mc{L} \in\mc{A}(2,1): \mc{L} \perp \th } \mu\left(\mc{L}^{(\e)} \right),\qquad  
        \t(\th,n) := t(\th,\rho^n).
    \end{equation*}

    The first step in applying Theorem~\ref{thm subadd UE} is the following continuity result.
    \begin{proposition}\label{prop: tau cont}
        For every $\e>0$, the function $\log t(\cdot,\e)$ is continuous on $[0,2\pi)$.
    \end{proposition}
    
    The key ingredient is the compactness of the support of $\mu$, which is used in the proof of the following lemma.
    \begin{lemma}\label{lem: projective continuity}
        Suppose $\Lcal_1$ is an affine line and suppose $0\leq \d \leq \pi/2$ is given. Let $K=\diam{\Kcal}$. Then, there exists an affine line $\Lcal_2$ which meets $\Lcal_1$ at an angle $\d$ and such that for every $\e>0$,
        \begin{equation*}
            \Lcal_1^{(\e)} \cap \Kcal \subset \Lcal_2^{(\e+A\d)},
        \end{equation*}
        where $A = 4(K+\e)$.
    \end{lemma}
    
    \begin{proof}
    If $\d=0$, we can take $\Lcal_2 = \Lcal_1$. Hence, we will assume $\d>0$.
    Let $B$ be a ball of radius $K+\e$ containing $\Kcal$. If $\Lcal_1\cap B = \emptyset$, then $\Lcal_1^{(\e)} \cap\Kcal=\emptyset$ and the statement follows trivially.
    Otherwise, let $x_0\in \Lcal_1\cap B$. Let $\Lcal_2$ be the line passing through $x_0$ at an angle $\d$ with $\Lcal_1$.
    
    For $i=1,2$, let $n_i$ be the unit normal vectors to $\Lcal_i$ so that the angle between $n_1$ and $n_2$ is $\d$. Then, for $x\in\R^2$, $d(x,\Lcal_i) =\left|\langle x-x_0,n_i\rangle\right| $. Let $x\in \Lcal_1^{(\e)} \cap \Kcal$. We then have
    \begin{equation*}
        d(x,\Lcal_2) = \left|\langle x-x_0,n_2\rangle\right|
        \leqslant \left|\langle x-x_0,n_1\rangle\right|+
        \left|\langle x-x_0,n_2-n_1\rangle\right|
        \leqslant \e + \norm{x-x_0}\norm{n_2-n_1}.
    \end{equation*}
    Since $x,x_0\in B$, we have $\norm{x-x_0} \leqslant 2(K+\e)$. Moreover, using the law of sines, one verifies that
    \begin{equation*}
        \norm{n_2-n_1} = \frac{\sin \d}{\sin (\frac{\pi-\d}{2})} \leqslant \frac{ \pi\d}{\pi-\d } \leqslant 2\d,
    \end{equation*}
    where we used the fact that $0<\d\leq \pi/2$.
    
    \end{proof}

    \begin{proof}[Proof of Proposition~\ref{prop: tau cont}]
    We may assume without loss of generality that the similarity dimension of $\Fcal$ is positive. Indeed, otherwise, $\Fcal$ consists of a single map, $\Kcal$ is a single point, and all the quantities in question are $0$.
    We then observe that the irrationality of the rotation angle $\a$ implies that $\Fcal$ is irreducible and in particular, that $\a_1(\mu) >0$ by Corollary~\ref{cor: alpha >0}.
    Fix some $0<\b<\a_1(\mu)$. It follows from the definition of $\a_1(\mu)$ that there exists $r_0>0$ so that for all $0<r<r_0$ and all affine lines $\Lcal$,
    \begin{equation} \label{eq: decay near lines}
        \mu \left( \Lcal^{(r)} \right) \leqslant r^\b.
    \end{equation}
    Moreover, by Proposition~\ref{propn: properties of self similar measure}, there exists a constant $C\geq 1$ such that
    \begin{equation} \label{decay near points}
        \frac{1}{C} r^s  \leqslant \mu(B(x,r)) \leqslant C r^s,
    \end{equation}
    for every $x\in \Kcal$ and every $r>0$, where $s=\dim_H(\Kcal)$.
    Fix $\e>0$ and assume $0<\eta<1$ is given.
    Denote by $K=\diam{\Kcal}$ and let $A = 4(K+\e)$.
    Let $0<\d\leq \pi/2$ be sufficiently small so that 
    \[\d^{\b/2} \leqslant  \min\left\{ \frac{\e^s}{2(A+1)^\b C},\eta\right\}  . \]
    Suppose $\vp, \th \in [0,2\pi)$ are two angles at distance at most $\d$ in $\mathbb{S}^1$.
    Let $\Lcal_1$ be a line satisfying $\Lcal_1 \perp \th$.
    By Lemma~\ref{lem: projective continuity}, there exists a line $\Lcal_2\perp \vp$ such that 
    \begin{equation*}
         \mc{L}_1^{(\e)} \cap \Kcal \subseteq 
        \mc{L}_2^{(\e+A\d)}.
    \end{equation*}
    Moreover, we can find lines $\Lcal_3,\Lcal_4$ parallel to $\Lcal_2$ and satisfying $\cup_{j=3,4} \Lcal_j\subset \mc{L}_2^{(\e+A\d)}\backslash \mc{L}_2^{(\e)}$ and $\mc{L}_2^{(\e+A\d)}\backslash \mc{L}_2^{(\e)}\subseteq \cup_{j=3,4} \Lcal_j^{((A+1)\d)} $.
    This, along with~\eqref{eq: decay near lines}, imply
    \begin{equation*}
        \mu\left(\mc{L}_1^{(\e)} \right)\leqslant \mu\left(\mc{L}_2^{(\e)} \right) + \mu\left(\mc{L}_3^{((A+1)\d)} \right) 
        + \mu\left(\mc{L}_4^{((A+1)\d)} \right)
        \leqslant t(\vp,\e) + 2 ((A+1)\d)^\b.
    \end{equation*}
    The next ingredient is to observe that $t(\vp,\e) \geqslant \e^s/C$. This follows by taking $\Lcal_0 \perp \vp$ to be a line passing through a point in $\Kcal$. Then, $\Lcal_0^{(\e)}$ contains a ball of radius $\e$ centered in $\Kcal$. The claim thus follows from the estimate in~\eqref{decay near points}.
    Our choice of $\d$ hence implies that
    \begin{equation*}
        \mu\left(\mc{L}_1^{(\e)} \right)\leqslant (1+\d^{\b/2})t(\vp,\e).
    \end{equation*}
    Since $\Lcal_1$ was arbitrary, we see that $\log t(\th,\e)\leqslant \eta+ \log t(\vp,\e)$.
    Since $\d$ is independent of $\th$ and $\vp$, one can run the above argument with the roles of $\vp$ and $\th$ reversed, to get the reverse inequality and conclude the proof.
    \end{proof}
    
    The next ingredient in the proof of Theorem~\ref{thm: alpha equal inf} is establishing the following cocycle property of the functions $\t$.
    \begin{proposition} \label{prop: tau submul cocycle}
    There exists a constant $D\geqslant 1$, such that for every $\th\in [0,2\pi)$ and $m,n\in\N$, 
    \begin{equation*}
        \t(\th,m+n) \leqslant D  \t(\th,m) \t(R^m_{-\a}( \th), n).
    \end{equation*}
    \end{proposition}

    We will need the following doubling property of the functions $t(\cdot)$.
    \begin{lemma}\label{lem: doubling of projection}
    For every $A\geq 1$, there exists $D\geq 1$ such that for every $\th$ and $\e>0$, \[t(\th,A\e)\leqslant D t(\th,\e).\]
    \end{lemma}
    \begin{proof}
    Let $D$ denote the cardinality of a finite cover of a ball of radius $1$ in $\R$ by balls of radius $1/A$.
        By applying scaling and translation, it follows that for every $\e >0$, any ball of radius $A\e$ in $\R$ can be covered by at most $D$ balls of radius $\e$.
        Hence, for every Borel measure $\nu$ on $\R$, one has
        \begin{equation*}
            \nu(B(x,A\e)) \leqslant D \sup_{y\in \R }
            \nu(B(y,\e)),
        \end{equation*}
        for all $x\in \R$ and all $\e>0$.
    \end{proof}
    
    \begin{proof}[Proof of Proposition~\ref{prop: tau submul cocycle}]
    
    Let $\Lcal \perp \th$ and $m,n\in\N$ be given and let $P = \L^m$.
    For every $\xi\in P$ and $\e>0$, observe that $h_\xi^{-1}(\Lcal^{(\e)}) = (h_\xi^{-1}(\Lcal))^{(\e/\rho^m)}$.
    Moreover, one has $h_\xi^{-1}(\Lcal)\perp R_{-\a}^m(\th)$
    Then, by equation~\eqref{eq: iterated P_l} and Lemma~\ref{lem: transformation of self-similar measures}, it follows that
    \begin{align*}
        \mu \left( \Lcal^{(\rho^{m+n})}  \right) &= \sum_{\xi\in P} \mu \left( \Lcal^{(\rho^{m+n})} \cap \Kcal_\xi \right)
        = \sum_{\xi\in P} \rho^{ms}
        \mu \left( h_\xi^{-1}\left(\Lcal^{(\rho^{m+n})}\right) \right)\\
        &= \sum_{\xi\in P} \rho^{ms}
        \mu \left( h_\xi^{-1}\left(\Lcal\right)^{(\rho^n)} \right)
        \leqslant \t(\th-m\a,n)
        \sum_{\substack{\xi\in P\\ \Kcal_\xi\cap \Lcal^{(\rho^{m+n})}\neq \emptyset}} \rho^{ms}.
    \end{align*}
    The other ingredient is the observation that for every $\xi\in P$ satisfying $\Kcal_\xi\cap \Lcal^{(\rho^{m+n})}\neq \emptyset$, we have that $\Kcal_\xi\subset \Lcal^{((1+K)\rho^{m})}$, where $K=\diam{\Kcal}$.
    Indeed, this follows from the fact that the diameter of $\Kcal_\xi$ is $K\rho^m$.
    Moreover, we have $\mu(\Kcal_\xi) = \rho^{ms}$ for every $\xi$.
    Hence, Proposition~\ref{prop: null overlap} on the null overlaps between the distinct $\Kcal_\xi$ implies that
    \begin{equation*}
        \sum_{\substack{\xi\in P\\ \Kcal_\xi\cap \Lcal^{(\rho^{m+n})}\neq \emptyset}} \rho^{ms}
        \leqslant \mu\left(\bigcup_{\substack{\xi\in P\\ \Kcal_\xi\cap \Lcal^{(\rho^{m+n})}}} \Kcal_\xi  \right)
        \leqslant \mu\left( \Lcal^{((1+K)\rho^{m})}\right)
        \leqslant t(\th,(1+K)\rho^{m}).
    \end{equation*}
    Finally, we apply the doubling property from Lemma~\ref{lem: doubling of projection} with $A=K+1$ to get that $t(\th,(1+K)\rho^m)\leqslant D \t(\th,m)$, for a constant $D\geqslant 1$ depending only on $K$.
    \end{proof}
    
    We are now ready for the proof of Theorem~\ref{thm: alpha equal inf}.
    \begin{proof}[Proof of Theorem~\ref{thm: alpha equal inf}]
    Let $\phi_n(\th):=\log\t(\th,n) + \log D$, where $D$ is the constant in the conclusion of Proposition~\ref{prop: tau submul cocycle}.
    One can then verify that for every $\th\in[0,2\pi)$:
    \begin{equation*}
        \dim_\infty(\pi_\th\mu) = \liminf_{n\r\infty} \frac{\log \t(\th,n)}{n\log \rho} = \liminf_{n\r\infty} \frac{\log\phi_n(\th)}{n\log\rho}.
    \end{equation*}
    Let $\nu$ be the Lebesgue probability measure on $\mathbb{S}^1$.
    Then, $\phi_n$ is a sub-additive cocycle over the transformation $T= R_{-\a}$ of $X= \mathbb{S}^1$.
    In particular, Kingman's sub-additive ergodic theorem implies that
    \begin{equation*}
        \dim_\infty(\pi_\th\mu)= \phi_\ast := \sup_{n\geqslant 1}  \int \frac{\phi_n}{n\log\rho} \;d\nu =\lim_{n\r\infty}
          \int \frac{\phi_n}{n\log\rho} \;d\nu, \qquad
          \text{for }\nu\text{-almost every } \th.
    \end{equation*}
    
    Moreover, the cocycle $(\phi_n)$ is continuous by Proposition~\ref{lem: projective continuity}.
    Hence, Theorem~\ref{thm subadd UE} implies that
    \begin{equation}\label{eq: consequence of furman}
        \dim_\infty(\pi_\th\mu) \geqslant \phi_\ast, \qquad 
        \text{for every }\th.
    \end{equation}
    It remains to show that $\a_1(\mu) = \phi_\ast$. That $\a_1(\mu)\leqslant \phi_\ast$ follows by definition.
    For the reverse inequality, we use the uniformity in covergence provided by Theorem~\ref{thm subadd UE}.
    Fix $0<\e<1$. For every $n\in \N$, let $\th_n$ be such that
    \begin{equation} \label{eq: choose th_n}
        \log \sup_\th \t(\th,n) =  \sup_{\th} \log\t(\th,n)\leqslant (1-\e)\log \t(\th_n,n).
    \end{equation}
    Here we used the fact that $\log\t(\cdot) \leqslant 0$ since $\mu$ is a probability measure.
    From uniform convergence of the $\liminf$ in $\th$ in~\eqref{eq: consequence of furman}, we can find $n_0\in\N$ so that for all $n\geq n_0$ and for all $\th$,
    \begin{equation*}
        \frac{\log\t(\th,n)}{n\log\rho} \geqslant \phi_\ast - \e.
    \end{equation*}
    Combining this with~\eqref{eq: choose th_n} and the fact that $\log \rho <0$, we get
    \begin{equation*}
        \frac{\log \sup_{\th\in\mathbb{S}^1} \t(\th,n) }{n\log\rho} \geqslant (1-\e)(\phi_\ast-\e),
    \end{equation*}
    for all $n\geqslant n_0$. Since $\e$ was arbitrary, it follows that $\a_1(\mu)\geq \phi_\ast$ as desired.
    \end{proof}

    
\subsection{Existence of the limit in Definition~\ref{eqn: proj spec}}    
    Throughout this section, we fix a finite set $\L$ and an IFS $\Fcal=\set{h_i:i\in \L}$ on $\R^d$ satisfying the open set condition. We denote by $\Kcal$ the limit set of $\Fcal$, $s$ the similarity dimension of $\Fcal$, and $\mu$ the restriction of $H^s$ to $\Kcal$.
    
    \begin{proposition} 
    \label{prop: existence proj spec}
    Suppose $\mu$ is as above. Then, for every $1\leq \ell \leq d$, the limit in the definition of $\a_\ell(\mu)$ exists.
    \end{proposition}

   For every $1\leq \ell\leq d$, $\w\in \L^\N$ and $n\in \N$, $\e>0$, we define
   \begin{equation*}
       t_\ell(\e) = \sup_{\mc{L}\in \mc{A}(d,\ell)} 
       \mu\left( \mc{L}^{(\e)} \right)
       , \qquad
       \t_\ell(\w,n)= t_\ell(\rho(\w,n)),
   \end{equation*}
    where $\rho(\w,n)$ is defined by~\eqref{defn: rho}.
    The functions $t_\ell$ satisfy the following doubling property.
    \begin{lemma}\label{lem: doubling t_ell}
    For every $A\geqslant 1$, there exists a constant $D \geqslant 1$, depending only on $A$ and $\ell$, such that for every $\e>0$,
    \[t_\ell(A\e) \leqslant D t_\ell(\e).\]
    \end{lemma}
    
    \begin{proof}
        The proof is completely analogous to Lemma~\ref{lem: doubling of projection}.
        
    \end{proof}

    The key step in proving Proposition~\ref{prop: existence proj spec} is to show that $\t_\ell$ is submultiplicative. This was essentially shown in~\cite{KleinbockLindenstraussWeiss}. We include a proof for completeness.
    \begin{lemma} \label{lemma: tau submultiplicative}
    There exists a constant $C_\ell \geqslant 1$ so that 
    \[ \t_\ell( \w,m+n) \leqslant C_\ell \t_\ell(\w,m) \t_\ell(\s^m \w,n) \]
    for all $\w \in \L^\N$ and all $m,n\in \N$.
    \end{lemma}
    
    Before proving this lemma, we record the following useful corollary.
    \begin{corollary}[Lemma 8.2,~\cite{KleinbockLindenstraussWeiss}]
    \label{cor: alpha >0}
    Suppose $\Fcal$ is an irreducible IFS on $\R^d$ satisfying the open set condition. Let $\Kcal$ be its limit set and let $s=\dim_H(\Kcal)$. Then, for each $1\leq \ell \leq d$, $\a_\ell(\mu)>0$, where $\mu$ is the restriction of $H^s$ to $\Kcal$.
    \end{corollary}
    
    \begin{proof}
     Since Lemma 8.2 of~\cite{KleinbockLindenstraussWeiss} is stated in a different form, we provide here a proof for completeness.
     It is shown over the course of the proof of~\cite[Theorem 2.3]{KleinbockLindenstraussWeiss} that the irreducibility of $\Fcal$ implies that $\mu(\Lcal)=0$ for all proper affine subspaces $\Lcal\subset \R^d$.
     We claim that this implies that
     \begin{equation*}
         \lim_{\e\r 0} t_\ell(\e) =0.
     \end{equation*}
     Indeed, suppose not. Then, there exist $\d>0$ and  sequences $\e_i \r 0$ and $\Lcal_i\in\mc{A}(d,\ell)$ such that $\Lcal_i\cap \Kcal \neq \emptyset$ and
     \begin{equation*}
         \mu\left(\Lcal^{(\e_i)}\right) >\d.
     \end{equation*}
     Since $\Kcal$ is compact, the subset of $\mc{A}(d,\ell)$ consisting of affine subspaces that meet $\Kcal$ is compact. In particular, by passing to a subsequence if necessary, we may assume that $\Lcal_i \r \Lcal$ for some $\Lcal\in \mc{A}(d,\ell)$ such that $\Lcal\cap\Kcal\neq \emptyset$. Hence, it follows by the dominated convergence theorem that $\mu(\Lcal) \geqslant \d >0$, which is a contradiction.
     
     Let $C_\ell \geqslant 1$ be the constant in Lemma~\ref{lemma: tau submultiplicative}. Let $\e_0>0$ be sufficiently small so that
     \begin{equation*}
         t_\ell(\e) < 1/2C_\ell,
     \end{equation*}
      for all $0<\e\leq\e_0$. Let $\rho_{\max} $ and $\rho_{\min}$ denote the largest and smallest contraction ratios of the maps in $\Fcal$ respectively. Then, $0<\rho_{\min}\leq \rho_{\max} <1$. 
     In particular, we can find $k\in \N$ sufficiently large so that $\rho_{\max}^k < \e_0$. Fix one such $k$.
     Then, for all $\a\in \L^\N$, we have
     \begin{equation}\label{eq: radius below cutoff}
         \rho(\a,k) \leqslant \rho_{\max}^k < \e_0.
     \end{equation}
     
     Fix some $\w\in \L^\N$. It follows from Lemma~\ref{lemma: tau submultiplicative} and~\eqref{eq: radius below cutoff} that for all $n\in\N$,
     \begin{equation*}
         \t_\ell(\w, nk) \leqslant C_\ell^n \prod_{i=0}^{n-1} \t_\ell(\s^i\w,k) \leqslant C_\ell^n t_\ell(\e_0)^n < 2^{-n}.
     \end{equation*}
     Moreover, note that $t_\ell(\e) \leq 1$ for all $\e>0$ since $\mu$ is a probability measure. 
     In particular, $\log t_\ell(\e)/\log \e \geqslant 0$ for all $\e>0$.
     Combined with that fact that $\rho(\w,nk) \geqslant \rho_{\min}^{nk}$, this implies that
     \begin{align*}
         \frac{\log \t_\ell(\w,nk)}{\log \rho(\w,nk)} \geqslant \frac{\log \t_\ell(\w,nk)}{nk\log \rho_{\min}} \geqslant 
         \frac{-\log 2}{k \log \rho_{\min}} >0,
     \end{align*}
     for all $n\in\N$. This proves that $\liminf_{n\r\infty}  \frac{\log \t_\ell(\w,nk)}{\log \rho(\w,nk)} >0$. 
     To conclude the proof, suppose that $0<\e <\rho(\w,k)$ is given and let $n(\e)\in \N$ be such that
     \begin{equation*}
         \rho(\w,(n(\e)+1)k) \leqslant \e < \rho(\w,n(\e)k).
     \end{equation*}
     In particular, we have
     $
         t_\ell(\e) \leqslant \t_\ell(\w,(n(\e))k)
     $.
    Moreover, the choice of $n(\e)$ implies
     \begin{equation*}
         \e \geqslant \rho(\w,n(\e)k) \rho_{\min}^{k}.
     \end{equation*}
     Thus, we obtain the following estimate:
     \begin{equation*}
         \frac{\log t_\ell(\e)}{\log\e} \geqslant \frac{\log \t_\ell(\w,n(\e)k)}{\log \e} \geqslant 
         \frac{\log \t_\ell(\w,n(\w)k)}{\log \rho(\w,n(\e)k) + k\log \rho_{\min} }.
     \end{equation*}
     Note that $n(\e)\r\infty $ as $\e\r 0$. In particular, $\log \rho(\w,n(\e)k)\r - \infty$. This shows that
     \begin{equation*}
         \a_\ell(\mu) = \liminf_{\e\r 0} \frac{\log t_\ell(\e)}{\log \e } \geqslant \liminf_{n\r\infty} \frac{\log \t_\ell(\w,nk)}{\log \rho(\w,nk)} >0,
     \end{equation*}
     as desired.
    \end{proof}
    
    Following~\cite{KleinbockLindenstraussWeiss}, we say a finite set $P\subset \cup_{k\geq 1} \L^k$ is a \textbf{complete prefix set} if every $\w\in \L^\N$, there is a unique word $\a\in P$ which occurs as a prefix for $\w$. It is easy to check that the set
    \begin{equation}\label{eq: complete prefix example}
        P(\e) :=\set{\a \in \cup_{k\geq 1} \L^k: \e \rho_{\min}  \leq \rho_\a \leq \e}
    \end{equation}
    forms a complete prefix set, where $\rho_{\min} = \min\set{\rho_i:i\in \L}$, and $\rho_\a$ is as in~\eqref{eq: composition parameters}. The following lemma allows us to handle the case where the contraction ratios are not all the same.
    
    \begin{lemma}\label{lem: comp prefix}
    Suppose $P$ is a complete prefix set. Then, for every continuous function $h$ on $\R^d$,
    \begin{equation*}
        \int h d\mu = \sum_{\a\in P} \int_{\Kcal_\a} h d\mu.
    \end{equation*}
    \end{lemma}
    \begin{proof}
    Proposition~\ref{prop: null overlap} implies that $\set{\Kcal_\a: \a\in P}$ forms a measurable partition of the support of $\mu$ with null overlaps.
    \end{proof}
    
    \begin{proof}[Proof of Lemma~\ref{lemma: tau submultiplicative}]
    Suppose $\w\in\L^\N$ and $m,n\in\N$ are given. Let $P=P(\rho(\w,m))$ be the set defined in~\eqref{eq: complete prefix example} with $\e=\rho(\w,m)$.
    Let $\Lcal$ be an affine subspace of dimension $d-\ell$.
    For simplicity, we write $\Lcal^{(\w,k)}:= \Lcal^{(\rho(\w,k))}$ for every $k\in\N$.
    Note that for each $\xi\in P$, $\rho(\w,m+n)/\rho_\xi \leqslant \rho_{\min}^{-1}\rho(\s^m(\w),n)$.
    By Lemma~\ref{lem: comp prefix} and a similar argument to the proof of Proposition~\ref{prop: tau submul cocycle}, one obtains
    \begin{align}\label{eq: remove first factor}
        \mu \left( \Lcal^{(\w,m+n))}  \right) &= \sum_{\xi\in P} \mu \left( \Lcal^{(\w,m+n)} \cap \Kcal_\xi \right)
        = \sum_{\xi\in P} \rho_\xi^{s}
        \mu \left( h_\xi^{-1}\left(\Lcal^{(\w,m+n)}\right) \right) \nonumber
        \\
        &= \sum_{\xi\in P} \rho_\xi^{s}
        \mu \left( h_\xi^{-1}\left(\Lcal\right)^{(\rho(\w,m+n)/\rho_\xi)} \right)
        \leqslant   t_\ell(\rho_{\min}^{-1} \rho(\s^m(\w),n)  )      
        \sum_{\substack{\xi\in P\\ \Kcal_\xi\cap \Lcal^{(\w,m+n)}\neq \emptyset}}  \rho_\xi^{s}.
    \end{align}
    Proposition~\ref{prop: null overlap} implies that $\mu(\Kcal_{\xi_1}\cap\Kcal_{\xi_2})=0$ for every $\xi_1,\xi_2\in P$ with $\xi_1\neq \xi_2$.
    Moreover, by definition of $P$, for every $\xi\in P$ satisfying $\Kcal_\xi\cap \Lcal^{(\w,m+n)}\neq \emptyset$, we have that $\Kcal_\xi\subset \Lcal^{((1+K)\rho(\w,m))}$, where $K=\diam{\Kcal}$.
    Hence, arguing as in the proof of Proposition~\ref{prop: tau submul cocycle}, we obtain
    \begin{equation}\label{remove second factor}
        \sum_{\substack{\xi\in P\\ \Kcal_\xi\cap \Lcal^{(\w,m+n)}\neq \emptyset}}  \rho_\xi^{s}
        \leqslant  t_\ell((1+K)\rho(\w,m)).
    \end{equation}
    In view of the doubling property of $t_\ell$ provided by Lemma~\ref{lem: doubling t_ell}, the conclusion follows by combining~\eqref{eq: remove first factor} and~\eqref{remove second factor}.
    \end{proof}
    
    We now deduce Proposition~\ref{prop: existence proj spec} from Lemma~\ref{lemma: tau submultiplicative}.
    \begin{proof}[Proof of Proposition~\ref{prop: existence proj spec}]
        Let $C_\ell$ be the constant given by Lemma~\ref{lemma: tau submultiplicative}.
        Consider the function $\vp_\ell: \L^\N \times \N \r \R_+ $ defined by
        $ \vp_\ell = \log C_\ell + \log \t_\ell$.
        Lemma~\ref{lemma: tau submultiplicative} implies that $\vp_\ell$ is a subadditive cocycle over $\s:\L^\N \r \L^\N$.
        
        Consider the constant sequence $\w = (i)_k \in \L^\N$, for some for some fixed $i\in\L$. Then, the sequence $a_n = \vp_\ell(\w,n)$ is a subadditive sequence. This follows from the fact that $\s \w = \w$.
        Thus, by Fekete's lemma, we obtain the following equalities
        $\l_\ell = \lim_{n\r\infty} \frac{a_n}{n} = \inf_{n\geqslant 1} \frac{a_n}{n}$.
        
        We claim that $\a_\ell(\mu) = \l_\ell/\log \rho_i$ and in particular that the limit defining $\a_\ell(\mu)$ exists.
        To see this, suppose $\e_k$ is a sequence tending to $0$. For each $k$, let $n_k \in \N$ be such that
        $\rho_i^{n_k +1} < \e_k \leqslant \rho_i^{n_k}$.
        Since $0<\rho_i <1$, it follows that
        \[ \frac{\log \sup_{\mc{L}:\dim \mc{L}=d-q} \mu \left(\mc{L}^{(\rho_i^{n_k})} \right) }{ n_k \log \rho_i } \leqslant
        \frac{\log \sup_{\mc{L}:\dim \mc{L}=d-\ell} \mu \left(\mc{L}^{(\e_k)} \right) }{ \log \e_k } \frac{n_k + 1}{n_k}. \]
        In particular, we see that
        \[ \liminf_{\e\r 0} \frac{\log \sup_{\mc{L}:\dim \mc{L}=d-\ell} \mu \left(\mc{L}^{(\e)} \right) }{ \log \e } \geqslant \l_\ell/\log \rho_i. \]
        The opposite inequality involving the $\limsup$ follows analogously.

    \end{proof}

}   

\bibliography{bibliography}{}
\bibliographystyle{amsalpha}

\end{document}